\documentclass[]{article}

\usepackage{amsmath}
\usepackage{amsthm}
\usepackage{amssymb}
\usepackage{array}
\usepackage{algorithm,algorithmic}
\usepackage[nohead,margin=1.1in]{geometry}
\usepackage{color}
\usepackage{verbatim}
\usepackage{booktabs}
\usepackage{threeparttable}
\usepackage{url}

\newtheorem{theorem}{Theorem}[section]

\newtheorem{proposition}{Proposition}[section]
\newtheorem{lemma}{Lemma}[section]
\newtheorem{corollary}{Corollary}[section]
\newtheorem{assumption}{Assumption}[section]

\newtheorem{remark}{Remark}[section]

\usepackage{graphicx,epstopdf}
\usepackage{subfigure}
\usepackage{subfig}
\graphicspath{{./pics/},{./result_v5/}}
\usepackage[toc,page]{appendix}
\allowdisplaybreaks

\numberwithin{equation}{section}
\newcommand{\E}{\mathbb{E}}

\begin{document}
\numberwithin{equation}{section}
\numberwithin{figure}{section}
\title{On the Saturation Phenomenon of Stochastic Gradient Descent for Linear Inverse Problems}

\author{Bangti Jin\thanks{Department of Computer Science, University College London, Gower Street, London WC1E 6BT, UK
(b.jin@ucl.ac.uk,bangti.jin@gmail.com). The work of BJ is supported by UK EPSRC grant EP/T000864/1.}
\and Zehui Zhou\thanks{Department of Mathematics, The Chinese University of Hong Kong, Shatin, New Territories, Hong Kong.
({zhzhou@math.cuhk.edu.hk, zou@math.cuhk.edu.hk}).
The work of JZ was substantially supported by Hong Kong RGC General Research Fund (projects 14306718 and
14304517).} \and Jun Zou\footnotemark[2].}

\date{}
\maketitle

\begin{abstract}
Stochastic gradient descent (SGD) is a promising method for solving large-scale inverse problems, due to its excellent scalability
with respect to data size. The current mathematical theory in the lens of regularization theory predicts that SGD
with a polynomially decaying stepsize schedule may suffer from an undesirable
saturation phenomenon, i.e., the convergence rate does not further improve
with the solution regularity index when it is beyond a certain range.
In this work, we present a refined convergence rate analysis of SGD, and prove that saturation actually does not occur if the initial
stepsize of the schedule is sufficiently small. Several numerical experiments are provided to complement the analysis.
\end{abstract}

\section{Introduction}

In this paper, we consider the numerical solution of the following finite-dimensional linear inverse problem:
\begin{equation}\label{eqn:lininv}
Ax=y^\dag,
\end{equation}
where $A\in \mathbb{R}^{n\times m}$ is the system matrix representing the data formation mechanism,
and $x\in \mathbb{R}^m$ is the unknown signal of interest. In the context of inverse problems, the matrix
$A$ is commonly ill-conditioned. When the matrix $A$ is rank-deficient,
equation \eqref{eqn:lininv} may have infinitely many solutions. The reference solution $x^\dag$ is
taken to be the minimum norm solution relative to the initial guess $x_1$, i.e.,
\begin{align*}
x^\dag=\mathop{\arg\min}_{x\in\mathbb{R}^m}\{\|x-x_1\|\quad \mbox{s.t.}\quad Ax=y^\dag \},
\end{align*}
with $\|\cdot\|$ being the Euclidean norm of a vector (and also the spectral norm of a matrix).
In practice, we only have access to a noisy version $y^\delta$ of the exact data $y^\dag = Ax^\dag$, i.e.,
\begin{equation*}
y^\delta =y^\dag +\xi,
\end{equation*}
where $\xi \in \mathbb{R}^n$ denotes the noise in the data with a noise level
$ \delta :=\|\xi\|.$
We denote the $i${th} row of the matrix $A$ by a column vector $a_i\in \mathbb{R}^m$, i.e., $A=[a_i^t]_{i=1}^n$
(with the superscript $t$ denoting the matrix/vector transpose), and the $i${th} entry of the vector
$y^\delta\in\mathbb{R}^n$ by $y_i^\delta $.  Linear inverse problems of the form \eqref{eqn:lininv} arise in a
broad range of applications, e.g., initial condition / source identification and optical imaging.
A large number of numerical methods have been developed, prominently variational
regularization \cite{EnglHankeNeubauer:1996,ItoJin:2015} and iterative regularization \cite{KaltenbacherNeubauerScherzer:2008}.

Stochastic gradient descent (SGD) is one very promising numerical method for solving problem \eqref{eqn:lininv}. In its simplest form,
it reads as follows. Given an initial guess $x_1^\delta \equiv x_1 \in \mathbb{R}^m$, we update the iterate $x_{k+1}^\delta $ recursively by
\begin{equation}\label{eqn:SGD}
  x_{k+1}^\delta =x_k^\delta -\eta_k ((a_{i_k} ,x_k^\delta )-y_{i_k}^\delta )a_{i_k},\quad k=1,2,\cdots,
\end{equation}
where the random row index $i_k$ is drawn i.i.d. uniformly from the index set $\{1,\cdots,n\}$, $\eta_k>0$ is the stepsize
at the $k$th iteration, and $(\cdot,\cdot)$ denotes the Euclidean inner product. We denote by $\mathcal{F}_k$
the filtration generated by the random indices $\{i_1,\ldots,i_{k-1}\}$, define $\mathcal{F}$ by
$\mathcal{F}={\bigvee_{k \in \mathbb{N}} \mathcal {F}_k},$
and let $(\Omega,\mathcal{F},\mathbb{P})$ {be the associated probability space.} The notation $\E[\cdot]$ denotes taking expectation
with respect to the filtration $\mathcal{F}$. The SGD iterate $x_k^\delta$ is
random, and measurable with respect to $\mathcal{F}_k$.  SGD is a randomized version of the classical Landweber method \cite{Landweber:1951}
\begin{equation}\label{eqn:Landweber}
  x_{k+1}^\delta = x_k^\delta -\eta_k n^{-1}A^t(Ax_k^\delta -y^\delta ),\quad k=1,2,\cdots,
\end{equation}
which is identical with the gradient descent applied to the following objective functional
\begin{equation}\label{eqn:obj}
  J(x) = (2n)^{-1} \|Ax-y^\delta\|^2.
\end{equation}
When compared with the Landweber method in \eqref{eqn:Landweber}, SGD \eqref{eqn:SGD} employs only
one data pair $(a_{i_{k}},y_{i_k}^\delta)$ instead of all
data pairs, and thus it enjoys excellent scalability with respect to the data size. It is worth noting
that due to the ill-conditioning of $A$ and the presence of noise in the data $y^\delta$, the exact minimizer
of $J(x)$ is not of interest.

Since its first proposal by Robbins and Monro \cite{RobbinsMonro:1951} for statistical inference, SGD has
received a lot of attention in many diverse research areas (see the monograph \cite{KushnerYin:2003} for
various asymptotic results). Due to its excellent scalability, the interest in SGD and its variants has
grown explosively in recent years in machine learning, and its accelerated variants, e.g., ADAM, have been
established as the workhorse in many challenging deep learning training tasks \cite{Bottou:2010,BottouCurtisNocedal:2018}. It
has also achieved great success in inverse problems, e.g., in computed tomography (known as algebraic
reconstruction techniques or randomized Kaczmarz method \cite{HermanLentLutz:1978,Natterer:2001,StrohmerVershynin:2009,JiaoJinLu:2017}) and optical tomography \cite{ChenLiLiu:2018}.

The theoretical analysis of SGD for solving inverse problems is still in its infancy. Let $e_k^\delta=
x_k^\delta-x^\dag$ be its error with respect to the minimum-norm solution $x^\dag$. Only very recently,
the regularizing property was proved in \cite{JinLu:2019}: when equipped with \textit{a priori} stopping rules,
the mean squared error $\E[\|e_k^\delta\|^2]^\frac12$ of the SGD iterate $x_k^\delta$ converges to zero as
$\delta$ tends to zero, and further, under the canonical power type source condition (see \eqref{eqn:source}
in Assumption \ref{ass:stepsize} below), it converges to zero at a certain rate. However, the result
predicts that SGD can suffer from an undesirable saturation phenomenon for smooth solutions (i.e., with
$\nu>\frac12$): $\E[\|e_k^\delta\|^2]^\frac12$ converges at most at a rate $O(\delta^{\frac12})$, which is
slower than that achieved by Landweber method \cite[Chapter 6]{EnglHankeNeubauer:1996}; see also \cite{JahnJin:2020}
for \textit{a posteriori} stopping using the discrepancy principle and numerical illustration on the saturation
phenomenon for SGD. Thus, SGD is suboptimal
for ``smooth'' inverse solutions with $\nu>\frac{1}{2}$. This phenomenon is attributed to the inherent computational variance of the
SGD approximation $x_k^\delta$, which arises from the use of a random gradient estimate in place of the true
gradient. To the best of our knowledge, it remains unclear whether the saturation phenomenon is intrinsic to SGD.

In this work, we revisit the convergence rate analysis of SGD with a polynomially decaying stepsize
schedule for small initial stepsize $c_0$, and aim at addressing the saturation phenomenon,
under the standard source condition. First we state the standing assumptions for the analysis of
SGD. The choice in (i) is commonly known as a polynomially decaying stepsize schedule.
Part (ii) is the classical source condition, which represents a type of smoothness of the
initial error $x^\dag-x_1$ (with respect to the matrix $B$), and the condition on $B$ is easily achieved by
rescaling the problem. Source type conditions are needed in order to derive convergence rate,
without which the convergence can be arbitrarily slow \cite{EnglHankeNeubauer:1996}. Loosely
speaking, it restricts $x^\dag-x_1$ to a suitable subspace which enables quantitatively bounding
the approximation error. Note that the condition generally is insufficient to ensure a contractive
map for the Landweber method. Below we shall focus on the case $\nu> \frac12$, for which the
current analysis \cite{JinLu:2019} exhibits the saturation phenomenon, as mentioned above. (iii)
assumes that the forward map $A$ takes a special form. Alternatively, it can be viewed as SGD applied to
a preconditioned version of problem \eqref{eqn:lininv}. To validate this condition,
we present some numerical results for typical inverse problems in Subsection\,\ref{sec:condition},
which indicates that this structure is irrelevant to the performance of SGD
in the sense that it performs nearly identically on the problems with or without this
structure. Thus this restriction is due to the limitation of the proof technique; 
{see Remark \ref{rmk:cond} for the obstruction in the proof in the general case.}
\begin{assumption}\label{ass:stepsize}
Let $B=n^{-1}A^tA$ with $\|B\|\leq 1$. The following assumptions hold.
\begin{itemize}
\item[$\rm(i)$] The stepsize $\eta_j = c_0 j^{-\alpha}$, $j=1,\cdots, \alpha\in [0,1)$, with
  $$c_0 \leq \min\big({(\max_i \|a_i\|^2)^{-1}},1\big)\quad\mbox{and}\quad c_0\|B\|\leq (2e)^{-1}.$$
\item[$\rm(ii)$]There exist $w\in\mathbb{R}^m$ and $\nu>\frac12$ such that the exact solution $x^\dag$ satisfies
\begin{equation}\label{eqn:source}
x^\dag-x_1=B^\nu w.
\end{equation}
\item[{\rm(iii)}] {The matrix $A=\Sigma V^t$ with $\Sigma$ being diagonal and nonnegative and $V$ column orthonormal.}
\end{itemize}
\end{assumption}

Now we can state the main result of the work. By choosing the stopping index $k(\delta)$ in accordance with the (unknown) regularity index $\nu$ as
$  k = O(\|w\|\delta^{-1})^\frac{2}{(1+2\nu)(1-\alpha)},$
the result implies a convergence rate
\begin{equation*}
  \E[\|e_{k(\delta)}^\delta\|^2]^\frac12\leq c\|w\|^\frac{1}{1+2\nu}\delta^\frac{2\nu}{1+2\nu},
\end{equation*}
which is identical with that for Landweber method \cite[Chapter 6]{EnglHankeNeubauer:1996}. Thus, under
the given condition, the aforementioned saturation
phenomenon does not occur for SGD. This result partly settles the saturation phenomenon, and complements existing
analysis \cite{JinLu:2019,JinZhouZou:2020}.
\begin{theorem}\label{thm:main}
Let Assumption \ref{ass:stepsize} hold, and $c_0$ be sufficiently small. Then there exist constants
$c_*$ and $c_{**}$, which depend on $\nu$, $n$, $c_0$ and $\alpha$, such that
\begin{equation*}
\E[\|e_k^\delta\|^2]\leq {c_*} k^{-2\nu(1-\alpha)}\|w\|^2+{c_{**}}\delta^2 k^{1-\alpha}.
\end{equation*}
\end{theorem}

The condition $c_0$ being sufficiently small can be made more precise as $c_0=O(n^{-1})$.
Note that the stepsize choice $O(n^{-1})$ has been extensively used in the convergence analysis
of stochastic gradient descent with random shuffling \cite{Ying:2018,Haochen:2019,SafranShamir:2020}.
In Theorem \ref{thm:main}, the constant condition on $c_0$ is not given explicitly.
When $\alpha=0$, the following condition is sufficient 
\begin{align}
2(1+\phi(2\epsilon))nc_0^{2-2\epsilon}\leq1,\quad \mbox{for some } \epsilon\in(\tfrac12,1),
\end{align}
and the function $\phi$ being defined in Lemma \ref{eqn:basic-bdd} below; see Theorem \ref{thm:main:al0}.
The numerical experiments in Section \ref{sec:numer} indicate that with a
small initial stepsize $c_0$, SGD can indeed deliver reconstructions with accuracy comparable with that by
the Landweber method, for a range of regularity index and noise levels, and in the absence of the
smallness condition on $c_0$, the results obtained by SGD are indeed suboptimal. These numerical
results indicate the necessity and sufficiency of a small stepsize for achieving the optimal convergence rate.

The general strategy of proof is to decompose
the error $e_k^\delta := x_k^\delta-x^\dag$ into three components (with $x_k$ being the SGD iterate for exact data $y^\dag$)
\begin{align*}
  x_{k}^\delta - x^\dag  & = (\E[x_k] - x^\dag) + (\E[x_k^\delta]-\E[x_k]) + (x_k^\delta-\E[x_k^\delta]),
\end{align*}
which represent respectively approximation error due to early stopping, propagation error due to data
noise, and stochastic error due to randomness of gradient estimate, and then to bound the terms by
bias-variance decomposition and the triangle inequality as
\begin{align*}
  \E[\|x_{k}^\delta - x^\dag\|^2]  & \leq 2\|\E[x_k] - x^\dag\|^2 + 2\|\E[x_k^\delta]-\E[x_k]\|^2 + \E[\|x_k^\delta-\E[x_k^\delta]\|^2].
\end{align*}
In our analysis, we refine this decomposition by repeatedly expanding the random iterate noise within
the third term and applying the bias-variance decomposition up to the $\ell$th fold; see Theorem
\ref{thm:decomp} for the details. In the decomposition, Assumption \ref{ass:stepsize}(iii) is used in
an essential manner to arrive at a simple recursion. It improves
the existing analysis \cite{JinLu:2019,JinZhouZou:2020} for SGD in the sense that
the stochastic component is further decomposed. Then the analysis proceeds by bounding the first two components separately, and the
third component by recursion, which in turn all involve lengthy computation of certain summations. It is
noteworthy that for the case of a constant stepsize, the convergence analysis can be greatly simplified; see Section
\ref{ssec:al=0} for the details.

Last, we situate the current work within a large body of literature on SGD.
The convergence issue of SGD has been extensively studied in different
senses, and two main lines of research that are related to this work are optimization
and statistical learning, besides the aforementioned results for inverse problems.
In the context of optimization, when the objective function is strictly convex,
many results on the convergence of the iterates to the global minimizer are available;
see, e.g. \cite{JentzenvonWurstemberger:2020} for matching lower and upper bounds, and the references
therein for further results. Note that $J(x)$ in \eqref{eqn:obj} is not strictly
convex. In general, the convergence of SGD is often measured by the optimality
gap (i.e., the expected objective function value to the optimal one) or the magnitude of
the gradient. See the survey \cite{BottouCurtisNocedal:2018} for a recent overview on this
line of research, including advances on nonconvex problems.
Very recently the work \cite{FehrmanGessJentzen:2020} proves the local convergence of SGD with rates to
minima of the objective function, while avoiding convexity or contractivity assumptions.
It is noteworthy that these results cannot be directly compared with the convergence rates
given in Theorem \ref{thm:main}, since the global minimizer to the objective function $J(x)$ is not of practical
interest, due to the ill-conditioning of $A$. This represents one essential difference between the
results from optimization and that from regularization theory. The second line of research is the
generalization error in reproducing kernel Hilbert spaces in statistical learning theory
\cite{YingPontil:2008,TarresYao:2014,DieuleveutBach:2016,LinRosasco:2017,PillaudRudiBach:2018,LeiHuTang:2021}.
These works aim at establishing upper bounds on the generalization error for SGD or its variants
(often combined with a suitable averaging scheme), which differs from the error bound on the
iterate itself. Nonetheless, the high level idea of analysis is similar: both use the
bias-variance decomposition to bound relevant quantities, which often depend on source type conditions given in Assumption \ref{ass:stepsize}(ii). {One major technical novelty
of this work} is to develop a recursive version of the bias-variance decomposition for
the mean squared error.

The rest of the paper is organized as follows. In Section \ref{sec:decomp}, we derive
a novel error decomposition, and then in Section \ref{sec:conv}, we give the convergence
rate analysis, by bounding the three error components of the SGD iterate $x_k^\delta$.
Finally, in Section \ref{sec:numer}, we provide some illustrative numerical experiments
to complement the theoretical analysis. Throughout, the notation $c$, with or without a
subscript, denotes a generic constant, which may differ at each occurrence, but it is
always independent of the iteration number $k$ (and the random index $i_k$) and the noise level $\delta$.

\section{Error decomposition}\label{sec:decomp}
In this part, we present several preliminary estimates and a refined error decomposition.

\subsection{Notation and preliminary estimates}

We will employ the following index sets extensively.
For any $k_1\leq k_2$ and $1\leq i\leq k_2-k_1+1$, let
{ \begin{align*}
   \mathcal{J}_{[k_1,k_2],i}& =\{\{j_\ell\}_{\ell=1}^i : k_1\leq j_i<j_{i-1}<\cdots<j_2<j_1\leq k_2\},\\
 J_i&=\{j_1, j_2, \cdots, j_i\}.
\end{align*}
Note that the set $\mathcal{J}_{[k_1,k_2],i}$ consists of (strictly monotone) multi-indices of length $i$, which arises naturally
in the proof of Theorem \ref{thm:decomp} below. For $i=0$, we adopt the convention
$\mathcal{J}_{[k_1,k_2],0}={\{\emptyset\}}$ and $J_0=\emptyset$.
For all $J_i=\{j_1, \dots, j_i\} \in \mathcal J_{[k_1, k_2],i}$, with $0\leq i\leq k_2-k_1+1$, we define
\begin{equation*}
   J^c_{[k_1, k_2],i}=\{k_1, \dots, k_2\}\setminus J_i,
\end{equation*}
where we omit the dependency on $J_i$ for notational simplicity. In particular,
\begin{equation*}
  \mathcal{J}_{[k_1,k_2],1}=\{\{k_1\},\ldots,\{k_2\}\}\quad \mbox{and}\quad J_{[k_1,k_2],0}^c=\{k_1,\cdots, k_2\}.
\end{equation*}
For $i> k_2-k_1+1$, we adopt the convention $\mathcal{J}_{[k_1,k_2],i}=\{\emptyset\}$, $J_i=\emptyset$,
$J_{[k_1,k_2],i}^c=\emptyset$.}

The next lemma collects useful identities on the summation over the indices $\mathcal{J}_{[1,k],i+1}$.
\begin{lemma}\label{lem:sum_J}
The following identities hold:
\begin{align}
\sum_{J_{i+1}\in\mathcal{J}_{[1,k],i+1}}&=\sum_{J_i\in\mathcal{J}_{[2,k],i}}\sum_{j_{i+1}=1}^{j_i-1} =\sum_{j_{i+1}=1}^{k-i}\sum_{J_i\in\mathcal{J}_{[j_{i+1}+1,k],i}}.\label{eqn:sum2}
\end{align}
\end{lemma}
\begin{proof}
The identities are direct from the definition:
\begin{equation*}
\begin{split}
\sum_{J_{i+1}\in\mathcal{J}_{[1,k],i+1}}&=\sum_{j_1=i+1}^{k}\cdots\sum_{j_{i}=2}^{j_{i-1}-1}\sum_{j_{i+1}=1}^{j_i-1}=\sum_{J_i\in\mathcal{J}_{[2,k],i}}\sum_{j_{i+1}=1}^{j_i-1},\\
\sum_{J_{i+1}\in\mathcal{J}_{[1,k],i+1}}&=\sum_{j_{i+1}=1}^{k-i}\sum_{j_{i}=j_{i+1}+1}^{k-i+1}\cdots\sum_{j_1=j_2+1}^{k}=\sum_{j_{i+1}=1}^{k-i}\sum_{J_i\in\mathcal{J}_{[j_{i+1}+1,k],i}}.
\end{split}
\end{equation*}
This shows directly the assertion.
\end{proof}

We use the following elementary inequality extensively.
\begin{lemma}
For any $k\in\mathbb{N}$ and $s\in\mathbb{R}$, there holds
\begin{equation}\label{eqn:basic-bdd}
  \sum_{j=1}^k j^{-s} \leq \left\{\begin{array}{ll}
   2^{1-s}(1-s)^{-1}k^{1-s}, &s<0,\\
    (1-s)^{-1}k^{1-s}, & s \in [0,1),\\
    2\max(\ln k,1), & s =1, \\
    s(s-1)^{-1}, & s >1.
  \end{array}\right.
\end{equation}
\end{lemma}
Throughout, we denote the constant and power on the right hand side of the inequality \eqref{eqn:basic-bdd} by
$\phi(s)$ and $k^{\max(1-s,0)}$, respectively, with the shorthand
$k^{\max(0,0)}=\max(\ln k,1).$

The next result bounds the spectral norm of the matrix product $\Pi_J(B)B^s$, which,
for each index set $J$, is defined by (with the convention $\Pi_{\emptyset}(B)=I$)
\begin{equation*}
  \Pi_{J}(B)=\prod_{j\in J}(I-\eta_{j}B).
\end{equation*}

\begin{lemma}\label{lem:kernel}
Under Assumption \ref{ass:stepsize}{\rm(i)}, { for any $s>0$ and $J_\ell \in \mathcal J_{[k', k],\ell}$ with $k'\leq k$, $0\leq\ell< k+1-k'$,}
\begin{align*}
\|\Pi_{J_{[k',k],\ell}^c}(B) B^s\|\leq s^s(ec_0)^{-s} (k+1-k'-\ell)^{-s}k^{\alpha s}.
\end{align*}
\end{lemma}
\begin{proof}
For any $s>0$ { and $J_\ell \in \mathcal J_{[k', k],\ell}$ with $k'\leq k$, $0\leq\ell< k+1-k'$,} there holds
\begin{align*}
\|\Pi_{J_{[k',k],\ell}^c}(B) B^s\|\leq\sup_{\lambda\in{\rm Sp}(B)} |\lambda^s \Pi_{J_{[k',k],\ell}^c}(\lambda)|=\sup_{\lambda
 \in{\rm Sp}(B)}\lambda^s \prod_{i\in J_{[k',k],\ell}^c}(1-\eta_i\lambda),
\end{align*}
where ${\rm Sp}(B)$ denotes the spectrum of $B$. For any $x\in\mathbb{R}$, there holds the inequality $1-x\leq e^{-x}$, and thus
\begin{align*}
\lambda^s \prod_{i\in J_{[k',k],\ell}^c}(1-\eta_i\lambda)\leq \lambda^s \prod_{i\in J_{[k',k],\ell}^c}e^{-\eta_i\lambda}= \lambda^s e^{-\lambda\sum_{i\in J_{[k',k],\ell}^c}\eta_i}.
\end{align*}
For the function $g(\lambda)=\lambda^s e^{-\lambda a}$, with $a>0$,
the maximum is attained at $\lambda^*=sa^{-1}$, with a maximum value $s^s(ea)^{-s}$.
Then setting $a=\sum_{i\in J_{[k',k],\ell}^c}\eta_i$ and applying the inequality
$a \geq c_0 (k+1-k'-\ell)k^{-\alpha}$
complete the proof of the lemma.
\end{proof}

The last lemma gives two useful bounds on the summations over the set $\mathcal{J}_{[1,k],i}$.
\begin{lemma}\label{lem:kernel2}
The following estimates hold.
\begin{itemize}
\item[$\rm(i)$] For any $k\geq2$, { $\alpha\in [0,1)$} and $2\leq i\leq k$, there holds
\begin{align*}
\sum_{J_i\in\mathcal{J}_{[1,k],i}}\prod_{t=1}^{i}j_t^{-2\alpha} 
\leq\phi(2\alpha)^i (k^{\max(1-2\alpha,0)})^i.
\end{align*}
\item[$\rm(ii)$]For any $j=0,\cdots,k-1$ and $i=1,\cdots,k-j$, we have
\begin{equation*}
\sum_{J_i\in\mathcal{J}_{[j
+1,k],i}}1\leq\frac{(k-j)^i}{i!}.
\end{equation*}
\end{itemize}
\end{lemma}
\begin{proof}
Assertion (i) follows from \eqref{eqn:basic-bdd} as
\begin{align*}
\sum_{J_i\in\mathcal{J}_{[1,k],i}}\prod_{t=1}^{i}j_t^{-2\alpha}&=
\sum_{j_1=i}^k\sum_{j_2={i-1}}^{j_1-1}\cdots\sum_{j_i=1}^{j_{i-1}-1}j_i^{-2\alpha}\prod_{t=1}^{i-1}j_t^{-2\alpha}\\
&\leq\prod_{t=1}^{i}\Big(\sum_{j_t=1}^k j_t^{-2\alpha}\Big)\leq
\big(\phi(2\alpha)k^{\max(1-2\alpha,0)}\big)^i.
\end{align*}
By the definition of the index set $\mathcal{J}_{[j+1,k],i}$, we have the identity
\begin{align*}
 &\sum_{J_i\in\mathcal{J}_{[j
+1,k],i}}1=\sum_{j_{i}=j+1}^{k-i+1}\cdots\sum_{j_2=j_3+1}^{k-1}\sum_{j_1=j_2+1}^{k}1\\
\leq&\sum_{j_{i}=j+1}^{k-1}\cdots\sum_{j_2=j_3+1}^{k-1}\sum_{j_1=j_2+1}^{k}1=\sum_{j_{i}=j+1}^{k-1}\cdots\sum_{j_2=j_3+1}^{k-1}(k-j_2).
\end{align*}
then assertion (ii) follows by repeatedly applying the inequality
\begin{equation*}
  \sum_{t=1}^T t^s\leq(s+1)^{-1}(T+1)^{s+1},\quad \forall T\in \mathbb{N}, s\geq 0.
\end{equation*}
This completes the proof of the lemma.
\end{proof}

\subsection{Error decomposition}

Now we derive an important error decomposition. Below, we denote the SGD iterates for the exact data $y^\dag$
and noisy data $y^\delta$ by $x_k$ and $x_k^\delta$, respectively, and also use the following
shorthand notation:
\begin{equation*}
  \bar A= n^{-\frac12}A,\quad \bar \xi = n^{-\frac12}\xi,\quad \bar{\delta}=n^{-\frac12}\delta, \quad \mbox{and}\quad e_k=x_k-x^\dag.
\end{equation*}
The following result plays a central role in the convergence analysis.
\begin{theorem}\label{thm:decomp}
Under Assumption \ref{ass:stepsize}{\rm(iii\rm)},
for any $0\leq\ell<k$, the following error decomposition holds
\begin{align}\label{eqn:recu}
\E[\|e_{k+1}^\delta\|^2] \leq \sum_{i=0}^\ell {\rm I}_{i,1}^\delta+\sum_{i=0}^\ell {\rm I}_{i,2}^\delta+{ ({\rm I}_\ell^\delta)^c},
\end{align}
where the terms ${\rm I}_{i,j}^\delta$, {$i=0,1,\cdots,\ell$,} $j=1,2$, are defined by
\begin{align*}
{\rm I_{0,1}^\delta}&=2\|\Pi_{J^c_{[1,k],0}}(B)e_1\|^2,\\
{\rm I^\delta_{0,2}}&=2{\bar\delta}^2\Big(\Big\|\sum_{j=1}^{k}\eta_j\Pi_{J^c_{[j+1,k],0}}(B)B^{\frac12}\Big\|^2+(n-1)\sum_{j=1}^{k}\eta_j^2\|\Pi_{J^c_{[j+1,k],0}}(B)B^\frac12\|^2\Big),\\
{\rm I}_{i,1}^\delta&=2^{i+1}(n-1)^i \sum_{J_i\in\mathcal{J}_{[1,k],i}}\prod_{t=1}^{i}\eta_{j_t}^2\|\Pi_{J_{[1,k],i}^c}(B) B^i e_1\|^2, \quad \forall 1\leq i\leq \ell,\\
{\rm I}_{i,2}^\delta&=2^{i+1}(n-1)^i{\bar\delta}^2 \sum_{J_i\in\mathcal{J}_{[2,k],i}}\prod_{t=1}^{i}\eta_{j_t}^2\Big(\Big\|\sum_{j_{i+1}=1}^{j_i-1}\eta_{j_{i+1}}\Pi_{J_{[j_{i+1}+1,k],i}^c}(B) B^{i+\frac12}\Big\|^2\\
&\;\qquad+(n-1)\sum_{j_{i+1}=1}^{j_i-1}\eta_{j_{i+1}}^2\|\Pi_{J_{[j_{i+1}+1,k],i}^c}(B) B^{i+\frac12}\|^2\Big), \quad \forall 1\leq i\leq \ell,\\
({\rm I}_\ell^\delta)^c
&=2^{\ell+1}(n-1)^{\ell+1} \sum_{J_{\ell+1}\in\mathcal{J}_{[1,k],\ell+1}}\prod_{t=1}^{\ell+1}\eta_{j_t}^2\E[\|\Pi_{J^c_{[j_{\ell+1}+1,k],\ell}}(B)B^{\ell+1} e_{j_{\ell+1}}^\delta\|^2].
\end{align*}
\end{theorem}
\begin{proof}
{Recall that $J_i=\{j_1,\cdots,j_i\}$ for any $i\geq1$ and $J_0=\emptyset$.} By the definition of the SGD iteration \eqref{eqn:SGD}, we have
\begin{equation}\label{eq:sgditer}
e_{k}^\delta  =(I-\eta_{k-1} B)e_{k-1}^\delta+\eta_{k-1} H_{k-1}
 =\Pi_{J^c_{[1,k-1],0}} (B)e_1^\delta+\sum_{j=1}^{k-1}\big(\eta_j \Pi_{J^c_{[j+1,k-1],0}}(B)H_j\big),
\end{equation}
where $H_j$ is defined by
\begin{align}\label{eqn:Hj}
H_j& =B e_j^\delta-\big((a_{i_j},x_j^\delta)-y_{i_j}^\delta\big)a_{i_j}
=\big(B-a_{i_j}a_{i_j}^t) e_j^\delta+\xi_{i_j} a_{i_j}.
\end{align}
By bias-variance decomposition and triangle inequality, we have
\begin{equation}\label{eq:de}
\begin{split}
\E[\|e_{k+1}^\delta\|^2] & =\|\E[e_{k+1}^\delta]\|^2+\E[\|\E[x_{k+1}^\delta]-x_{k+1}^\delta\|^2]\\
& \leq 2\|\E[e_{k+1}]\|^2+2\|\E[x_{k+1}-x_{k+1}^\delta]\|^2+\E[\|\E[x_{k+1}^\delta]-x_{k+1}^\delta\|^2].
\end{split}
\end{equation}
It is known that the following estimates hold \cite{JinLu:2019}
\begin{align}
\|\E[e_{k+1}]\|&=\|\Pi_{J^c_{[1,k],0}}(B)e_1\|,\nonumber\\
\|\E[x_{k+1}-x_{k+1}^\delta]\|&\leq {\bar\delta}\|\sum_{j=1}^{k}\eta_j\Pi_{J^c_{[j+1,k],0}}(B)B^{\frac12}\|,\nonumber\\
\E[\|x_{k+1}^\delta-\E[x_{k+1}^\delta]\|^2]&= \sum_{j=1}^{k}\eta_j^2\E[\|\Pi_{J^c_{[j+1,k],0}}(B)A^t N_j^\delta\|^2],\label{eqn:stco_0}
\end{align}
with the iteration noise $N_j^\delta$ (at the $j$th SGD iteration) given by
\begin{equation*}
N_j^\delta=n^{-1}(A x_j^\delta-y^\delta)-((a_{i_j},x_j^\delta)-y^\delta_{i_j})b_{i_j},
\end{equation*}
where $b_i=(0,\ldots,0,1,0,\ldots,0)^t\in\mathbb{R}^n$ denotes the ${i}$th canonical Cartesian basis vector. Let
\begin{equation*}
  \tilde{N}^\delta_j=((a_{i_j},x_j^\delta)-y^\delta_{i_j})b_{i_j}=((a_{i_j},e_j^\delta)-\xi_{i_j})b_{i_j}.
\end{equation*}
Then the iteration noise $N_j^\delta$ can be rewritten as
\begin{equation*}
  N^\delta_j=\E[\tilde{N}^\delta_j|{ \mathcal{F}_{j}}]-\tilde{N}^\delta_j.
\end{equation*}
Next we claim that under Assumption \ref{ass:stepsize}(iii), there holds
\begin{align}\label{eqn:Axi-N}
\|\Pi_{J_{[j+1,k],0}^c}(B)A^t(A e_j^\delta-\xi)\|^2
=\sum_{i=1}^n\|\Pi_{J_{[j+1,k],0}^c}(B)A^t \big((a_{i},e^\delta_j)-\xi_i\big)b_{i}\|^2.
\end{align}
Actually, in view of Assumption \ref{ass:stepsize}(iii), for any $1\leq j\leq k$, there hold
\begin{align*}
Ae_j^{\delta}-\xi=\sum_{i=1}^n \big((a_{i},e^\delta_j)-\xi_i\big)b_{i} \quad \mbox{and}\quad \Pi_{J_{[j+1,k],0}^c}(B)A^t = V\Pi_{J_{[j+1,k],0}^c}(n^{-1}\Sigma^t \Sigma) \Sigma^t.
\end{align*}
Then the claim \eqref{eqn:Axi-N} follows from these two identities and column orthonormality of $V$ as
\begin{align*}
&\|\Pi_{J_{[j+1,k],0}^c}(B)A^t(A e_j^\delta-\xi)\|^2 
= 
\Big\|V\sum_{i=1}^n\Pi_{J_{[j+1,k],0}^c}(n^{-1}\Sigma^t \Sigma) \Sigma^t   \big((a_{i},e^\delta_j)-\xi_i\big)b_{i}\Big\|^2\\
=&\sum_{i=1}^n\|V\Pi_{J_{[j+1,k],0}^c}(n^{-1}\Sigma^t \Sigma) \Sigma^t   \big((a_{i},e^\delta_j)-\xi_i\big)b_{i}\|^2
=\sum_{i=1}^n\|\Pi_{J_{[j+1,k],0}^c}(B)A^t \big((a_{i},e^\delta_j)-\xi_i\big)b_{i}\|^2.
\end{align*}
Thus, by the bias-variance decomposition and the definitions of the notation $\bar A$
and $\bar\xi$ etc, we have, for any $j=1,\cdots,k,$
\begin{align*}
 &\E[\|\Pi_{J_{[j+1,k],0}^c}(B)A^t N_j^\delta\|^2|{ \mathcal{F}_{j}}]
=\E[\|\Pi_{J_{[j+1,k],0}^c}(B)A^t \tilde{N}^\delta_j\|^2|{ \mathcal{F}_{j}}]-\|\Pi_{J_{[j+1,k],0}^c}(B)A^t \E[\tilde{N}^\delta_j|{ \mathcal{F}_{j}}]\|^2\\
=&\frac1n \sum_{i=1}^n\|\Pi_{J_{[j+1,k],0}^c}(B)A^t \big((a_{i},e^\delta_j)-\xi_i\big)b_{i} \|^2-\|\Pi_{J_{[j+1,k],0}^c}(B)\frac{A^t }{n}(A e^\delta_j-\xi)\|^2\\
=&(n-1)\|\Pi_{J_{[j+1,k],0}^c}(B)(B e_j^\delta-\bar A^t\bar\xi)\|^2.
\end{align*}
By the Cauchy-Schwarz inequality, the identity
$\|\Pi_{J_{[j+1,k],0}^c}(B)\bar A^t\|^2=\|\Pi_{J_{[j+1,k],0}^c}(B)B^\frac12\|^2,$
and the triangle inequality, we deduce from \eqref{eqn:stco_0} that
\begin{align*}
&\E[\|x_{k+1}^\delta-\E[x_{k+1}^\delta]\|^2]
=\sum_{j=1}^{k}\eta_j^2\E\big[\E[\|\Pi_{J_{[j+1,k],0}^c}(B)A^t N_j^\delta\|^2|{ \mathcal{F}_{j}}]\big]\\
=&(n-1)\sum_{j=1}^{k}\eta_j^2\E[\|\Pi_{J_{[j+1,k],0}^c}(B)(B e_j^\delta-\bar A^t\bar\xi)\|^2]\\
\leq&2(n-1)\sum_{j=1}^{k}\eta_j^2\Big(\E[\|\Pi_{J_{[j+1,k],0}^c}(B)B e_j^\delta\|^2]+\|\Pi_{J_{[j+1,k],0}^c}(B)\bar A^t\bar\xi\|^2\Big)\\
\leq&2(n-1)\sum_{j=1}^{k}\eta_j^2\E[\|\Pi_{J_{[j+1,k],0}^c}(B)B e_j^\delta\|^2]+2(n-1){\bar\delta}^2\sum_{j=1}^{k}\eta_j^2\|\Pi_{J_{[j+1,k],0}^c}(B)B^\frac12\|^2.
\end{align*}
By the definitions of ${\rm I^\delta_{0,1}}$, ${\rm I^\delta_{0,2}}$ and $({\rm I}^\delta_0)^c$,
we have
\begin{align*}
\E[\|e_{k+1}^\delta\|^2] \leq {\rm I^\delta_{0,1}}+{\rm I^\delta_{0,2}}+({\rm I}^\delta_0)^c.
\end{align*}
Next further expanding $\{e^\delta_j\}_{j=2}^k$ in the expression of $({\rm I}^\delta_0)^c$ using \eqref{eq:sgditer} gives
\begin{align*}
 &({\rm I}_0^\delta)^c =2(n-1)\eta_1^2\|\Pi_{J_{[2,k],0}^c}(B)B e_1^\delta\|^2\\
&+2(n-1)\sum_{j_1=2}^{k}\eta_{j_1}^2\E\Big[\Big\|\Pi_{J_{[j_1+1,k],0}^c}(B)B \Big(\Pi_{J_{[1,j_1-1],0}^c}(B)e_1^\delta+\sum_{{j_2}=1}^{{j_1}-1}\big(\Pi_{J_{[j_2+1,j_1-1],0}^c}(B)\eta_{j_2} H_{j_2}\big)\Big)\Big\|^2\Big].
\end{align*}
Now using the definition of $H_j$ in \eqref{eqn:Hj}, we obtain
\begin{align*}
\sum_{{j_2}=1}^{{j_1}-1}\big(\Pi_{J_{[j_2+1,j_1-1],0}^c}(B)\eta_{j_2} H_{j_2}\big)
=&\sum_{{j_2}=1}^{{j_1}-1}\big(\eta_{j_2}\Pi_{J_{[j_2+1,j_1-1],0}^c}(B) (B-a_{i_{j_2}}a_{i_{j_2}}^t) e_{j_2}^\delta\big)\\
 &+\sum_{{j_2}=1}^{{j_1}-1}\big(\eta_{j_2}\Pi_{J_{[j_2+1,j_1-1],0}^c}(B) \xi_{i_{j_2}}
a_{i_{j_2}}\big).
\end{align*}
Thus, we can further bound $({\rm I}_0^\delta)^c$ by
\begin{align*}
({\rm I}_0^\delta)^c\leq&2(n-1)\eta_1^2\|\Pi_{J_{[2,k],0}^c}(B)B e_1^\delta\|^2\\
  &+4(n-1)\sum_{j_1=2}^{k}\eta_{j_1}^2\E\Big[\Big\|\Pi_{J_{[j_1+1,k],0}^c}(B)B\sum_{{j_2}=1}^{{j_1}-1}\big(\eta_{j_2}\Pi_{J_{[j_2+1,j_1-1],0}^c}(B) \xi_{i_{j_2}}
a_{i_{j_2}}\big)\Big\|^2\Big]\\
&+4(n-1)\sum_{j_1=2}^{k}\eta_{j_1}^2\E\Big[\Big\|\Pi_{J_{[j_1+1,k],0}^c}(B)B \Big(\Pi_{J_{[1,j_1-1],0}^c}(B)e_1^\delta\\
 &\quad +\sum_{{j_2}=1}^{{j_1}-1}\big(\eta_{j_2}\Pi_{J_{[j_2+1,j_1-1],0}^c}(B) (B-a_{i_{j_2}}a_{i_{j_2}}^t) e_{j_2}^\delta\big)\Big)\Big\|^2\Big]\\
\leq&2(n-1)\eta_1^2\|\Pi_{J_{[2,k],0}^c}(B)B e_1^\delta\|^2
+4(n-1)\sum_{j_1=2}^{k}\eta_{j_1}^2\underbrace{\E[\|\sum_{{j_2}=1}^{{j_1}-1}\big(\eta_{j_2}\Pi_{J_{[j_2+1,k],1}^c}(B)B \xi_{i_{j_2}}
a_{i_{j_2}}\big)\|^2]}_{\rm II_1}\\
&+4(n-1)\sum_{j_1=2}^{k}\eta_{j_1}^2\underbrace{\E[\|\Pi_{J_{[1,k],1}^c}(B) B e_1^\delta+\sum_{{j_2}=1}^{{j_1}-1}\big(\eta_{j_2}\Pi_{J_{[j_2+1,k],1}^c}(B)B (B-a_{i_{j_2}}a_{i_{j_2}}^t) e_{j_2}^\delta\big)\|^2]}_{\rm II_2}.
\end{align*}
Next we simplify the two terms ${\rm II}_1$ and ${\rm II}_2$. Under Assumption \ref{ass:stepsize}(iii), direct computation gives, for any $j_1=2,\cdots,k$ and $j,j'=1,\cdots,j_1-1$,
\begin{align}\label{eqn:Axi-exp}
&\E[\langle \Pi_{J_{[j'+1,k],1}^c}(B)B\xi_{i_{j'}}a_{i_{j'}}, \Pi_{J_{[j+1,k],1}^c}(B)B\xi_{i_j}a_{i_j}\rangle]\nonumber\\
=&\left\{\begin{aligned}
  n\langle \Pi_{J_{[j'+1,k],1}^c}(B)B\bar A^t\bar\xi,\Pi_{J_{[j+1,k],1}^c}(B)B \bar A^t\bar\xi\rangle,&  \quad j'=j,\\
 \langle \Pi_{J_{[j'+1,k],1}^c}(B)B\bar A^t\bar\xi,\Pi_{J_{[j+1,k],1}^c}(B)B \bar A^t\bar\xi\rangle, & \quad j'\neq j.
\end{aligned}\right.
\end{align}
Indeed, the case $j'\neq j$ follows directly. Meanwhile, under Assumption \ref{ass:stepsize}(iii), we have
$$\Pi_{J_{[j+1,k],1}^c}(B)BA^t\xi = V\Pi_{J_{[j+1,k],1}^c}(n^{-1}\Sigma^t\Sigma)(n^{-1}\Sigma^t\Sigma)\Sigma^t\sum_{i=1}^n\xi_{i}b_{i},$$ and  $V\Sigma^t \xi_i b_i= A^t\xi_i b_i =  a_i\xi_i$.
This and the column orthonormality of the matrix $V$ imply{
\begin{align*}
  &n\|\Pi_{J_{[j+1,k],1}^c}(B)B \bar{A}^{t}\bar{\xi}\|^2  =n^{-1}\Big\| V\Pi_{J_{[j+1,k],1}^c}(n^{-1}\Sigma^t\Sigma)(n^{-1}\Sigma^t\Sigma)\Sigma^t\sum_{i=1}^n\xi_{i}b_{i}\Big\|^2\\
  =&n^{-1}\sum_{i=1}^n\|V\Pi_{J_{[j+1,k],1}^c}(n^{-1}\Sigma^t\Sigma)(n^{-1}\Sigma^t\Sigma)\Sigma^t \xi_{i}b_i\|^2
  =n^{-1}\sum_{i=1}^n\|\Pi_{J_{[j+1,k],1}^c}(B)B a_i\xi_{i}\|^2\\
  =&\E[\|\Pi_{J_{[j+1,k],1}^c}(B)B\xi_{i_{j}}a_{i_{j}}\|^2].
\end{align*}}
This and the bias-variance decomposition imply that the term ${\rm II}_1$ can be simplified to
\begin{align*}
{\rm II}_1=&\Big\|\sum_{{j_2}=1}^{{j_1}-1}\eta_{j_2}\Pi_{J_{[j_2+1,k],1}^c}(B)B \bar A^t\bar\xi\Big\|^2+(n-1)\sum_{{j_2}=1}^{{j_1}-1}\eta_{j_2}^2\|\Pi_{J_{[j_2+1,k],1}^c}(B)B \bar A^t\bar \xi\|^2\\
\leq&{\bar\delta}^2\Big\|\sum_{{j_2}=1}^{{j_1}-1}\eta_{j_2}\Pi_{J_{[j_2+1,k],1}^c}(B)B^{\frac{3}{2}}\Big\|^2+(n-1){\bar\delta}^2\sum_{{j_2}=1}^{{j_1}-1}\eta_{j_2}^2\|\Pi_{J_{[j_2+1,k],1}^c}(B)B^{\frac32}\|^2.
\end{align*}
Further, by the measurability of $e_j^\delta$ with respect to {$\mathcal{F}_{j}$}, we have
\begin{align}\label{eqn:noise-ind1}
\E[\langle e_1^\delta, (B-a_{i_j}a_{i_j}^t)e_j^\delta\rangle]=&\langle e_1^\delta, \E[\E[(B-a_{i_j}a_{i_j}^t)e_j^\delta|{\mathcal{F}_{j}}]]\rangle=0, \quad \forall j,
\end{align}
since $e_1^\delta$ is deterministic, and similarly,
\begin{equation}\label{eqn:noise-ind2}
\begin{aligned}
 &\E[\langle (B-a_{i_{j'}}a_{i_{j'}}^t)e_{j'}^\delta, (B-a_{i_j}a_{i_j}^t)e_j^\delta\rangle]\\
=&\E[\langle (B-a_{i_{j'}}a_{i_{j'}}^t)e_{j'}^\delta, \E[(B-a_{i_j}a_{i_j}^t)e_j^\delta|{\mathcal{F}_{j}}]\rangle]=0, \quad \forall j'<j.
\end{aligned}
\end{equation}
Consequently, there holds
\begin{align*}
{\rm II}_2
=&\|\Pi_{J_{[1,k],1}^c}(B) B e_1^\delta\|^2+\sum_{{j_2}=1}^{{j_1}-1}\eta_{j_2}^2\E[\|\Pi_{J_{[j_2+1,k],1}^c}(B)B (B-a_{i_{j_2}}a_{i_{j_2}}^t) e_{j_2}^\delta\|^2].
\end{align*}
Combining these estimates with the definitions of the quantities ${\rm I}_{1,1}^\delta$, ${\rm I}_{1,2}^\delta$ and $({\rm I}_{1}^\delta)^c$ gives
\begin{align*}
&({\rm  I}_0^\delta)^c
\leq2(n-1)\eta_1^2\|\Pi_{J_{[2,k],0}^c}(B)B e_1^\delta\|^2\\
&+4(n-1){\bar\delta}^2\sum_{j_1=2}^{k}\eta_{j_1}^2\Big(\Big\|\sum_{{j_2}=1}^{{j_1}-1}\eta_{j_2}\Pi_{J_{[j_2+1,k],1}^c}(B)B^{\frac{3}{2}}\Big\|^2+(n-1)\sum_{{j_2}=1}^{{j_1}-1}\eta_{j_2}^2\|\Pi_{J_{[j_2+1,k],1}^c}(B)B^{\frac32}\|^2\Big)\\
&+4(n-1)\sum_{j_1=2}^{k}\eta_{j_1}^2\Big(\|\Pi_{J_{[1,k],1}^c}(B) B e_1^\delta\|^2+\sum_{{j_2}=1}^{{j_1}-1}\eta_{j_2}^2\E[\|\Pi_{J_{[j_2+1,k],1}^c}(B)B (B-a_{i_{j_2}}a_{i_{j_2}}^t) e_{j_2}^\delta\|^2]\Big)\\
\leq&4(n-1)\sum_{j_1=1}^{k}\eta_{j_1}^2\|\Pi_{J_{[1,k],1}^c}(B) B e_1^\delta\|^2\\
&+4(n-1){\bar\delta}^2\sum_{j_1=2}^{k}\eta_{j_1}^2\Big(\Big\|\sum_{{j_2}=1}^{{j_1}-1}\eta_{j_2}\Pi_{J_{[j_2+1,k],1}^c}(B)B^{\frac{3}{2}}\Big\|^2+(n-1)\sum_{{j_2}=1}^{{j_1}-1}\eta_{j_2}^2\|\Pi_{J_{[j_2+1,k],1}^c}(B)B^{\frac32}\|^2\Big)\\
&+4(n-1)\sum_{j_1=2}^{k}\eta_{j_1}^2\Big(\sum_{{j_2}=1}^{{j_1}-1}\eta_{j_2}^2\E[\|\Pi_{J_{[j_2+1,k],1}^c}(B)B (B-a_{i_{j_2}}a_{i_{j_2}}^t) e_{j_2}^\delta\|^2]\Big)\\
=&{\rm I^\delta_{1,1}}+{\rm I^\delta_{1,2}}+({\rm I}^\delta_1)^c.
\end{align*}
Similar to the analysis of $({\rm I}^\delta_0)^c$, by repeating the argument, we obtain
\begin{align*}
({\rm I}^\delta_1)^c
=&4(n-1)^2\sum_{j_1=2}^{k}\sum_{{j_2}=1}^{{j_1}-1}\eta_{j_1}^2\eta_{j_2}^2\E[\|\Pi_{J_{[j_2+1,k],1}^c}(B)B^2 e_{j_2}^\delta\|^2].
\end{align*}
In general, we can derive
\begin{equation*}
({\rm I}_\ell^\delta)^c=2^{\ell+1}(n-1)^{\ell} \sum_{J_{\ell+1}\in\mathcal{J}_{[1,k],\ell+1}}\prod_{t=1}^{\ell+1}\eta_{j_t}^2\E[\|\Pi_{J^c_{[j_{\ell+1}+1,k],\ell}}(B)B^{\ell}(B-a_{i_{j_{\ell+1}}}a_{i_{j_{\ell+1}}}^t) e_{j_{\ell+1}}^\delta\|^2].
\end{equation*}
Then repeating the preceding argument, and noting the relation $e_1^\delta=e_1$ complete the proof.
\end{proof}

\begin{remark}
In Theorem \ref{thm:decomp}, Assumption \ref{ass:stepsize}(iii) plays a central role in the refined error
decomposition, at two places, i.e., \eqref{eqn:Axi-N} and \eqref{eqn:Axi-exp}.
Intuitively, the condition essentially assumes low correlation between the rows of the matrix $A$,
in analogy to the mutual coherence condition in compressed sensing \cite{Donoho:2003}.
The numerical experiments in Section \ref{sec:numer} indicate that SGD performs
comparably with or without this assumption.
\end{remark}

\begin{remark}\label{rmk:cond}
It is instructive to see the obstruction in extending the argument of Theorem \ref{thm:decomp} to a general matrix $A$
with exact data (i.e., $\xi=0$), in the absence of Assumption \ref{ass:stepsize}(iii). Let the singular value decomposition of $A$ be
$A=U\Sigma V^t$, with $\Sigma\in\mathbb{R}^{n\times m}$ being diagonal with
positive diagonal entries $\{\sigma_i\}_{i=1}^r$ (with $r\leq \min(m,n)$ being the rank, ordered nonincreasingly)
and $U=[u_1,\cdots,u_n]{\in\mathbb{R}^{n\times n}}$ and $V=[v_1,\cdots,v_m]{\in\mathbb{R}^{m\times m}}$ being column
orthonormal. Now consider the right-hand side and left-hand side, denoted by $\rm
RHS$ and $\rm LHS$, respectively, of the crucial identity \eqref{eqn:Axi-N} with
a random index set $J$ and a random vector $e\in\mathbb{R}^m$
(by suppressing the subscripts). In view of the identity $a_i^t =b_i^t A$, we have
\begin{align*}
{\rm LHS}=&\|V\Pi_{J}(n^{-1}\Sigma^t \Sigma)\Sigma^t U^t Ae\|^2=\|D U^t A e\|^2=\sum_{j=1}^n (d_j u_j^t(Ae))^2=\sum_{j=1}^r d_j^2(u_j^t(Ae))^2,\\ 
{\rm RHS}
=&\sum_{i=1}^n\|V\Pi_{J}(n^{-1}\Sigma^t \Sigma)\Sigma^t U^t b_{i}b_{i}^t Ae\|^2=\sum_{i=1}^n\|D U^t b_{i}(Ae)_i\|^2
=\sum_{j=1}^r d_j^2\sum_{i=1}^n(u_{ji}(Ae)_i)^2,
\end{align*}
with the diagonal matrix $D$ given by $D=\Pi_{J}(n^{-1}\Sigma^t \Sigma)\Sigma^t:={\rm diag}(d_1,\cdots, d_n),$
with the first $r$ entries being strictly positive. Since the index set $J$ is arbitrary, the existence of
a constant $c$ (independent of $J$) such that ${\rm RHS}\leq c {\rm LHS}$ essentially requires
\begin{align*}
\sum_{i=1}^n(u_{ji}(Ae)_i)^2\leq c (u_j^t(Ae))^2,\quad j=1,\ldots,r . 
\end{align*}
Since $Ae=\sum_{
\ell=1}^{r}\sigma_\ell u_\ell v_\ell^t e$, the above inequality is equivalent to
\begin{align}\label{eqn:inequal-onecomp}
\sum_{i=1}^n(u_{ji}(Ae)_i)^2\leq c (\sigma_j v_j^t e)^2.
\end{align}
{When Assumption \ref{ass:stepsize} (iii) does not hold, there exist some $j\leq r$ and two nonzero elements $u_{j{i_1}}, u_{j{i_2}}$.
Now we take any $ e\in\mathbb{R}^m$ such that $v_j^t e=0$ and $(Ae)_{i_1}\neq 0$ or $(Ae)_{i_2}\neq 0$.}
Then the left hand side of \eqref{eqn:inequal-onecomp} is strictly positive,
and the right hand side vanishes. Thus, there is no constant $c$ such that this inequality holds. This shows the delicacy
of the analysis for a general matrix $A$. Nonetheless, the numerical experiments in
Section \ref{sec:numer} indicate that the saturation phenomenon actually also does
not occur for a general matrix, so long as the stepsize $c_0$ is sufficiently small.
Thus, we believe that the restriction is due to the limitation of the proof
technique. Note that the convergence analysis in Section \ref{sec:conv} remain valid
provided that relaxed versions of the identities \eqref{eqn:Axi-N} and \eqref{eqn:Axi-exp} hold
but with different constants in the final estimate.
\end{remark}

The proof of Theorem \ref{thm:decomp} also gives the following error decomposition
for exact data $y^\dag$.
\begin{corollary}
For any $0\leq\ell<k$, the following error decomposition holds
\begin{align}\label{eqn:recu-exact}
 \E[\|e_{k+1}\|^2] \leq \sum_{i=0}^\ell {\rm I}_{i}+{ ({\rm I}_\ell)^c},
\end{align}
where the terms ${\rm I}_{i}$, {$i=0,1,\cdots,\ell$}, are defined by
\begin{align*}
{\rm I_{0}}&=\|\Pi_{J^c_{[1,k],0}}(B)e_1\|^2,\\
{\rm I}_{i}&=(n-1)^i \sum_{J_i\in\mathcal{J}_{[1,k],i}}\prod_{t=1}^{i}\eta_{j_t}^2\|\Pi_{J_{[1,k],i}^c}(B) B^i e_1\|^2, \quad \forall 1\leq i\leq \ell,\\
({\rm I}_\ell)^c&=(n-1)^{\ell+1} \sum_{J_{\ell+1}\in\mathcal{J}_{[1,k],\ell+1}}\prod_{t=1}^{\ell+1}\eta_{j_t}^2\E[\|\Pi_{J^c_{[j_{\ell+1}+1,k],\ell}}(B)B^{\ell+1} e_{j_{\ell+1}}\|^2].
\end{align*}
\end{corollary}

In view of Theorem \ref{thm:decomp}, the error $\E[\|e_{k+1}^\delta\|^2]$ can be decomposed into three components: approximation error $\sum_{i=0}^\ell {\rm I}_{i,1}^\delta$,
propagation error $\sum_{i=0}^\ell {\rm I}_{i,2}^\delta$, and stochastic error $({\rm I}_\ell^\delta)^c$.
Here we have slightly abused the terminology for approximation and propagation errors, since the approximation error only depends on the
regularity of the exact solution $x^\dag$ (indicated by the source condition \eqref{eqn:source} in Assumption \ref{ass:stepsize}(ii)), whereas the
propagation error is determined by the noise level. With the choice $\ell=0$, the decomposition recovers that in \cite{JinLu:2019,JinZhouZou:2020}.
When compared with the classical error decomposition for the Landweber method, the summands for $\ell\geq 1$ arise from the stochasticity
of the iterates (due to the random row index at each iteration), so is the stochastic error $({\rm I}_\ell^\delta)^c$. This refined decomposition
is crucial to analyze the saturation phenomenon (under suitable conditions on the initial stepsize). Below we first
derive bounds on the first two terms in Propositions \ref{err:apx} and \ref{err:ppg}, and then we prove optimal
convergence rates of SGD by mathematical induction in Section \ref{ssec:conv}.

\section{Convergence rate analysis}\label{sec:conv}

In this section, we present the convergence rate analysis, and establish Theorem \ref{thm:main}. The proof proceeds by first analyzing
the approximation error and propagation error in Sections \ref{ssec:approx} and \ref{ssec:propag}, respectively,
and then bound the mean squared error $\E[\|e_k^\delta\|^2]$ via mathematical induction. We also give an alternative (simplified)
convergence analysis for the case $\alpha=0$ in Section \ref{ssec:al=0}.

\subsection{Bound on the approximation error}\label{ssec:approx}

We begin with bounding the approximation error  $\sum_{i=0}^\ell {\rm I}_{i,1}^\delta$ for any fixed $\ell\geq \nu$.
The summand ${\rm I}_{0,1}^\delta$ is the usual approximation error (for Landweber method), and the
remaining terms arise from the random row index. Thus, the approximation error decays
at the optimal rate.
\begin{proposition}\label{err:apx}
Let Assumption \ref{ass:stepsize} be fulfilled, and
\begin{equation*}
   h_0(k)=2(\nu+\ell)^2 n \phi(2\alpha)k^{-2(1-\alpha)+\max(1-2\alpha,0)}.
\end{equation*}
Then for any integer $\ell\geq \nu$, $\alpha\in[0,1)$ and $k\geq2\ell$, there holds
\begin{align*}
\sum_{i=0}^\ell{\rm I}_{i,1}^\delta\leq c_{\nu,\ell,\alpha,n} c_0^{-2\nu} k^{-2\nu(1-\alpha)} \|w\|^2,
\end{align*}
with the constant
\begin{align*}
c_{\nu,\ell,\alpha,n}=\left\{\begin{array}{ll}
    4(\nu+\ell)^{2\nu}, &\mbox{if } h_0(k)\leq\frac12,\\
    2(\nu+\ell)^{2\nu}\sum_{i=0}^\ell(h_0(2\ell))^i, & otherwise.
  \end{array}\right.
\end{align*}
\end{proposition}
\begin{proof}
In view of the source condition \eqref{eqn:source} and Lemma \ref{lem:kernel}, we have
\begin{align*}
{\rm I}_{0,1}^\delta&=2\|\Pi_{J^c_{[1,k],0}}(B)e_1\|^2\leq2\|\Pi_{J^c_{[1,k],0}}(B)B^\nu\|^2\|w\|^2
\leq 2\nu^{2\nu} (ec_0)^{-2\nu} k^{-2\nu(1-\alpha)}\|w\|^2.
\end{align*}
Similarly, for any $1\leq i\leq \ell$,
\begin{align*}
  \prod_{t=1}^{i}\eta_{j_t}^2\|\Pi_{J_{[1,k],i}^c}(B) B^i e_1\|^2
&\leq \prod_{t=1}^{i}\eta_{j_t}^2\|\Pi_{J_{[1,k],i}^c}(B) B^{\nu+i}\|^2\|w\|^2\\
&\leq (\tfrac{\nu+i}{e})^{2(\nu+i)}c_0^{-2\nu}k^{2(\nu+i)\alpha}\|w\|^2\prod_{t=1}^{i}j_t^{-2\alpha}(k-i)^{-2(\nu+i)}.
\end{align*}
By the definition of ${\rm I}_{i,1}^\delta$, since $k\geq2\ell$, $k-i\geq \frac{k}{2}$, for $i=1,\ldots,\ell$, by Lemma \ref{lem:kernel2}(i),
\begin{align*}
{\rm I}_{i,1}^\delta
&\leq 2^{i+1}n^i (\tfrac{\nu+i}{e})^{2(\nu+i)}c_0^{-2\nu}k^{2(\nu+i)\alpha}\|w\|^2 \sum_{J_i\in\mathcal{J}_{[1,k],i}}\prod_{t=1}^{i}j_t^{-2\alpha}(k-i)^{-2(\nu+i)}\\
&\leq 2(2e^{-1})^{2\nu}(\nu+i)^{2\nu}c_0^{-2\nu} k^{-2\nu(1-\alpha)}\|w\|^2 \Big[8e^{-2}(\nu+i)^2 n \phi(2\alpha)k^{-2(1-\alpha)+\max(1-2\alpha,0)}\Big]^i.
\end{align*}
Clearly, the quantity in the square bracket is bounded by $h_0(k)$.
Next we treat the two cases $h_0(k)\leq \frac12$ and $h_0(k)>\frac12$ separately.
{ If $h_0(k)\leq\frac12$, we deduce}
\begin{align*}
\sum_{i=0}^\ell{\rm I}_{i,1}^\delta&\leq
2(\nu+\ell)^{2\nu}c_0^{-2\nu}k^{-2\nu(1-\alpha)}\|w\|^2
\sum_{i=0}^\ell h_0(k)^i\leq
c_{\nu,\ell,\alpha,n}c_0^{-2\nu}k^{-2\nu(1-\alpha)}\|w\|^2.
\end{align*}
Further, when $h_0(k)>\frac12$, since $k\geq2\ell$, we have $h_0(k)\leq h_0(2\ell)$, and thus  obtain
\begin{align*}
\sum_{i=0}^\ell{\rm I}_{i,1}^\delta&\leq2(\nu+\ell)^{2\nu}c_0^{-2\nu}k^{-2\nu(1-\alpha)}\|w\|^2
\sum_{i=0}^\ell h_0(2\ell)^i \leq{c}_{\nu,\ell,\alpha,n}c_0^{-2\nu}k^{-2\nu(1-\alpha)}\|w\|^2.
\end{align*}
Finally, combining the last two estimates completes the proof.
\end{proof}

\begin{remark}\label{rem:apx1}
For any $k$ satisfying $h_0(k)\leq \frac12$, the constant $c_{\nu,\ell,\alpha,n}$ is actually
independent of $\alpha$ and $n$. Further, if $k<2\ell$, then by setting $\ell$ to $0$, we obtain
\begin{equation*}
   {\rm I_{0,1}^\delta}\leq 2^{1-2\nu}\nu^{2\nu} c_0^{-2\nu} k^{-2\nu(1-\alpha)}\|w\|^2.
\end{equation*}
\end{remark}

\subsection{Bound on the propagation error}\label{ssec:propag}

Now we bound the propagation error $\sum_{i=0}^\ell {\rm I}_{i,2}^\delta$, which
arises from the presence of the data noise $\xi$. The summands for $\ell\geq 1$
arise from the stochasticity of the SGD iterates $x_k^\delta$. We bound
each summand ${\rm I}_{i,2}^\delta$, $i=0,\ldots,\ell$, separately,
equivalently the following two quantities for $k\geq 4i$:
\begin{align}
  {\rm I}(i,k) &: = \sum_{J_i\in\mathcal{J}_{[2,k],i}}\prod_{t=1}^{i}\eta_{j_t}^2\Big\|\sum_{j_{i+1}=1}^{j_i-1}\eta_{j_{i+1}}\Pi_{J_{[j_{i+1}+1,k],i}^c}(B) B^{i+\frac12}\Big\|^2,\label{eqn:ppg1}\\
   {\rm II}(i,k) &: = \sum_{J_i\in\mathcal{J}_{[2,k],i}}\prod_{t=1}^{i}\eta_{j_t}^2\sum_{j_{i+1}=1}^{j_i-1}\eta_{j_{i+1}}^2\|\Pi_{J_{[j_{i+1}+1,k],i}^c}(B) B^{i+\frac12}\|^2,\label{eqn:ppg2}
\end{align}
with the convention $\sum_{J_0\in\mathcal{J}_{[2,k],0}}\prod_{t=1}^{0}\eta_{j_t}^2=1$ and $j_0=k+1$.
The condition $k\geq 4i$ implies the following two elementary estimates:
\begin{align}
   k-j_{i+1}-i\geq\tfrac k4,&\quad j_{i+1}=1,2,\cdots,[\tfrac k2],\label{eqn:bdd-k-j1}\\
   k-j_{i+1}\leq (i+1)(k-j_{i+1}-i), &\quad j_{i+1}=[\tfrac{k}{2}]+1,\ldots, k-i-1.\label{eqn:bdd-k-j2}
\end{align}

First we bound ${\rm I}(i,k)$. The notation $[\cdot]$ denotes taking
the integral part of a real number.
\begin{lemma}\label{lem:ppg1}
Let ${\rm I}(i,k)$ be defined in \eqref{eqn:ppg1}, and Assumption \ref{ass:stepsize} be fulfilled. Then for any
fixed $i\in\mathbb{N}$ and $k\geq 4i$, the following estimate holds
\begin{align*}
 {\rm I}(i,k) \leq& 2\Big(2^{2\alpha-1}e^{-1}(2i+1)c_0 \big((2e^{-1})^2(2i+1)^2\phi(2\alpha)k^{-2(1-\alpha)+\max(1-2\alpha,0)}\big)^{i}\\
    &+ 25c_0^{\max(i,1)}\big(2^{2\alpha-1}e^{-1}(i+2)^2k^{-\alpha}\big)^i\Big)\phi(\alpha)^2k^{1-\alpha}.
\end{align*}
\end{lemma}
\begin{proof}
We abbreviate ${\rm I}(i,k)$ as ${\rm I}$. By triangle inequality and Lemma \ref{lem:kernel}, for any $s\in (0,i+\frac12]$
{and any $j_{i+1}\neq k-i$ (when $j_{i+1}=k-i$, $j_i=k-i+1,\cdots,j_1=k$)}, we have
\begin{align*}
\Big\|\sum_{j_{i+1}=1}^{j_i-1}\eta_{j_{i+1}}\Pi_{J_{[j_{i+1}+1,k],i}^c}(B) B^s\Big\|
\leq&\sum_{j_{i+1}=1}^{j_i-1}\eta_{j_{i+1}}\|\Pi_{J_{[j_{i+1}+1,k],i}^c}(B) B^s\|\\
\leq&s^s (ec_0)^{-s}k^{\alpha s}\sum_{j_{i+1}=1}^{j_i-1}\eta_{j_{i+1}}(k-j_{i+1}-i)^{-s},
\end{align*}
By the identity \eqref{eqn:sum2}, and since the quantity $\big(\sum_{j_{i+1}=1}^{k-i-1}\eta_{j_{i+1}}s^s
(ec_0)^{-s}(k-j_{i+1}-i)^{-s}k^{\alpha s}\big)^2$ is independent of the indices $\{j_1,\cdots,j_{i}\}$,
there holds
\begin{align*}
{\rm I}^\frac12\leq&\big(\sum_{J_{i+1}\in\mathcal{J}_{[k-i,k],i+1}}\prod_{t=1}^{i+1}\eta_{j_t}^2\|B^{s}\|^2\big)^\frac12\\
  &+s^s(ec_0)^{-s}k^{\alpha s}\sum_{j_{i+1}=1}^{k-i-1}\eta_{j_{i+1}}(k-j_{i+1}-i)^{-s}\Big(\sum_{J_i\in\mathcal{J}_{[j_{i+1}+1,k],i}}\prod_{t=1}^{i}\eta_{j_t}^2\Big)^\frac12\\
=&c_0^{i+1}\prod_{j=k-i}^{k}j^{-\alpha}\|B^{s}\|
  +s^s(ec_0)^{-s}k^{\alpha s}\sum_{j_{i+1}=1}^{k-i-1}\eta_{j_{i+1}}(k-j_{i+1}-i)^{-s}\Big(\sum_{J_i\in\mathcal{J}_{[j_{i+1}+1,k],i}}\prod_{t=1}^{i}\eta_{j_t}^2\Big)^\frac12.
\end{align*}
The two terms on the right are denoted by ${\rm I}'_0$ and ${\rm I}'$. For $i\geq 1$, setting
$s=\frac{i}{2}+1\leq i+\frac12$ in the first term, the inequalities $k-i\geq\frac34 k$ and $c_0\|B\|\leq(2e)^{-1}$ imply that
\begin{align*}
{\rm I}'_0\leq c_0^{\frac i2}(2e)^{-\frac{i}{2}-1}(\tfrac43)^{(i+1)\alpha}k^{-(i+1)\alpha}\leq
e^{-1}\big(2^{2\alpha-1}e^{-1}c_0k^{-\alpha}\big)^{\frac{i}{2}}.
\end{align*}
Likewise, for $i=0$, setting $s=i+\frac12$ gives
${\rm I}'_0\leq (2e)^{-\frac12}c_0^{\frac12}k^{-\alpha}.$
Next we split ${\rm I}'$ into two summations ${\rm I}_1'$ and ${\rm I}_2'$ over the index $j_{i+1}$,
one from $1$ to $[\frac{k}{2}]$, and the other from $[\frac{k}{2}]+1$ to $k-i-1$, respectively. It
suffices to bound ${\rm I}_1'$ and ${\rm I}_2'$. First, setting $s$ to
$i+\frac12$ in ${\rm I}_1'$ and then applying the inequality %
\begin{equation*}
  \sum_{J_i\in\mathcal{J}_{[j_{i+1}+1,k],i}}\prod_{t=1}^{i}j_t^{-2\alpha} \leq \sum_{J_i\in\mathcal{J}_{[1,k],i}}\prod_{t=1}^{i}j_t^{-2\alpha},
\end{equation*}
and the estimate \eqref{eqn:bdd-k-j1} lead to
\begin{align*}
{\rm I}_1'\leq&(\tfrac{2i+1}{2e})^{i+\frac12}c_0^{\frac12}k^{(i+\frac12)\alpha}(\tfrac{k}{4})^{-(i+\frac12)}\Big(\sum_{j_{i+1}=1}^{[\frac k2]}j_{i+1}^{-\alpha}\Big)\Big(\sum_{J_i\in\mathcal{J}_{[1,k],i}}\prod_{t=1}^{i}j_t^{-2\alpha}\Big)^\frac12.
\end{align*}
Then by Lemma \ref{lem:kernel2}(i) and the estimate \eqref{eqn:basic-bdd}, we obtain
\begin{align*}
{\rm I}'_1
\leq&2^{\alpha-1}(2e^{-1})^\frac12(2i+1)^\frac12c_0^\frac{1}{2}\phi(\alpha) \big((2e^{-1})^2(2i+1)^2\phi(2\alpha)k^{-2(1-\alpha)+\max(1-2\alpha,0)}\big)^\frac{i}{2} k^{\frac{1-\alpha}{2}}.
\end{align*}
For the term ${\rm I}_2'$, we analyze the cases $i=0$ and $i\geq1$ separately. Since $c_0\|B\|\leq(2e)^{-1}$, cf. Assumption \ref{ass:stepsize}, if $i=0$, then, Lemma \ref{lem:kernel} with $s=\frac12$ gives
\begin{align*}
{\rm I}_2'
\leq (\tfrac{1}{2e})^\frac12 c_0^{\frac12}k^{\frac{\alpha}{2}}\sum_{j=[\frac k2]+1}^{k-1} j^{-\alpha} (k-j)^{-\frac12},
\end{align*}
Now the estimate \eqref{eqn:basic-bdd} implies
\begin{align*}
\sum_{j=[\frac k2]+1}^{k-1} j^{-\alpha} (k-j)^{-\frac12}
\leq (\tfrac{k}{2})^{-\alpha}2(\tfrac{k}{2})^{\frac12}\leq 2(\tfrac{k}{2})^{\frac12-\alpha}.
\end{align*}
Consequently, when $i=0$, we have
\begin{align*}
 {\rm I}_2' \leq 2^{\alpha}e^{-\frac 12} c_0^{\frac12} k^{\frac{1-\alpha}2} \quad \mbox{and} \quad {\rm I}_2'+{\rm I}_0'\leq 2^{\alpha+1}e^{-\frac 12} c_0^{\frac12} k^{\frac{1-\alpha}2}.
\end{align*}
Meanwhile, when $i\geq 1$, setting $s=\frac{i}{2}+1\leq i+\frac12$ in Lemma \ref{lem:kernel} gives
\begin{align*}
{\rm I}_2'
\leq&(\tfrac{\frac{i}{2}+1}{e})^{\frac{i}{2}+1} c_0^\frac{i}{2}k^{\alpha (\frac{i}{2}+1)}(\tfrac k2)^{-(i+1)\alpha}{\rm I}_2'',
\end{align*}
with
\begin{align*}
 {\rm I}_2'' :=\sum_{j_{i+1}=[\frac k2]+1}^{k-i-1}(k-j_{i+1}-i)^{-{(\frac{i}{2}+1)}}\Big(\sum_{J_i\in\mathcal{J}_{[j_{i+1}+1,k],i}}\prod_{t=1}^{i}1\Big)^\frac12.
\end{align*}
Now Lemma \ref{lem:kernel2}(ii), and the estimates \eqref{eqn:bdd-k-j2} and \eqref{eqn:basic-bdd} yield
\begin{align*}
{\rm I}_2'' &\leq \frac{1}{i!^\frac12}\sum_{j_{i+1}=[\frac k2]+1}^{k-i-1}(k-j_{i+1}-i)^{-(\frac{i}{2}+1)}(k-j_{i+1})^\frac{i}{2}\\
&\leq \frac{(i+1)^\frac{i}{2}}{i!^{\frac12}}\sum_{j_{i+1}=[\frac k2]+1}^{k-i-1}(k-j_{i+1}-i)^{-1}\leq \frac{2(i+1)^\frac{i}{2}}{i!^{\frac12}}\max(\ln k,1).
\end{align*}
Combining the last two identities gives
\begin{align*}
{\rm I}'_2 &\leq 2^{\alpha}(i+2){i!^{-\frac12}}e^{-1} \max(\ln k,1)\big(2^{2\alpha-1}e^{-1}c_0(i+2)^2k^{-\alpha}\big)^\frac{i}{2},\quad i\geq 1.
\end{align*}
Now by the estimate $\sup_{i\in\mathbb{N}}\frac{i+2}{i!^\frac{1}{2}}\leq 3$ and the elementary inequality (for { $s\in (0,1]$})
\begin{equation}\label{eqn:log-exp}
  k^{-s} \max(\ln k,1) \leq s^{-1},
\end{equation}
with $s=\frac{1-\alpha}{2}$, we obtain
\begin{align*}
{\rm I}'_2\leq
12e^{-1}\phi(\alpha)\big(2^{2\alpha-1}e^{-1}c_0(i+2)^2k^{-\alpha}\big)^\frac{i}{2}k^\frac{1-\alpha}{2},\quad i\geq 1,
\end{align*}
and thus for $i\geq 1$, there holds
$${\rm I}_2'+{\rm I}'_0\leq 13e^{-1}\phi(\alpha)\big(2^{2\alpha-1}e^{-1}c_0(i+2)^2k^{-\alpha}\big)^\frac{i}{2}k^\frac{1-\alpha}{2}\leq 5\phi(\alpha)\big(2^{2\alpha-1}e^{-1}c_0(i+2)^2k^{-\alpha}\big)^\frac{i}{2}k^\frac{1-\alpha}{2}.$$
The bounds on ${\rm I}_1'$ and ${\rm I}_2'+{\rm I}'_0$ and the triangle inequality
complete the proof.
\end{proof}

The next result bounds the quantity ${\rm II}(i,k)$.
\begin{lemma}\label{lem:ppg2}
Let ${\rm II}(i,k)$ be defined in \eqref{eqn:ppg2}, and Assumption \ref{ass:stepsize}
hold. Then for any fixed $i\in\mathbb{N}$, and $k\geq 4i$, the following estimate holds
\begin{align*}
{\rm II}(i,k) \leq &\Big(\tfrac{ec_0}{2(2i+1)}(4e^{-2}(2i+1)^2\phi(2\alpha)k^{-2(1-\alpha)+\max(1-2\alpha,0)})^{i+1}\\
  &+3\phi(\alpha)\big(2^{2\alpha-1}e^{-1}c_0(i+1)^2k^{-\alpha}\big)^{i+1}\Big) k^{1-\alpha}.
\end{align*}
\end{lemma}
\begin{proof}
Like before, we abbreviate ${\rm II}(i,k)$ to ${\rm II}$.
By \eqref{eqn:sum2}, ${\rm II}$ can be rewritten as
\begin{align*}
{\rm II} = \sum_{j_{i+1}=1}^{k-i}\sum_{J_i\in\mathcal{J}_{[j_{i+1}+1,k],i}}\prod_{t=1}^{i}\eta_{j_t}^2\eta_{j_{i+1}}^2\|\Pi_{J_{[j_{i+1}+1,k],i}^c}(B) B^{i+\frac12}\|^2.
\end{align*}
Now we split the summation into three terms, i.e., $j_{i+1}=k-i$, one from $j_{i+1}=1$ to $[\frac{k}{2}]$
and one from $j_{i+1}=[\frac{k}{2}]+1$ to $k-i-1$, denoted by ${\rm II}_0$, ${\rm II}_1$ and ${\rm II}_2$,
respectively. Since $k-i\geq\frac34 k$, $\|B\|\leq 1$ and $c_0\|B\|\leq(2e)^{-1}$, cf.
Assumption \ref{ass:stepsize}(i), we obtain that, for any $i\geq 0$,
\begin{align*}
{\rm II_0}=c_0^{2i+2}\prod_{j=k-i}^{k}j^{-2\alpha}\|B^{i+\frac12}\|^2\leq c_0^{i+1} (c_0\|B\|)^{i+1}(k-i)^{-2(i+1)\alpha}\leq \big( 2^{2\alpha-1}e^{-1} c_0k^{-2\alpha}\big)^{i+1}.
\end{align*}
By Lemma \ref{lem:kernel} with $s=i+\frac12$ and \eqref{eqn:bdd-k-j1},
\begin{align*}
{\rm II_1} \leq&(\tfrac{2i+1}{2e})^{2i+1} c_0 k^{(2i+1)\alpha}\sum_{j_{i+1}=1}^{[\frac{k}{2}]}(k-j_{i+1}-i)^{-(2i+1)}\sum_{J_i\in\mathcal{J}_{[j_{i+1}+1,k],i}}\prod_{t=1}^{i+1} j_t^{-2\alpha}\\
\leq & (\tfrac{2i+1}{2e})^{2i+1} c_0 k^{(2i+1)\alpha}(\tfrac{k}{4})^{-(2i+1)}\sum_{j_{i+1}=1}^{[\frac{k}{2}]}\sum_{J_i\in\mathcal{J}_{[j_{i+1}+1,k],i}}\prod_{t=1}^{i+1} j_t^{-2\alpha}
\end{align*}
Meanwhile, Lemma \ref{lem:kernel2} and the estimate \eqref{eqn:basic-bdd} imply
\begin{align*}
\sum_{j_{i+1}=1}^{[\frac{k}{2}]}\sum_{J_i\in\mathcal{J}_{[j_{i+1}+1,k],i}}\prod_{t=1}^{i+1} j_t^{-2\alpha}
\leq (\phi(2\alpha)k^{\max(1-2\alpha,0)})^{i+1}.
\end{align*}
The last two estimates together imply
\begin{align*}
{\rm II}_1\leq
&\tfrac{ec_0}{2(2i+1)}(4e^{-2}(2i+1)^2\phi(2\alpha)k^{-2(1-\alpha)+\max(1-2\alpha,0)})^{i+1}k^{1-\alpha}.
\end{align*}
Now we bound the term ${\rm II}_2$. In this case, we analyze the cases $i=0$ and $i\geq1$ separately. { When $i=0$,
by Lemma \ref{lem:kernel},}
\begin{align*}
{\rm II_2}
=&\sum_{j=[\frac k2]+1}^{k-1}\eta_j^2\|\Pi_{J_{[j+1,k],0}^c}(B) B^{\frac12}\|^2
\leq \frac{c_0k^\alpha}{2e}\sum_{j=[\frac k2]+1}^{k-1}j^{-2\alpha}(k-j)^{-1}.
\end{align*}
The estimates \eqref{eqn:basic-bdd} and \eqref{eqn:log-exp} with $s=1-\alpha$ imply
\begin{align*}
\sum_{j=[\frac k2]+1}^{k-1}j^{-2\alpha}(k-j)^{-1} \leq 2(\tfrac{k}{2})^{-2\alpha}\max(\ln k,1) \leq  2^{2\alpha+1}\phi(\alpha)k^{1-3\alpha}.
\end{align*}
The last two estimates together show that for $i=0$, there holds
\begin{align*}
{\rm II}_2\leq  2^{2\alpha}e^{-1}c_0\phi(\alpha) k^{1-2\alpha}.
\end{align*}
Next, when $i\geq 1$, by Lemma \ref{lem:kernel} with $s=\frac{i+1}{2}$,
\begin{align*}
{\rm II}_2\leq& \sum_{j_{i+1}=[\frac{k}{2}]+1}^{k-i-1}\sum_{J_i\in\mathcal{J}_{[j_{i+1}+1,k],i}}\prod_{t=1}^{i}\eta_{j_t}^2\eta_{j_{i+1}}^2\|\Pi_{J_{[j_{i+1}+1,k],i}^c}(B) B^{\frac{i+1}{2}}\|^2\\
\leq&(\tfrac{i+1}{2e})^{i+1}k^{(i+1)\alpha}c_0^{i+1}(\tfrac{k}{2})^{-2(i+1)\alpha}\sum_{j_{i+1}=[\frac{k}{2}]+1}^{k-i-1}\sum_{J_i\in\mathcal{J}_{[j_{i+1}+1,k],i}}(k-j_{i+1}-i)^{-(i+1)}.
\end{align*}
Now Lemma \ref{lem:kernel2}(ii), and \eqref{eqn:basic-bdd} imply
\begin{align*}
\sum_{j_{i+1}=[\frac{k}{2}]+1}^{k-i-1}\sum_{J_i\in\mathcal{J}_{[j_{i+1}+1,k],i}}(k-&j_{i+1}-i)^{-(i+1)}
\leq\sum_{j_{i+1}=[\frac{k}{2}]+1}^{k-i-1}(k-j_{i+1}-i)^{-(i+1)}\frac{(k-j_{i+1})^i}{i!}\\
\leq& \frac{(i+1)^i}{i!}\sum_{j_{i+1}=[\frac{k}{2}]+1}^{k-i-1}(k-j_{i+1}-i)^{-1}\leq \frac{2 (i+1)^i}{i!}\max(\ln k,1).
\end{align*}
Combining the last two bounds with \eqref{eqn:log-exp} with $s=1-\alpha$ leads to
\begin{align*}
{\rm II}_2
\leq{\frac{2}{i!}}\phi(\alpha)\big(2^{2\alpha-1}e^{-1}c_0(i+1)^2k^{-\alpha}\big)^{i+1}k^{1-\alpha},\quad i\geq 1.
\end{align*}
Clearly, the preceding discussion shows that the last inequality holds actually also for $i=0$.
Therefore, the bounds on ${\rm II}_0$, ${\rm II}_1$ and ${\rm II}_2$ complete the proof of the lemma.
\end{proof}

Now we can bound the propagation error $\sum_{i=0}^\ell {\rm I}_{2,i}^\delta$.
The bound is largely comparable with that for the Landweber
method \cite[Theorem 3.2]{JinLu:2019}.
\begin{proposition}\label{err:ppg}
Let Assumption \ref{ass:stepsize} be fulfilled, and let
\begin{align*}
h_1(k)&=2(2\ell+1)^{2}n\phi(2\alpha) k^{-2(1-\alpha)+\max(1-2\alpha,0)}\quad\mbox{and}\quad
h_2(k)=2^{2\alpha-1} (\ell+2)^2nc_0 k^{-\alpha}.
\end{align*}
Then for any fixed $\ell\in\mathbb{N}$, and $k\geq 4\ell$, there holds
\begin{align*}
\sum_{i=0}^\ell{\rm I}_{i,2}^\delta\leq c_{\ell,\alpha,n,c_0}{\bar\delta}^2 k^{1-\alpha},
\end{align*}
with the constant $c_{\ell,\alpha,n,c_0}$ given by
\begin{align*}
  c_{\ell,\alpha,n,c_0}&=\left\{\begin{array}{ll}
    \displaystyle(2^{4}(\ell+1)c_0 + {203})\phi(\alpha)^2, &\mbox{if } h_1(k),h_2(k)\leq\frac12,\\
    \displaystyle\Big({ 8(\ell+1)c_0\sum_{i=0}^{\ell+1}h_1(4\ell)^{i}}+{103}\sum_{i=0}^{\ell+1}h_2(4\ell)^i\Big)\phi(\alpha)^2, &\mbox{otherwise},
  \end{array}\right.
\end{align*}
\end{proposition}
\begin{proof}
For $i=0,1,\cdots,\ell$, we bound the summands ${\rm I}_{i,2}^\delta$ by
\begin{align*}
{\rm I}_{i,2}^\delta
\leq&2^{i+1}n^i{\bar\delta}^2 \sum_{J_i\in\mathcal{J}_{[2,k],i}}\prod_{t=1}^{i}\eta_{j_t}^2\|\sum_{j_{i+1}=1}^{j_i-1}\eta_{j_{i+1}}\Pi_{J_{[j_{i+1}+1,k],i}^c}(B) B^{i+\frac12}\|^2\\
&+2^{i+1}n^{i+1}{\bar\delta}^2 \sum_{J_i\in\mathcal{J}_{[2,k],i}}\prod_{t=1}^{i}\eta_{j_t}^2\sum_{j_{i+1}=1}^{j_i-1}\eta_{j_{i+1}}^2\|\Pi_{J_{[j_{i+1}+1,k],i}^c}(B) B^{i+\frac12}\|^2.
\end{align*}
The two terms on the right hand side, denoted by ${\rm I}_{i,2,1}^\delta$ and
${\rm I}_{i,2,2}^\delta$, can be bounded using Lemmas \ref{lem:ppg1} and \ref{lem:ppg2}, respectively. Indeed, for
${\rm I}_{i,2,1}^\delta$, Lemma \ref{lem:ppg1} yields that for any $k\geq 4\ell\geq 4i$,
\begin{align*}
{\rm I}_{i,2,1}^\delta \leq& 4\Big(2^{2\alpha-1}e^{-1}(2i+1)c_0 \big(2^3e^{-2}(2i+1)^2n\phi(2\alpha){k^{-2(1-\alpha)+\max(1-2\alpha,0)}}\big)^{i}\\
&\quad +  25\big(2^{2\alpha}e^{-1}(i+2)^2nc_0k^{-\alpha}\big)^i\Big)\phi(\alpha)^2\bar{\delta}^2k^{1-\alpha}.
\end{align*}
Thus, by the definitions of $h_1(k),h_2(k)$, for any $k\geq4\ell$, if $h_1(k)\leq\frac12$ and
$h_2(k)\leq\frac12$, then the condition $i\leq \ell$ implies
\begin{align*}
  \sum_{i=0}^\ell {\rm I}_{i,2,1}^\delta
  &\leq 4 \Big(2^{2\alpha-1}e^{-1}(2\ell+1) c_0\sum_{i=0}^\ell h_1(k)^i + 25\sum_{i=0}^\ell  h_2(k)^i\Big)\phi(\alpha)^2\bar\delta^2k^{1-\alpha}\\
  & \leq 4\big(2^{2\alpha}e^{-1}(2\ell+1)c_0+50\big)\phi(\alpha)^2\bar\delta^2 k^{1-\alpha}.
\end{align*}
Meanwhile, if $k$ does not satisfy the condition, by the monotonicity of $h_1(k)$ and $h_2(k)$
in $k$, we have $h_1(k)\leq h_1(4\ell)$ and $h_2(k)\leq h_2(4\ell)$, and consequently,
\begin{align*}
\sum_{i=0}^\ell {\rm I}_{i,2,1}^\delta
& \leq 4\Big(2^{2\alpha-1}e^{-1}(2\ell+1)c_0\sum_{i=0}^\ell h_1(4\ell)^{i}+25\sum_{i=0}^\ell h_2(4\ell)^{i}\Big)\phi(\alpha)^{2}{\bar\delta}^2 k^{1-\alpha}.
\end{align*}
Next we bound the term ${\rm I}_{i,2,2}^\delta$. Actually, by Lemma \ref{lem:ppg2}, for any $k\geq 4\ell\geq 4i$, there holds
\begin{align*}
 {\rm I}_{i,2,2}^\delta\leq &
   \Big(\tfrac{ec_0}{2(2i+1)}(2^3e^{-2}(2i+1)^2n\phi(2\alpha){k^{-2(1-\alpha)+\max(1-2\alpha,0)}})^{i+1}\\
     &+3\phi(\alpha)(2^{2\alpha}e^{-1}(i+1)^2nc_0k^{-\alpha})^{i+1}\Big)\bar\delta^2 k^{1-\alpha}.
\end{align*}
Then repeating the preceding arguments yield
\begin{align*}
 \sum_{i=0}^\ell{\rm I}_{i,2,2}^\delta&\leq \left\{\begin{aligned}
   \big(2c_0+3\phi(\alpha)\big){\bar\delta}^2 k^{1-\alpha}, &\quad \mbox{if } h_1(k),h_2(k)\leq\tfrac{1}{2},\\
   \Big(2c_0\sum_{i=0}^\ell h_1(4\ell)^{i+1}+3\phi(\alpha)\sum_{i=0}^\ell\big( h_2(4\ell)\big)^{i+1}\Big){\bar\delta}^2 k^{1-\alpha}, &\quad \mbox{otherwise}.
   \end{aligned}\right.
\end{align*}
Now combining the bounds on $\sum_{i=0}^\ell{\rm I}_{i,2,1}^\delta$ and $\sum_{i=0}^\ell{\rm I}_{i,2,2}^\delta$ yields
the desired assertion.
\end{proof}
\begin{remark}\label{rem:ppg1}
If $k<4\ell$, we can replace $\ell$ by $0$. By Assumption \ref{ass:stepsize}(i), $c_0<1$,
repeating the argument of the proposition and Lemmas \ref{lem:ppg1} and \ref{lem:ppg2} yields
{ \begin{align*}
{\rm I_{0,2}^\delta}\leq 2c_0\big(n(\phi(2\alpha)+3\phi(\alpha))+11\phi(\alpha)^2\big)\bar{\delta}^2 k^{1-\alpha}.
\end{align*}}
Note that in the conditions $h_0(k),h_1(k),h_2(k)\leq\frac12$, $h_0$ and $h_1$, apart from the factor $\phi(2\alpha)$,
do not depend sensitively on the exponent $\alpha,$ but for $\alpha$ close to zero, $h_2(k)\leq\frac12$ essentially
requires a small $c_0=O(n^{-1})$, and further the larger $\ell$ is, and the smaller $c_0$ should be in order
to fulfill the conditions. The latter condition also appears in the proof of Theorem \ref{thm:main} below.
\end{remark}

\subsection{Bound on the error $\E[\|e_k^\delta\|^2]$}\label{ssec:conv}

To prove Theorem \ref{thm:main}, we need a useful technical estimate,
where the notation $k^{\max(0,0)}$ denotes $ \ln k$. Note the restricted
range of $s$ is sufficient for the proof of Theorem \ref{thm:main}.
\begin{lemma}\label{lem:bdd-gen}
Let Assumption \ref{ass:stepsize} hold. Then for any $\epsilon,\eta\in[0,1]$, { $s\in (-\infty,0]\cup (\max(0,1-2\alpha),+\infty)$}, and $k\geq 4\ell$
with $\ell_\epsilon=\epsilon(\ell+1)$ and $\ell_\eta=\eta(\ell+1)$, the following two estimates hold:
\begin{align}
&
{\sum_{J_{\ell+1}\in \mathcal J_{[k-\ell,k],\ell+1}} \prod_{t=1}^{\ell+1} \eta_{j_t}^2\|B^{\ell+1}\|^2(k-\ell)^{-s}}\label{eqn:bdd-err-s0}\\
  \leq&{\max(2^{s},1)} (2^{2\alpha-2\eta}e^{-2\eta}c_0^{2-2\eta}k^{-2\alpha})^{\ell+1} k^{-s},\nonumber\\
  &\sum_{j_{\ell+1}=1}^{[\frac{k}{2}]}\sum_{J_{\ell}\in\mathcal{J}_{[j_{\ell+1}+1,k],\ell}}\prod_{t=1}^{\ell+1}\eta_{j_t}^2\|\Pi_{J^c_{[j_{\ell+1}+1,k],\ell}}(B) B^{\ell+1}\|^2j_{\ell+1}^{-s}\label{eqn:bdd-err-s1}\\
  \leq & c_s\Big(\big(4e^{-1}\ell_\epsilon\big)^{2\epsilon}\phi(2\alpha)c_0^{2-2\epsilon}{k^{-s_{\epsilon}(s)}}\Big)^{\ell+1}k^{-s},\nonumber\\
  &\sum_{j_{\ell+1}=[\frac{k}{2}]+1}^{k-\ell-1}\sum_{J_{\ell}\in\mathcal{J}_{[j_{\ell+1}+1,k],\ell}}\prod_{t=1}^{\ell+1}\eta_{j_t}^2\|\Pi_{J^c_{[j_{\ell+1}+1,k],\ell}}(B) B^{\ell+1}\|^2j_{\ell+1}^{-s}\label{eqn:bdd-err-s2}\\
  \leq & \tfrac{\max(2^s,1)\phi(2\ell_\eta-\ell)}{(\ell+1)!}(2^{2\alpha}e^{-2\eta}(\ell+1)\ell_\eta^{2\eta}c_0^{2-2\eta}
k^{-s_\eta})^{\ell+1}k^{-s},\nonumber
\end{align}
with the constant $c_s$, the exponents $s_{\epsilon}(s)$ and $s_\eta$ respectively defined by
\begin{align*}
c_s &= \left\{\begin{aligned}
    2^{s}, &\quad \mbox{if }s\leq 0,\\
    \tfrac{\phi(2\alpha+s)}{\phi(2\alpha)}, &\quad \mbox{if } s>\max(0,1-2\alpha),
  \end{aligned}\right.\\
s_{\epsilon}(s)&=2\epsilon(1-\alpha)-\max(1-2\alpha,0)-(\ell+1)^{-1}\max\big(s-\max(1-2\alpha,0),0\big),\\
s_\eta&=(2-2\eta)\alpha-\max(1-2\eta,0).
\end{align*}
\end{lemma}
\begin{proof}
The proof is similar to that of Lemma \ref{lem:ppg2}. We denote the three terms on the left
hand side by ${\rm I}_0$, ${\rm I}_1$ and ${\rm I}_2$, respectively. It is easy to check that, for any $\eta\in[0,1]$, with the
inequalities $c_0\|B\|\leq(2e)^{-1}$, $\|B\|\leq 1$, cf. Assumption \ref{ass:stepsize}(i), and $k-\ell\geq\frac34 k$,
\begin{align*}
{\rm I}_0=\sum_{J_{\ell+1}\in \mathcal J_{[k-\ell,k],\ell+1}} \prod_{t=1}^{\ell+1} \eta_{j_t}^2\|B^{\ell+1}\|^2(k-\ell)^{-s}\leq (2^{2\alpha-2\eta}e^{-2\eta}c_0^{2-2\eta}k^{-2\alpha})^{\ell+1}\max(2^{s},1) k^{-s}.
\end{align*}
For ${\rm I}_1$, by the condition $\|B\|\leq 1$ in Assumption
\ref{ass:stepsize} and Lemma \ref{lem:kernel} with $s=\ell_\epsilon$, we have
\begin{align*}
{\rm I_{1}} \leq
  &\sum_{j_{\ell+1}=1}^{[\frac{k}{2}]}\sum_{J_{\ell}\in\mathcal{J}_{[j_{\ell+1}+1,k],\ell}}\prod_{t=1}^{\ell+1}\eta_{j_t}^2\|\Pi_{J^c_{[j_{\ell+1}+1,k],\ell}}(B) B^{\ell_\epsilon}\|^2j_{\ell+1}^{-s}\\
\leq& (\tfrac{\ell_{\epsilon}}{e})^{2\ell_\epsilon}c_0^{2(1-\epsilon)(\ell+1)}
k^{2\ell_{\epsilon}\alpha}\sum_{j_{\ell+1}=1}^{[\frac{k}{2}]}\sum_{J_{\ell}\in\mathcal{J}_{[j_{\ell+1}+1,k],\ell}}
\prod_{t=1}^{\ell+1}j_t^{-2\alpha}(k-j_{\ell+1}-\ell)^{-2\ell_{\epsilon}} j_{\ell+1}^{-s}.
\end{align*}
Now by the estimates \eqref{eqn:bdd-k-j1} and \eqref{eqn:basic-bdd}, and Lemma \ref{lem:kernel2}(i),
\begin{align*}
   &\sum_{j_{\ell+1}=1}^{[\frac{k}{2}]}\sum_{J_{\ell}\in\mathcal{J}_{[j_{\ell+1}+1,k],\ell}}
\prod_{t=1}^{\ell+1}j_t^{-2\alpha}(k-j_{\ell+1}-\ell)^{-2\ell_{\epsilon}} j_{\ell+1}^{-s}\\
  \leq &(\tfrac{k}{4})^{-2 \ell_{\epsilon}}\Big(\sum_{j_{\ell+1}=1}^{[\frac{k}{2}]}j_{\ell+1}^{-(2\alpha+s)}\Big)\Big(\sum_{J_{\ell}\in\mathcal{J}_{[1,k],\ell}}\prod_{t=1}^{\ell}j_t^{-2\alpha}\Big)\\
\leq&(\tfrac{k}{4})^{-2\ell_\epsilon} \Big(\sum_{j_{\ell+1}=1}^{[\frac{k}{2}]}j_{\ell+1}^{-(2\alpha+s)}\Big) \big(\phi(2\alpha){k^{\max(1-2\alpha,0)}}\big)^\ell.
\end{align*}
Direct computation with \eqref{eqn:basic-bdd} gives
\begin{equation*}
  \sum_{j_{\ell+1}=1}^{[\frac{k}{2}]}j_{\ell+1}^{-(2\alpha+s)} \leq \left\{\begin{aligned}
    2^{s}\phi(2\alpha) (\tfrac{k}{2})^{\max(1-2\alpha,0)}k^{-s}, &\quad s\leq0,\\
    {\tfrac{\phi(2\alpha+s)}{\phi(2\alpha)}(\phi(2\alpha)k^{\max(1-2\alpha,0)})k^{s-\max(1-2\alpha,0)}k^{-s}}, &\quad s>\max(0, 1-2\alpha).
  \end{aligned}\right.
\end{equation*}
These two estimates together give \eqref{eqn:bdd-err-s1}.
Similarly, Lemma \ref{lem:kernel} with $s=\ell_\eta$ yields
\begin{align*}
{\rm I_{2}}\leq&(\tfrac{\ell_{\eta}}{e})^{2\ell_{\eta}}c_0^{2(1-\eta)(\ell+1)} k^{2\ell_\eta\alpha}
\sum_{j_{\ell+1}=[\frac{k}{2}]+1}^{k-\ell-1}\sum_{J_{\ell}\in\mathcal{J}_{[j_{\ell+1}+1,k],\ell}}\prod_{t=1}^{\ell+1}j_t^{-2\alpha}(k-j_{\ell+1}-\ell)^{-2\ell_\eta}
 j_{\ell+1}^{-s}.
\end{align*}
Note that for $j_{\ell+1}=[\frac{k}{2}]+1,\ldots,{k-\ell-1}$, $j_{\ell+1}^{-s}\leq \max(1,2^s)k^{-s}$, and thus
Lemma \ref{lem:kernel2}(ii), and the estimates \eqref{eqn:basic-bdd} and \eqref{eqn:bdd-k-j2} give
\begin{align*}
&\sum_{j_{\ell+1}=[\frac{k}{2}]+1}^{k-\ell-1}\sum_{J_{\ell}\in\mathcal{J}_{[j_{\ell+1}+1,k],\ell}}\prod_{t=1}^{\ell+1}j_t^{-2\alpha}(k-j_{\ell+1}-\ell)^{-2\ell_\eta}
 j_{\ell+1}^{-s}\\
  \leq&\max(1,2^s) (\tfrac{k}{2})^{-2(\ell+1)\alpha} k^{-s} \sum_{j_{\ell+1}=[\frac{k}{2}]+1}^{k-\ell-1}\Big(\sum_{J_{\ell}\in\mathcal{J}_{[j_{\ell+1}+1,k],\ell}}1\Big)(k-j_{\ell+1}-\ell)^{-2\ell_{\eta}}\\
\leq&\max(1,2^s)\tfrac{(\ell+1)^\ell}{\ell!}(\tfrac{k}{2})^{-2(\ell+1)\alpha}k^{-s}
\sum_{j_{\ell+1}=[\frac{k}{2}]+1}^{k-\ell-1}(k-j_{\ell+1}-\ell)^{-2\ell_\eta+\ell}\\
\leq& \max(1,2^s)\tfrac{(\ell+1)^\ell\phi(2\ell_\eta-\ell)}{\ell!}(\tfrac{k}{2})^{-2(\ell+1)\alpha}k^{-s}k^{\max((1-2\eta)(\ell+1),0)},
\end{align*}
where the last line follows from \eqref{eqn:basic-bdd} and the identity $-2\ell_\eta+\ell +1 = (1-2\eta)(\ell+1)$.
Combining the preceding estimates yields the bound \eqref{eqn:bdd-err-s2}, and completes the proof of the lemma.
\end{proof}

Now, we can prove the order-optimal convergence rate of SGD in Theorem \ref{thm:main}.
\begin{proof}[Proof of Theorem \ref{thm:main}]
 Let $r_k=\E[\|e_k^\delta\|^2]$. We prove that for any $\epsilon\in(\frac12,1]$, there
exist $c_*$ and $c_{**}$ such that
\begin{align}\label{claim:k-le-k0}
 r_k\leq c_*k^{-\beta }+c_{**}\bar{\delta}^2 k^{\gamma},
\end{align}
with { $\beta= \min(2\nu,1+(2\epsilon-1)(\ell+1))(1-\alpha)$} and $\gamma=1-\alpha$.
Then the desired assertion holds by choosing { $\epsilon\in (\frac12,1)$} and $\ell\in\mathbb{N}$
such that { $(2\epsilon-1)(\ell+1)\geq 2\nu-1$}. The proof proceeds by
mathematical induction. { We treat the cases (i) $\alpha\in[0,\frac12)\cup(\frac12,1)$ and  (ii) $\alpha=\frac12$ separately.}
First we consider case (i). If { $k\leq 4\ell$}, the estimate \eqref{claim:k-le-k0} holds for
any sufficiently large $c_*$ and $c_{**}$. Assume that it holds up to some
$k\geq 4\ell$, and we prove it for $k+1$. It follows from Theorem \ref{thm:decomp} and
Propositions \ref{err:apx} and \ref{err:ppg} that
\begin{align*}
r_{k+1} \leq &c_{\nu,\ell,\alpha,n} c_0^{-2\nu} k^{-2\nu(1-\alpha)} \|w\|^2+c_{\ell,\alpha,n,c_0}{\bar\delta}^2 k^{1-\alpha}\\
&+(2n)^{\ell+1} \sum_{J_{\ell+1}\in\mathcal{J}_{[1,k],\ell+1}}\prod_{t=1}^{\ell+1}\eta_{j_t}^2\|\Pi_{J^c_{[j_{\ell+1}+1,k],\ell}}(B) B^{\ell+1}\|^2 r_{j_{\ell+1}}.
\end{align*}
Applying the induction hypothesis $r_j\leq c_*j^{-\beta}+c_{**}\bar{\delta}^2 j^{\gamma}$, $j=1,2,\cdots,k$, to the recursion gives
\begin{align*}
r_{k+1}\leq&c_{\nu,\ell,\alpha,n} c_0^{-2\nu} k^{-2\nu(1-\alpha)} \|w\|^2+ c_{\ell,\alpha,n,c_0}{\bar\delta}^2 k^{1-\alpha}+(2n)^{\ell+1}c_* {\rm I} +(2n)^{\ell+1}c_{**} \bar\delta^2{\rm II},
\end{align*}
with
\begin{align*}
{\rm I}: =& \sum_{J_{\ell+1}\in\mathcal{J}_{[1,k],\ell+1}}\prod_{t=1}^{\ell+1}\eta_{j_t}^2\|\Pi_{J^c_{[j_{\ell+1}+1,k],\ell}}(B) B^{\ell+1}\|^2  j_{\ell+1}^{-\beta},\\
{\rm II}:= &\sum_{J_{\ell+1}\in\mathcal{J}_{[1,k],\ell+1}}\prod_{t=1}^{\ell+1}\eta_{j_t}^2\|\Pi_{J^c_{[j_{\ell+1}+1,k],\ell}}(B) B^{\ell+1}\|^2 j_{\ell+1}^{\gamma}.
\end{align*}
Using \eqref{eqn:sum2}, we split each of ${\rm I}$ and ${\rm II}$ into {three terms over
the index $j_{\ell+1}$, one for $j_{\ell+1}=k-\ell$, one from $1$ to
$[\frac{k}{2}]$, and one from $[\frac{k}{2}]+1$ to $k-\ell-1$,} respectively. Now by Lemma \ref{lem:bdd-gen},
\begin{align}\label{eqn:one-step}
r_{k+1}
\leq&c_{\nu,\ell,\alpha,n} c_0^{-2\nu} k^{-2\nu(1-\alpha)} \|w\|^2+c_{\ell,\alpha,n,c_0}{\bar\delta}^2 k^{1-\alpha}\\
  &+c_*\xi_1(k)(k+1)^{-\beta}+c_{**}\bar\delta^2\xi_2(k)(k+1)^{\gamma},\nonumber
\end{align}
where the functions $\xi_1$ and $\xi_2$ are given by (for any { $\epsilon,\eta\in[\frac12,1]$, which implies $\frac12\leq\ell_\eta$})
\begin{align*}
\xi_1(k)=&\tfrac{2^\beta\phi(2\alpha+\beta)}{\phi(2\alpha)}(2^{1+2\epsilon}\ell_\epsilon^{2\epsilon} \phi(2\alpha)n c_0^{2-2\epsilon}k^{-s_{\epsilon}(\beta)})^{\ell+1}
+\tfrac{2^{2\beta}c_\eta}{(\ell+1)!}(2^{2\alpha+1}e^{-2\eta}(\ell+1)\ell_\eta^{2\eta}nc_0^{2-2\eta}
k^{-s_\eta})^{\ell+1},\\
\xi_2(k)=&(2^{1+2\epsilon}\ell_\epsilon^{2\epsilon} \phi(2\alpha)nc_0^{2-2\epsilon}k^{-s_{\epsilon}(-\gamma)})^{\ell+1}
+\tfrac{c_\eta}{(\ell+1)!}\big(2^{2\alpha+1}e^{-2\eta}(\ell+1)\ell_\eta^{2\eta} n c_0^{2-2\eta}k^{-s_\eta}\big)^{\ell+1},
\end{align*}
with the constants $c_\eta=1+\phi(2\ell_\eta-\ell)$,  $\ell_\epsilon$, $\ell_\eta$, { $s_{\epsilon}(\cdot)$},
and $s_\eta$ defined in Lemma \ref{lem:bdd-gen}. By choosing $1\geq\epsilon=\eta>\frac12$ and $\ell$ such that
{ $(2\epsilon-1)(\ell+1)\geq2\nu-1$}, we have { $s_\epsilon(\beta),s_\epsilon(-\gamma), s_\eta\geq 0$},
and
\begin{align*}
\xi_1(k)\leq c_1:=&\tfrac{2^\beta\phi(2\alpha+\beta)}{\phi(2\alpha)}(2^{1+2\epsilon}\ell_\epsilon^{2\epsilon} \phi(2\alpha)n c_0^{2-2\epsilon})^{\ell+1}
+{\tfrac{2^{2\beta}c_\eta}{(\ell+1)!}(2^{2\alpha+1}e^{-2\eta}(\ell+1)\ell_\eta^{2\eta}nc_0^{2-2\eta})^{\ell+1}},\\
\xi_2(k)\leq c_2:=&{(2^{1+2\epsilon}\ell_\epsilon^{2\epsilon} \phi(2\alpha)nc_0^{2-2\epsilon})^{\ell+1}}
+{\tfrac{c_\eta}{(\ell+1)!}\big(2^{2\alpha+1}e^{-2\eta}(\ell+1)\ell_\eta^{2\eta} n c_0^{2-2\eta}\big)^{\ell+1}}.
\end{align*}
For small $c_0$, $c_1,c_2\leq \frac12$ hold, and then \eqref{claim:k-le-k0} follows by setting
$c_*=2^{2\nu+1}c_0^{-2\nu}c_{\nu,\ell,\alpha,n}\|w\|^2$ and $c_{**}=2c_{\ell,\alpha,n,c_0}$. This proves
the theorem for case (i).
{ In case (ii), repeating the preceding argument, by choosing $1\geq\epsilon=\eta>\frac12$ and $\ell\geq 1$ such that
$(2\epsilon-1)(\ell+1)\geq2\nu-1$ gives
\begin{align*}
k^{-s_{\epsilon}(\beta)}&=k^{-\epsilon+\max(0,0)+(\ell+1)^{-1}\max(\beta-\max(0,0),0)}\leq k^{-\epsilon+(\ell+1)^{-1}\beta}\ln k\leq k^{-\frac14}\ln k\leq 4,\\
k^{-s_{\epsilon}(-\gamma)}&=k^{-\epsilon+\max(0,0)+(\ell+1)^{-1}\max(-\gamma-\max(0,0),0)}\leq  k^{-\epsilon}\ln k\leq \epsilon^{-1} \quad \mbox{and}\quad
k^{-s_\eta}\leq 1.
\end{align*}
Then repeating the preceding argument shows that the assertion holds when
\begin{align*}
\xi_1(k)\leq c_1:=&\tfrac{2^\beta\phi(2\alpha+\beta)}{\phi(2\alpha)}(2^{3+2\epsilon}\ell_\epsilon^{2\epsilon} \phi(2\alpha)n c_0^{2-2\epsilon})^{\ell+1}
+{\tfrac{2^{2\beta}c_\eta}{(\ell+1)!}(2^{2\alpha+1}e^{-2\eta}(\ell+1)\ell_\eta^{2\eta}nc_0^{2-2\eta})^{\ell+1}}\leq\tfrac12,\\
\xi_2(k)\leq c_2:=&(2^{1+2\epsilon}{\epsilon^{-1}}\ell_\epsilon^{2\epsilon} \phi(2\alpha)nc_0^{2-2\epsilon})^{\ell+1}
+{\tfrac{c_\eta}{(\ell+1)!}\big(2^{2\alpha+1}e^{-2\eta}(\ell+1)\ell_\eta^{2\eta} n c_0^{2-2\eta}\big)^{\ell+1}}\leq\tfrac12,
\end{align*}
which can be satisfied with sufficiently small $c_0$. This completes the proof of the theorem.}
\end{proof}

\begin{remark}\label{rem:c_0}
In practice, it is desirable to take small $\alpha$. With the choice $\alpha=0$ and
$\epsilon=\eta\in(\frac12,1)$, the proof requires that the initial stepsize $c_0$ satisfy
\begin{align*}
2^\beta\phi(\beta)(2^{1+2\epsilon}\ell_\epsilon^{2\epsilon} n c_0^{2-2\epsilon})^{\ell+1}
+{\tfrac{2^{2\beta}c_\eta}{(\ell+1)!}(2e^{-2\eta}(\ell+1)\ell_\eta^{2\eta}nc_0^{2-2\eta}
)^{\ell+1}}&\leq\tfrac12,\\
(2^{1+2\epsilon}\ell_\epsilon^{2\epsilon} nc_0^{2-2\epsilon})^{\ell+1}
+{\tfrac{c_\eta}{(\ell+1)!}\big(2e^{-2\eta}(\ell+1)\ell_\eta^{2\eta} n c_0^{2-2\eta}\big)^{\ell+1}}&\leq\tfrac12.
\end{align*}
These two conditions are fulfilled provided that {$nc_0^{2-2\epsilon}$ and $nc_0^{2-2\eta}$ are small constants.} In particular,
with $\epsilon=\eta$  close to $\frac12$, the conditions essentially amount to $c_0=O(n^{-1})$, agreeing with the
condition in Remark \ref{rem:ppg1}. Under this condition, by choosing an \textit{a priori} stopping index $k_*(\delta)
=O((\|w\|\delta^{-1})^\frac{2}{1+2\nu})$, we obtain the following bound
$\E[\|e_{k_*(\delta)}^\delta\|^2]^\frac12\le c\delta^\frac{2\nu}{1+2\nu}$.
This result is essentially identical with that for the Landweber method \cite[Chapter 6]{EnglHankeNeubauer:1996},
and higher than the existing convergence rate $O(\delta^\frac{\min(2\nu,1)}{\min(2\nu,1)+1})$ \cite{JinLu:2019} for
SGD. In particular, it proves that SGD with small initial stepsizes is actually order optimal.
\end{remark}

\begin{remark}\label{rem:ell}
One may slightly refine Theorem \ref{thm:main}. Indeed, for any $\alpha\in(0,1)$, let
\begin{align*}
H_1(k)&:=2^3\max(\nu,\ell+1)^2n\phi(2\alpha){k^{-(1-\alpha)+\max(1-2\alpha,0)}},\\
H_2(k)&:={2^{2\alpha-1} (\ell+2)^2} nc_0 k^{-\alpha}\ln k.
\end{align*}
Note that the following inequalities hold $h_0(k)\leq H_1(k)$, $ h_1(k)\leq H_1(k)$ and $h_2(k)\leq H_2(k)$.
Then there exists some $k_0$, dependent of $\alpha$, $n$, $\ell$, $\nu$ and $c_0$, such that
$H_1(k),H_2(k)\leq \tfrac12$ for any $k\geq k_0$.
The claim \eqref{claim:k-le-k0} shows that the assertion holds for any { $k\leq k_0$} with sufficiently large
$c_*$ and $c_{**}$. Then we refine the estimate by mathematical induction. Assume that the assertion
up to some $k\geq k_0$, and prove it for $k+1$. Since $k\geq k_0$, $h_i(k)\leq \frac12$, $i=1,2,3$, { it
follows from the estimate \eqref{eqn:one-step}
with} { $\beta= \min(2\nu,1+(2\epsilon-1)(\ell+1))(1-\alpha)$}, $\gamma=1-\alpha$, $\epsilon=1$ and $\eta=\frac12$ and Lemma \ref{lem:bdd-gen}, that
\begin{align*}
r_{k+1}\leq&{ c_{\nu,\ell,\alpha,n} c_0^{-2\nu} k^{-2\nu(1-\alpha)} \|w\|^2+
c_{\ell,\alpha,n,c_0}{\bar\delta}^2 k^{1-\alpha}}\\
 &+c_*\xi_1(k) (k+1)^{-\beta} + c_{**}\bar\delta^2 \xi_2(k) (k+1)^{1-\alpha},
\end{align*}
with the functions $\xi_1(k)$ and $\xi_2(k)$ given by
\begin{align*}
\xi_1(k)=&\tfrac{2^{\beta}\phi(2\alpha+\beta)}{\phi(2\alpha)}(2^{3}(\ell+1)^{2} \phi(2\alpha)n k^{{ -2(1-\alpha)+\max(1-2\alpha,0)}+(\ell+1)^{-1}(\beta-\max(1-2\alpha,0))})^{\ell+1}\\
 &+{\tfrac{3\cdot2^{2\beta}}{(\ell+1)!}(2^{2\alpha-1}(\ell+1)^2nc_0
k^{-\alpha}\ln k)^{\ell+1}},\\
\xi_2(k)=&(2^{3}(\ell+1)^2 \phi(2\alpha)n{ k^{-2(1-\alpha)+\max(1-2\alpha,0)}})^{\ell+1}
+{\tfrac{3}{(\ell+1)!}\big(2^{2\alpha-1}(\ell+1)^2n c_0k^{-\alpha}\ln k\big)^{\ell+1}}.
\end{align*}
Since $(\ell+1)^{-1}(\beta-\max(1-2\alpha,0))\leq 1-\alpha$, the terms on the right hand side can be
bounded by either $H_1(k)$ or $H_2(k)$, and thus for $k\geq k_0$, we have $H_1(k),H_2(k)\leq\frac12$, and consequently
\begin{align*}
\xi_1(k) \leq &{2^{-(\ell+1)}(\tfrac{2^\beta\phi(2\alpha+\beta)}{\phi(2\alpha)}+{\tfrac{3\cdot 2^{2\beta}}{(\ell+1)!})}}\quad\mbox{and}\quad
\xi_2(k)\leq 2^{-(\ell+1)}(1+{\tfrac{3}{(\ell+1)!}}).
\end{align*}
By choosing suitable {$\ell\geq 2\nu-2$} {\rm(}dependent of $\nu${\rm)}, we can ensure $\xi_1(k),\xi_2(k)\leq \frac12$, and
then taking the constants $c_*$ and $c_{**}$ as before yield the desired assertion.
\end{remark}

\subsection{Error analysis for $\alpha=0$}\label{ssec:al=0}
{ In this part, we revisit the case $\alpha=0$ separately, and
derive an error bound directly with more explicit constants.
\begin{lemma}\label{lem:0}
Let Assumption \ref{ass:stepsize}{\rm(i)} and {\rm(iii)} holds with $\alpha=0$.
Further, suppose that the following condition holds
\begin{align}\label{cond:0}
2(1+\phi(2\epsilon))nc_0^{2-2\epsilon}\leq1,\quad\mbox{for some } \epsilon\in(\tfrac12, 1).
\end{align}
Then for any $s\geq 0$, there holds
\begin{align*}
\E[\|(I-c_0B)^sB^{-\frac12} (Be_{k}^\delta-\bar{A}^t\bar{\xi})\|^2]\leq 2\|(I-c_0B)^{\frac{k-1}{2}+s} B^{-\frac12}(Be_{1}-\bar{A}^t\bar{\xi})\|^2.
\end{align*}
\end{lemma}
\begin{proof}
We prove the assertion by mathematical induction. When $k=1$, the inequality holds trivially
true for any $s\geq 0$. Now we assume that {it holds up to some $k-1\geq 1$, and prove it for $k$.}
With the condition $\alpha=0$, $\eta_j=c_0$ and $\Pi_{J^c_{[j,j'],0}} (B)=(I-c_0B)^{j'-j+1}$
for any $j'\geq j\geq1$. {By the definitions of $H_j$ and $N_j^\delta$, we can rewrite
$H_j$ as $H_j=\bar{A}^t\bar{\xi}+A^t N_j^\delta$.} Consequently, we derive from \eqref{eq:sgditer} that for any $s\geq0$
\begin{align*}
&(I-c_0B)^sB^{-\frac12} (Be_{k}^\delta-\bar{A}^t\bar{\xi})  \\
=&(I-c_0B)^sB^{-\frac12}\Big((I-c_0B)^{k-1}Be_1^\delta-\bar{A}^t\bar{\xi}+c_0\sum_{j=1}^{k-1} (I-c_0B)^{k-j-1}B\bar{A}^t\bar{\xi}\\
&\qquad\qquad\quad\qquad+c_0\sum_{j=1}^{k-1} (I-c_0B)^{k-j-1}B{ A^t N_j^\delta}\Big)\\
=&(I-c_0B)^{k-1+s}B^{-\frac12}(Be_1^\delta-\bar{A}^t\bar{\xi})+c_0\sum_{j=1}^{k-1} (I-c_0B)^{k-j-1+s}B^\frac12{A^t N_j^\delta},
\end{align*}
in view of the identity
$c_0\sum_{j=1}^{k-1} (I-c_0B)^{k-j-1}B=I-(I-c_0B)^{k-1}.$
By the recursion \eqref{eq:sgditer} and \eqref{eqn:noise-ind1}--\eqref{eqn:noise-ind2},
we have the following bias-variance decomposition
\begin{align*}
&\E[\|(I-c_0B)^s B^{-\frac12}(Be_{k}^\delta-\bar{A}^t\bar{\xi})\|^2]\\
=& \|(I-c_0B)^{k-1+s}B^{-\frac12}(Be_1^\delta-\bar{A}^t\bar{\xi})\|^2+c_0^2\sum_{j=1}^{k-1}\E[\| (I-c_0B)^{k-j-1+s}B^\frac12{A^t N_j^\delta}\|^2].
\end{align*}
Next we denote the summation by ${\rm I}(s)$. Then the argument for $N_j^\delta$ in the proof
of Theorem \ref{thm:decomp} and the condition $\|B\|\leq1$ imply that for any $\epsilon\in(\frac12,1)$,
\begin{align*}
&{\rm I}(s)\leq nc_0^2\sum_{j=1}^{k-1}\E[\| (I-c_0B)^{k-j-1+s}B^\frac12(Be_{j}^\delta-\bar{A}^t\bar{\xi})\|^2]\\
\leq&nc_0^2\|(I-c_0B)^{-1}\|\sum_{j=1}^{k-1}\| (I-c_0B)^{\frac{k-j-1}{2}}B^{\epsilon}\|^2\E[\| (I-c_0B)^{\frac{k-j}{2}+s}B^{-\frac12}(Be_{j}^\delta-\bar{A}^t\bar{\xi})\|^2].
\end{align*}
With the identity $\|(I-c_0B)^{-1}\|
 =(1-c_0\|B\|)^{-1}$ and the induction hypothesis
\begin{align*}
  &\E[\|(I-c_0B)^{\frac{k-j}{2}+s} B^{-\frac12}(Be_{j}^\delta-\bar{A}^t\bar{\xi})\|^2]\\
\leq& 2\|(I-c_0B)^{\frac{k-1}{2}+s}B^{-\frac12}(Be_{1}-\bar{A}^t\bar{\xi})\|^2,\quad j=1,\ldots,k-1,
\end{align*}
we deduce
\begin{align*}
{\rm I}(s)\leq \frac{2nc_0^2}{1-c_0\|B\|}\|(I-c_0B)^{\frac{k-1}{2}+s}{ B^{-\frac12}}(Be_1-\bar{A}^t\bar{\xi})\|^2\sum_{j=1}^{k-1}\| (I-c_0B)^{\frac{k-j-1}{2}}B^{\epsilon}\|^2.
\end{align*}
By Lemma \ref{lem:kernel} and the estimate \eqref{eqn:basic-bdd},
\begin{align*}
&\sum_{j=1}^{k-1}\| (I-c_0B)^{\frac{k-j-1}{2}}B^{\epsilon}\|^2=\|B^{\epsilon}\|^2+\sum_{j=1}^{k-2}\| (I-c_0B)^{k-j-1}B^{2\epsilon}\|\\
\leq&(\tfrac{2\epsilon}{ec_0})^{2\epsilon}\sum_{j=1}^{k-1}(k-j-1)^{-2\epsilon}\leq (\tfrac{2\epsilon}{ec_0})^{2\epsilon}(1+\phi(2\epsilon)).
\end{align*}
This, the assumption $c_0\|B\|\leq (2e)^{-1}$ and
the condition \eqref{cond:0} imply
\begin{align*}
{\rm I}(s)\leq \|(I-c_0B)^{\frac{k-1}{2}+s}B^{-\frac12}(Be_1-\bar{A}^t\bar{\xi})\|^2.
\end{align*}
This completes the induction step of the proof and thus also the proof of the lemma.
\end{proof}

Last, we can state a refined error estimate for the case $\alpha=0$.
\begin{theorem}\label{thm:main:al0}
Let Assumption \ref{ass:stepsize} holds with $\alpha=0$. Under condition \eqref{cond:0}, there holds
\begin{align*}
\E[\|e_{k}^\delta\|^2]\leq c^0_*k^{-2\nu}\|w\|^2+ c^0_{**}{\bar\delta}^2k,
\end{align*}
with constants $c^0_*=2(\frac{2\nu}{ec_0})^{2\nu}+6nc_0 (\frac{2(2\nu+1)}{ec_0})^{2\nu+1}$ and $c^0_{**}=3+6nc_0$.
\end{theorem}
\begin{proof}
Under condition \eqref{cond:0}, the proof of Theorem \ref{thm:decomp} and Lemma \ref{lem:0} imply
\begin{align*}
\E[\|e_{k}^\delta\|^2]
\leq &2\|(I-c_0B)^{k-1}e_1\|^2+
2c_0^2 {\bar\delta}^2\Big\|\sum_{j=1}^{k-1}(I-c_0B)^{k-j-1}B^{\frac12}\Big\|^2\\
 & + nc_0^2 \sum_{j=1}^{k-1}\E[\|(I-c_0B)^{k-j-1}(Be_j^\delta-\bar{A}^t\bar\xi)\|^2].
\end{align*}
Below we denote the last term by ${\rm II}$. Note that
\begin{align*}
  &\E[\|(I-c_0B)^{k-j-1}(Be_j^\delta-\bar{A}^t\bar\xi)\|^2]\\
\leq&(1-c_0\|B\|)^{-1}\|(I-c_0B)^{\frac{k-j-1}{2}}B^\frac12\|^2\E[\|(I-c_0B)^{\frac{k-j}{2}}B^{-\frac12}(Be_j^\delta-\bar{A}^t\bar\xi)\|^2]\\
\leq&2(1-c_0\|B\|)^{-1}\|(I-c_0B)^{k-j-1}B\|\|(I-c_0B)^{\frac{k-1}{2}}B^{-\frac12}(Be_1-\bar{A}^t\bar\xi)\|^2.
\end{align*}
By Lemma \ref{lem:kernel}, the assumption $c_0\|B\|\leq(2e)^{-1}$, the estimates \eqref{eqn:basic-bdd} and \eqref{eqn:log-exp}, we have
\begin{align*}
\sum_{j=1}^{k-1}\|(I-c_0B)^{k-j-1}B\|\leq (ec_0)^{-1} \sum_{j=1}^{k-1} (k-j-1)^{-1}\leq 3(ec_0)^{-1}\max(\ln k,1)\leq 3(ec_0)^{-1}k.
\end{align*}
Meanwhile, by the Cauchy-Schwarz inequality, we have
\begin{align*}
\|(I-c_0B)^{\frac{k-1}{2}}(B^{\frac12}e_1-B^{-\frac12}\bar{A}^t\bar\xi)\|^2
\leq 2 \big(\|e_1^t(I-c_0B)^{k-1}Be_1\|+\bar{\delta}^2\|(I-c_0B)^{\frac{k-1}{2}}\|^2\big).
\end{align*}
Combining the preceding estimates with Assumption \ref{ass:stepsize}(ii) and Lemma \ref{lem:kernel} leads to
\begin{align*}
{\rm II}
\leq&6nc_0k \big(\|(I-c_0B)^{k-1}B^{2\nu+1}\|\|w\|^2+\bar{\delta}^2\big)
\leq6nc_0\big((\tfrac{2(2\nu+1)}{ec_0})^{2\nu+1}k^{-2\nu}\|w\|^2+\bar{\delta}^2k).
\end{align*}
Similarly, there hold
\begin{align*}
   \|(I-c_0B)^{k-1}e_1\|^2 
   &\leq (\tfrac{2\nu}{ec_0})^{2\nu}k^{-2\nu}\|w\|^2,\\
\Big\|\sum_{j=1}^{k-1}(I-c_0B)^{k-j-1}B^{\frac12}\Big\|^2&\leq \Big(\sum_{j=1}^{k-1}\|(I-c_0B)^{k-j-1}B^{\frac12}\|\Big)^2\leq \tfrac32c_0^{-2}k.
\end{align*}
Combining the preceding estimates completes the proof of the theorem.
\end{proof}
}

\section{Numerical experiments and discussions}\label{sec:numer}

In this section, we provide numerical experiments to complement the analysis. To this end, we employ three examples,
denoted by \texttt{s-phillips} (mildly ill-posed), \texttt{s-gravity} (severely ill-posed) and \texttt{s-shaw}
(severely ill-posed), adapted from \texttt{phillips}, \texttt{gravity} and \texttt{shaw} in
the public MATLAB package Regutools \cite{P.C.Hansen2007} (available at \url{http://people.compute.dtu.dk/pcha/Regutools/},
last accessed on August 20, 2020). These examples are Fredholm/Volterra integral
equations of the first kind, discretized by means of either Galerkin approximation with piecewise constant basis functions
or quadrature rules, and all discretized into a linear system of size $n = m = 1000$. To explicitly control
the regularity index $\nu$ in the source condition \eqref{eqn:source}, we generate the true solution $x^\dag$ by
\begin{equation*}
  x^\dag = \frac{(A^tA)^\nu x_e}{\|(A^tA)^\nu x_e\|_{\ell^\infty}},
\end{equation*}
where $x_e$ is the exact solution provided by the package, and $\|\cdot\|_{\ell^\infty}$ denotes the maximum norm of vectors.
In the test, the exponent $\nu$ is taken from the set $\{0,1,2,4\}$. Note that the exponent $\nu$ in the source condition \eqref{eqn:source}
is slightly larger than $\nu$ defined above, due to the inherent regularity of $x_e$ (which is less than one half for all examples).
The exact data $y^\dag$ is generated by $y^\dag=A x^\dag$ and the noise data $y^\delta$ by
\begin{equation*}
  y^\delta_i:=y^\dag_i+\epsilon\|y^\dag\|_{\ell^\infty}\xi_i,\quad i=1,\cdots,n,
\end{equation*}
where $\xi_i$s are i.i.d. and follow the standard Gaussian distribution, and $\epsilon > 0$ represent the relative noise
level (the exact noise level being $\delta= \|y^\delta-y^\dag\|$).  SGD
is always initialized with $x_1 = 0$, and the maximum number of epochs is fixed at $9$e5, where one epoch
refers to $n$ SGD iterations.  All statistical quantities presented below are computed from 100 independent runs.
To verify the order optimality of SGD, we evaluate it against an order optimal regularization method with infinite qualification, i.e., Landweber
method \cite[Chapter 6]{EnglHankeNeubauer:1996}, since it is the population version of SGD,
converges steadily while enjoys order optimality, and thus serves a good benchmark
for performance comparison in terms of the convergence rate. (However, one may employ any other
order optimal methods). It is initialized with $x_1 = 0$,
with a constant stepsize $\frac{1}{\|A\|^2}$, which can be much larger than that taken by SGD.

\subsection{Numerical results for general $A$}
\begin{table}[htp!]
  \centering\small
  \begin{threeparttable}
  \caption{Comparison between SGD and LM for \texttt{s-phillips}.\label{tab:phil}}
    \begin{tabular}{cccccccccccc}
    \toprule
    \multicolumn{2}{c}{Method}&
    \multicolumn{3}{c}{ SGD($\alpha=0$)}&\multicolumn{3}{c}{ SGD($\alpha=0.1$)}&\multicolumn{2}{c}{LM}\\
    \cmidrule(lr){3-5} \cmidrule(lr){6-8} \cmidrule(lr){9-10}
    $\nu$& $\epsilon$ &$c_0$&$e_{\rm sgd}$&$k_{\rm sgd}$&$c_0$&$e_{\rm sgd}$&$k_{\rm sgd}$&$e_{\rm lm}$&$k_{\rm lm}$\\
    \midrule
    $0$&1e-3 & $4c/n$  & 1.66e-2 &  4691.28 & $c/30$ & 1.67e-2 &  2176.23 & 1.65e-2 &  5851 \cr
       &5e-3 & $4c/n$  & 9.35e-2 &  782.10  & $c/30$ & 9.49e-2 &  336.33  & 9.28e-2 &  1036  \cr
       &1e-2 & $4c/n$  & 1.29e-1 &  204.90  & $c/30$ & 1.32e-1 &  69.69   & 1.28e-1 &  249   \cr
       &5e-2 & $4c/n$  & 5.42e-1 &  108.90  & $c/30$ & 5.57e-1 &  34.11   & 5.34e-1 &  136   \cr
    \hline
    $1$&1e-3 & $c/n$   & 3.48e-4 &  539.19  & $c/n $ & 2.88e-4 &  2089.62 & 2.28e-4 &  157   \cr
       &5e-3 & $c/n$   & 3.69e-3 &  73.44   & $c/n $ & 3.32e-3 &  218.94  & 2.74e-3 &  20    \cr
       &1e-2 & $c/n$   & 6.64e-3 &  57.81   & $c/n $ & 6.12e-3 &  166.47  & 5.12e-3 &  16   \cr
       &5e-2 & $c/n$   & 3.52e-2 &  29.40   & $c/n $ & 3.31e-2 &  80.79   & 3.16e-2 &  8    \cr
    \hline
    $2$&1e-3 &$c/(30n)$& 7.02e-5 &  2115.54 &$c/(20n)$& 5.48e-5&  5912.91 & 3.22e-5 &  19    \cr
       &5e-3 &$c/(30n)$& 4.47e-4 &  1197.48 &$c/(20n)$& 4.13e-4&  3201.63 & 3.76e-4 &  11    \cr
       &1e-2 &$c/(30n)$& 1.09e-3 &  938.70  &$c/(20n)$& 1.04e-3&  2441.85 & 9.82e-4 &  8    \cr
       &5e-2 &$c/(30n)$& 2.92e-2 &  636.51  &$c/(20n)$& 2.90e-2&  1597.56 & 1.57e-2 &  5    \cr
    \hline
    $4$&1e-3 &$c/(30n)$& 9.77e-5 &  1966.38 &$c/(20n)$& 6.91e-5&  3291.18 & 1.30e-5 &  8    \cr
       &5e-3 &$c/(30n)$& 7.55e-4 &  879.51  &$c/(20n)$& 6.97e-4&  2263.89 & 3.83e-4 &  6    \cr
       &1e-2 &$c/(30n)$& 2.56e-3 &  785.94  &$c/(20n)$& 2.50e-3&  1996.83 & 1.42e-3 &  5    \cr
       &5e-2 &$c/(30n)$& 5.23e-2 &  596.73  &$c/(20n)$& 5.21e-2&  1489.29 & 2.49e-2 &  3    \cr
    \bottomrule
    \end{tabular}
    \end{threeparttable}
\end{table}

The numerical results for the three examples are shown in Tables \ref{tab:phil}--\ref{tab:shaw}, where
the notation $e_{\rm sgd}=\E[
\|x_{k_{\rm sgd}}^\delta-x^\dag\|^2]$ denotes the mean squared error achieved at the $k_{\rm sgd}$th
iteration (counted in epoch) by SGD, and $e_{\rm lm}=\|x_{k_{\rm lm}}^\delta-x^\dag\|^2$ and $k_{\rm lm}$
denote the squared $\ell^2$ error and the stopping index for Landweber method. The stopping indices
$k_{\rm sgd}$ and $k_{\rm lm}$ are taken such that the corresponding error is smallest for SGD and the
Landweber method, respectively, along the iteration trajectory. The choice of the stopping index is
motivated by a lack of provably order optimal \textit{a posteriori} stopping rules for SGD. The initial stepsize
$c_0$ is also indicated in the tables, in the form of a multiple of the constant $c=\frac{1}{\max_i(\|a_i\|^2)}$.
In the experiments, we consider two decay rates for the stepsize schedule, i.e., $\alpha=0$ and $\alpha=0.1$.

\begin{table}[htp!]
  \centering\small
  \begin{threeparttable}
  \caption{Comparison between SGD and LM for \texttt{s-gravity}.\label{tab:gravity}}
    \begin{tabular}{ccccccccccccccc}
    \toprule
    \multicolumn{2}{c}{Method}&
    \multicolumn{3}{c}{ SGD($\alpha=0$)}&\multicolumn{3}{c}{ SGD($\alpha=0.1$)}&\multicolumn{2}{c}{LM}\cr
    \cmidrule(lr){3-5} \cmidrule(lr){6-8}
    \cmidrule(lr){9-10}
    $\nu$& $\epsilon$ &$c_0$&$e_{\rm sgd}$&$k_{\rm sgd}$&$c_0$&$e_{\rm sgd}$&$k_{\rm sgd}$&$e_{\rm lm}$&$k_{\rm lm}$\\
    \midrule
    $0$&1e-3 & $c/20$  & 9.37e-2&  1000.50 & $c/10$ & 9.39e-2 &  1894.14 & 9.39e-2  & 27201 \cr
       &5e-3 & $c/20$  & 3.29e-1&  86.43   & $c/10$ & 3.29e-1 &  134.85  & 3.27e-1  & 2515  \cr
       &1e-2 & $c/20$  & 5.81e-1&  34.11   & $c/10$ & 5.80e-1 &  34.17   & 5.73e-1  & 793  \cr
       &5e-2 & $c/20$  & 2.23e0 &  5.61    & $c/10$ & 2.22e0  &  6.03    & 2.07e0   & 149   \cr
    \hline
    $1$&1e-3 &$c/(30n)$& 5.90e-4&  5604.80 &$c/(10n)$& 5.95e-4&  8095.17 & 5.68e-4  & 99   \cr
       &5e-3 &$c/(30n)$& 5.13e-3&  2069.25 &$c/(10n)$& 5.14e-3&  2707.74 & 5.02e-3  & 37    \cr
       &1e-2 &$c/(30n)$& 1.15e-2&  1356.87 &$c/(10n)$& 1.15e-2&  1688.04 & 1.12e-2  & 24    \cr
       &5e-2 &$c/(30n)$& 6.48e-2&  613.41  &$c/(10n)$& 6.48e-2&  709.08  & 6.19e-2  & 11    \cr
    \hline
    $2$&1e-3 &$c/(50n)$& 1.32e-4&  2441.85 &$c/(20n)$& 1.28e-4&  3983.34 & 6.82e-5  & 23    \cr
       &5e-3 &$c/(50n)$& 7.74e-4&  1274.67 &$c/(20n)$& 7.64e-4&  1940.70 & 6.06e-4  & 12    \cr
       &1e-2 &$c/(50n)$& 1.92e-3&  1047.03 &$c/(20n)$& 1.91e-3&  1580.49 & 1.47e-3  & 10    \cr
       &5e-2 &$c/(50n)$& 2.35e-2&  708.72  &$c/(20n)$& 2.34e-2&  1013.31 & 1.61e-2  & 6    \cr
    \hline
    $4$&1e-3 &$c/(60n)$& 1.03e-4&  2212.26 &$c/(30n)$& 8.38e-5&  3982.77 & 1.30e-5  & 10    \cr
       &5e-3 &$c/(60n)$& 4.65e-4&  1054.53 &$c/(30n)$& 4.24e-4&  2002.71 & 2.04e-4  & 7    \cr
       &1e-2 &$c/(60n)$& 1.29e-3&  941.19  &$c/(30n)$& 1.25e-3&  1782.99 & 6.42e-4  & 6    \cr
       &5e-2 &$c/(60n)$& 2.25e-2&  746.67  &$c/(30n)$& 2.26e-2&  1398.72 & 8.58e-3  & 3     \cr
    \bottomrule
    \end{tabular}
    \end{threeparttable}
\end{table}

First we comment on the SGD results. Clearly, for each fixed $\nu$, the mean squared error $e_{\rm sgd}$ (and
also $e_{\rm lm}$) decreases to zero as the noise level $\epsilon$ tends to zero, but it takes
more iterations to reach the optimal error, and the decay rate depends on
the regularity index $\nu$ roughly as the theoretical prediction $O(\delta^\frac{4\nu}{2\nu+1})$. The larger is
the regularity index $\nu$, the faster the error decays and the fewer iterations it needs in order to
reach the optimal error.  The results obtained by SGD with $\alpha=0$ and $\alpha=0.1$ are largely
comparable with each other, but generally the former imposes a more stringent condition on the initial
stepsize $c_0$ than the latter so as to achieve comparable accuracy. This is attributed to the fact that
polynomially decaying stepsize schedules have built-in variance reduction mechanism as the iteration
proceeds. Nonetheless, at the low-regularity index (indicated by $\nu=0$ in the table), the initial
stepsize can be taken independent of $n$. Next we compare the results of SGD with Landweber method. For all regularity indices, SGD,
with either constant or decaying stepsize schedule, can achieve an accuracy comparable with that by the Landweber
method, provided that the initial stepsize $c_0$ for SGD is taken to be of order $O(n^{-1})$. Generally, the
larger the index $\nu$ is, the smaller the value $c_0$ should be taken in the stepsize schedule, in
order to fully realize the benefit of smooth solutions. This observation agrees well with the observation
in Remark \ref{rem:ppg1}. These observations hold for all three examples, which have different degree of
ill-posedness. Thus they are fully in line with the convergence analysis in Section \ref{sec:conv}.

\begin{table}[htp!]
  \centering\small
  \begin{threeparttable}
  \caption{Comparison between SGD and LM for \texttt{s-shaw}.\label{tab:shaw}}
    \begin{tabular}{ccccccccccccccc}
    \toprule
    \multicolumn{2}{c}{Method}&
    \multicolumn{3}{c}{ SGD($\alpha=0$)}&\multicolumn{3}{c}{ SGD($\alpha=0.1$)}&\multicolumn{2}{c}{LM}\cr
    \cmidrule(lr){3-5} \cmidrule(lr){6-8}
    \cmidrule(lr){9-10}
    $\nu$& $\epsilon$ &$c_0$&$e_{\rm sgd}$&$k_{\rm sgd}$&$c_0$&$e_{\rm sgd}$&$k_{\rm sgd}$&$e_{\rm lm}$&$k_{\rm lm}$\\
    \midrule
    $0$&1e-3 &$ c   $& 2.81e-1  & 2704.92 &$  2c  $& 2.81e-1  & 5853.54 & 2.81e-1  & 760983 \cr
       &5e-3 &$ c   $& 5.37e-1  & 67.14   &$  2c  $& 5.33e-1  & 94.92   & 5.25e-1  & 18588  \cr
       &1e-2 &$ c   $& 7.08e-1  & 42.42   &$  2c  $& 6.98e-1  & 60.18   & 6.67e-1  & 12385  \cr
       &5e-2 &$ c   $& 3.91e0   & 10.59   &$  2c  $& 3.66e0   & 14.91   & 2.91e0   & 3392   \cr
    \hline
    $1$&1e-3 &$2c/n $& 1.21e-4  & 275.70  &$ 4c/n $& 1.26e-4  & 453.00  & 5.95e-5  & 144    \cr
       &5e-3 &$2c/n $& 1.45e-3  & 142.05  &$ 4c/n $& 1.48e-3  & 202.50  & 1.26e-3  & 71    \cr
       &1e-2 &$2c/n $& 5.75e-3  & 113.01  &$ 4c/n $& 5.62e-3  & 148.11  & 5.21e-3  & 54     \cr
       &5e-2 &$2c/n $& 1.51e-1  & 64.77   &$ 4c/n $& 1.54e-1  & 97.02   & 1.47e-1  & 36    \cr
    \hline
    $2$&1e-3 &$2c/n   $& 1.53e-4  & 255.27  &$ 4c/n  $& 1.29e-4  & 746.46  & 6.36e-5  & 50     \cr
       &5e-3 &$2c/n   $& 2.00e-3  & 84.60   &$ 4c/n  $& 1.73e-3  & 235.08  & 1.51e-3  & 37     \cr
       &1e-2 &$2c/n   $& 6.43e-3  & 64.77   &$ 4c/n  $& 6.05e-3  & 172.32  & 5.71e-3  & 30    \cr
       &5e-2 &$2c/n   $& 8.17e-2  & 11.88   &$ 4c/n  $& 8.00e-2  & 29.49   & 7.08e-2  & 5      \cr
    \hline
    $4$&1e-3 &$c/(30n)$&5.79e-5 & 1966.38 &$c/(10n)$&5.92e-5  & 2863.35 & 3.13e-5  & 9     \cr
       &5e-3 &$c/(30n)$&6.00e-4 & 941.19  &$c/(10n)$&6.06e-4  & 1116.81 & 3.71e-4  & 5      \cr
       &1e-2 &$c/(30n)$&1.99e-3 & 828.45  &$c/(10n)$&2.00e-3  & 1002.93 & 1.01e-3  & 4      \cr
       &5e-2 &$c/(30n)$&3.61e-2 & 645.75  &$c/(10n)$&3.61e-2  & 746.67  & 6.45e-3  & 1      \cr
    \bottomrule
    \end{tabular}
    \end{threeparttable}
\end{table}

In order to shed further insights into the convergence behavior of SGD, we present numerical results
with four different values of $c_0$ (i.e., $ \min(c,nc^*)$, 10$c^*$, $c^*$ and $\frac{c^*}{10}$, with
$c^*$ from the tables) in Figs. \ref{fig:err-NL0} and \ref{fig:err}, for the examples with $\nu=1$ and
exact and noisy data, respectively. In the case of exact data, the mean squared
error $e_{\rm sgd}$ consists of only approximation and stochastic errors, and it decreases to zero
as the iteration proceeds. With a large initial stepsize, the error $e_{\rm sgd}$ decreases fast during the initial
iterations, but only at a slow rate $O(k^{-(1-\alpha)})$, whereas with a small $c_0$, the initial decay
is much slower. The asymptotic decay rate matches the optimal decay $O(k^{-2\nu(1-\alpha)})$ only when
$c_0$ decreases to $O(n^{-1})$, which otherwise exhibits only a slower decay $O(k^{-\min(2\nu,1)(1-\alpha)})$
and thus an undesirable saturation phenomenon.  Note that for small $c_0$,
the asymptotic decay $O(k^{-2\nu(1-\alpha)})$ kicks in only after a sufficient number of iterations,
which agrees with the condition $h_0(k)\leq \frac12$ etc in the analysis. Further, there is
an interesting transition layer for medium $c_0$ (but still of order $O(n^{-1})$), for which it first
exhibits the desired asymptotic decay, and then shifts back to a slower decay rate eventually. The
presence of the wide transition region indicates that the optimal convergence can still be achieved
for noisy data, even if the employed $c_0$ is larger than the critical value suggested by the theoretical
analysis in Section \ref{sec:conv}. These observations hold for both constant and polynomially decaying stepsize schedules. These numerical
results show that a small initial stepsize $c_0$ is necessary for overcoming the saturation phenomenon of SGD.

These empirical observations remain largely valid also for noisy data in Fig. \ref{fig:err}. It is observed that the asymptotic decay rate
is higher for smaller initial stepsizes, but now only up to a certain iteration number, due to the presence
of the propagation error, which increases monotonically as the iteration proceeds and eventually dominates
the total error. This leads to the familiar semi-convergence behavior in the second and third columns of
Fig. \ref{fig:err}. The proper balance between the decaying approximation error and the increasing
propagation error determines the attainable accuracy. One clearly observes that the larger is $c_0$,
the faster the asymptotic decay kicks in, but also the quicker the SGD iterate starts to diverge, which
can compromise greatly the attainable accuracy along the trajectory, leading to the undesirable
saturation phenomenon. When the initial stepsize $c_0$ becomes smaller, { the attainable accuracy improves
steadily}. In particular, with a sufficiently small $c_0$, the attained error is optimal (but
of course at the expense of a much increased computational complexity). This observation naturally leads to
the important question whether it is possible to design novel stepsize schedules (possibly not of polynomially decaying
type) that enjoy both fast pre-asymptotic and asymptotic convergence behavior.

\begin{figure}
  \centering
  \setlength{\tabcolsep}{4pt}
  \begin{tabular}{ccc}
\includegraphics[width=0.31\textwidth,trim={1.5cm 0 0.5cm 0.5cm}]{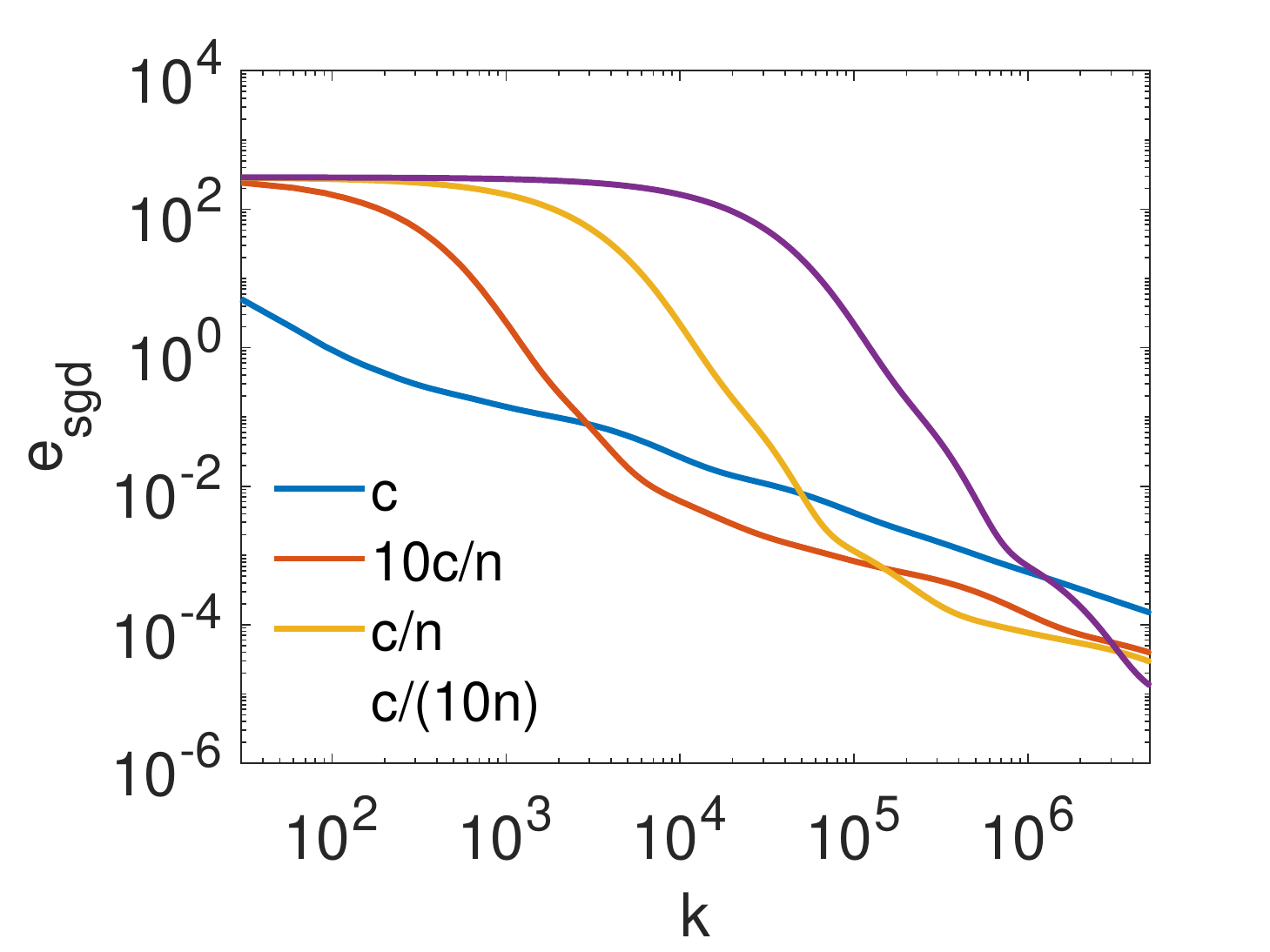}&
\includegraphics[width=0.31\textwidth,trim={1.5cm 0 0.5cm 0.5cm}]{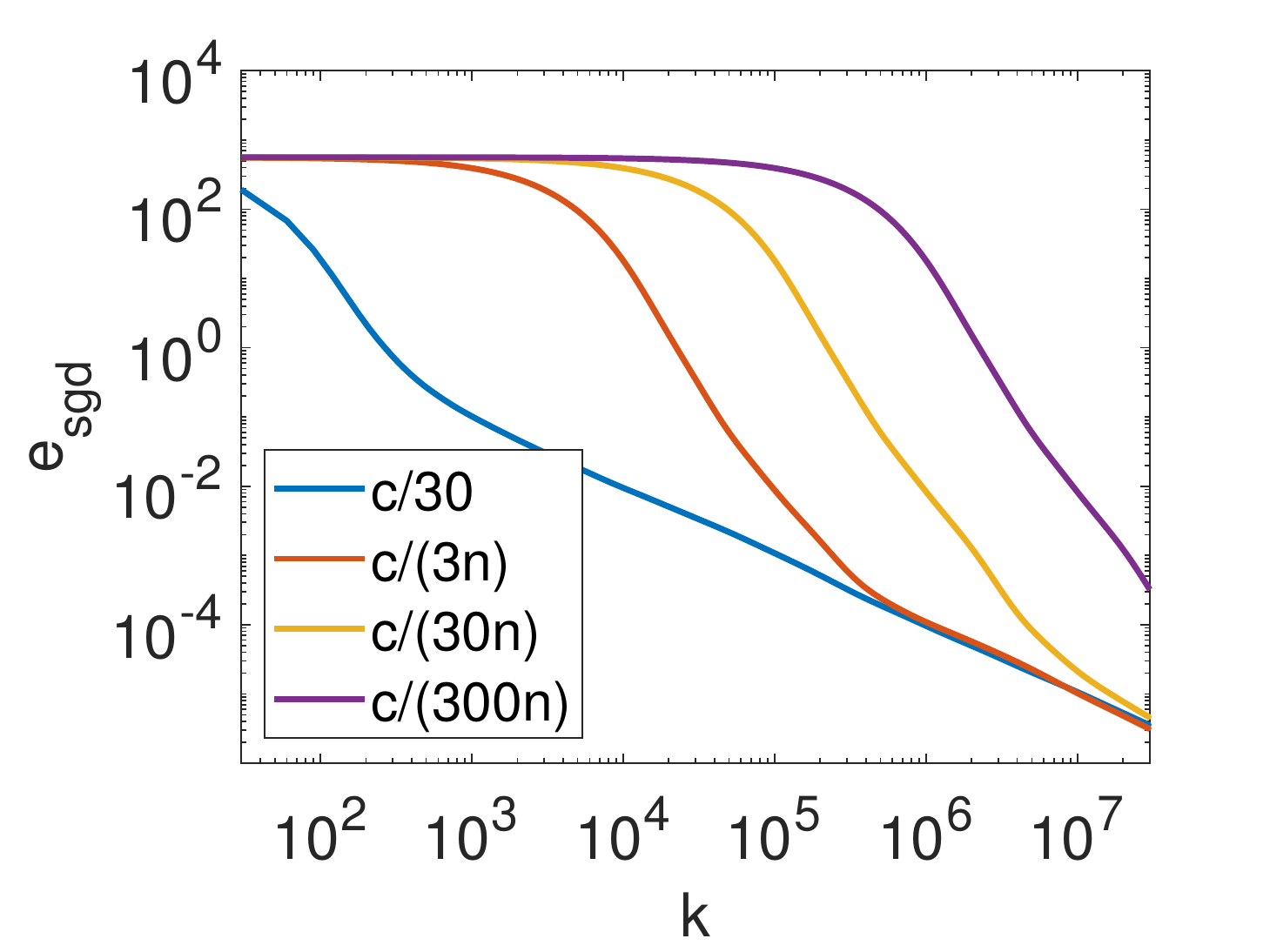}&
\includegraphics[width=0.31\textwidth,trim={1.5cm 0 0.5cm 0.5cm}]{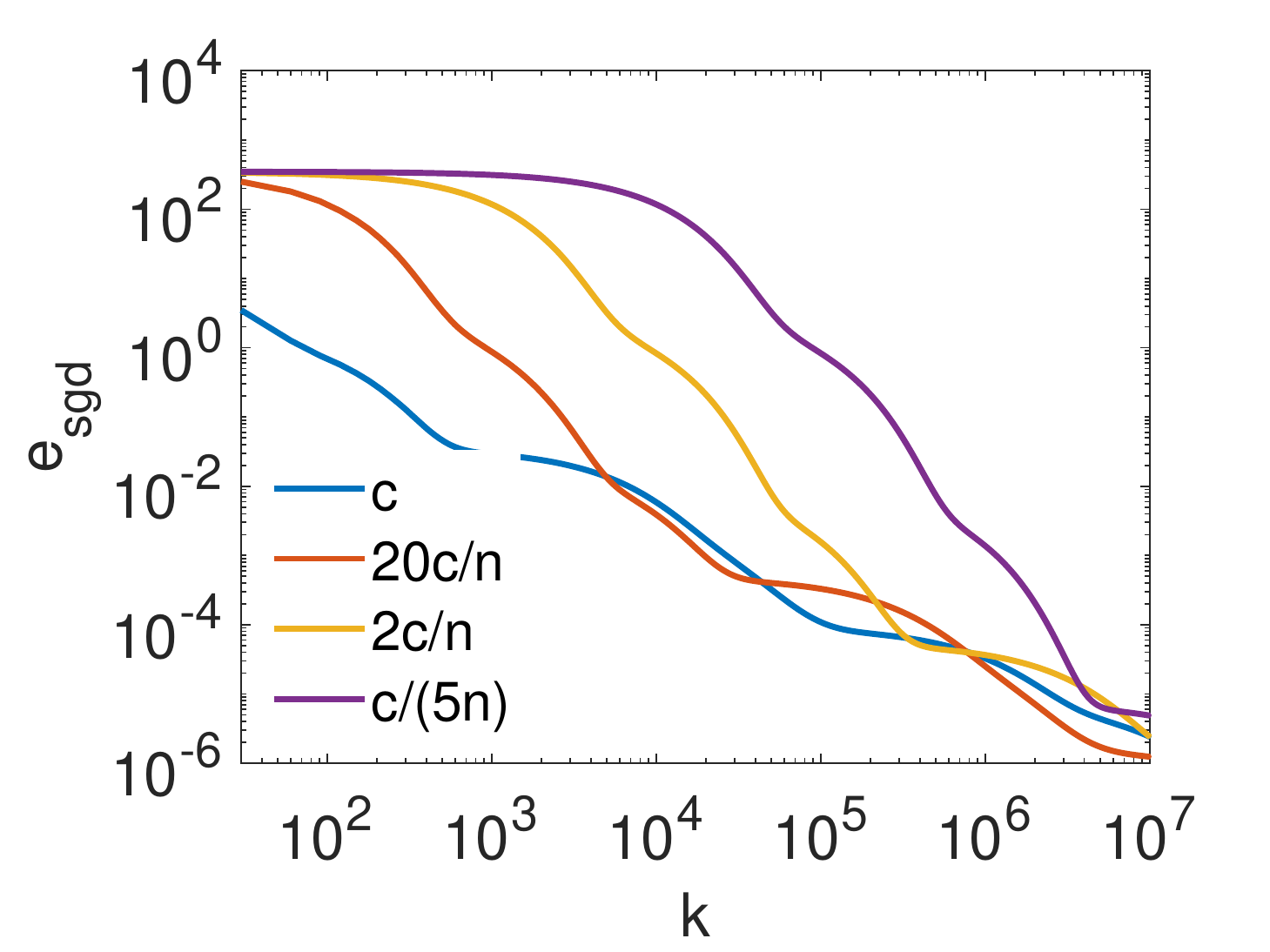}\\
\includegraphics[width=0.31\textwidth,trim={1.5cm 0 0.5cm 0.5cm}]{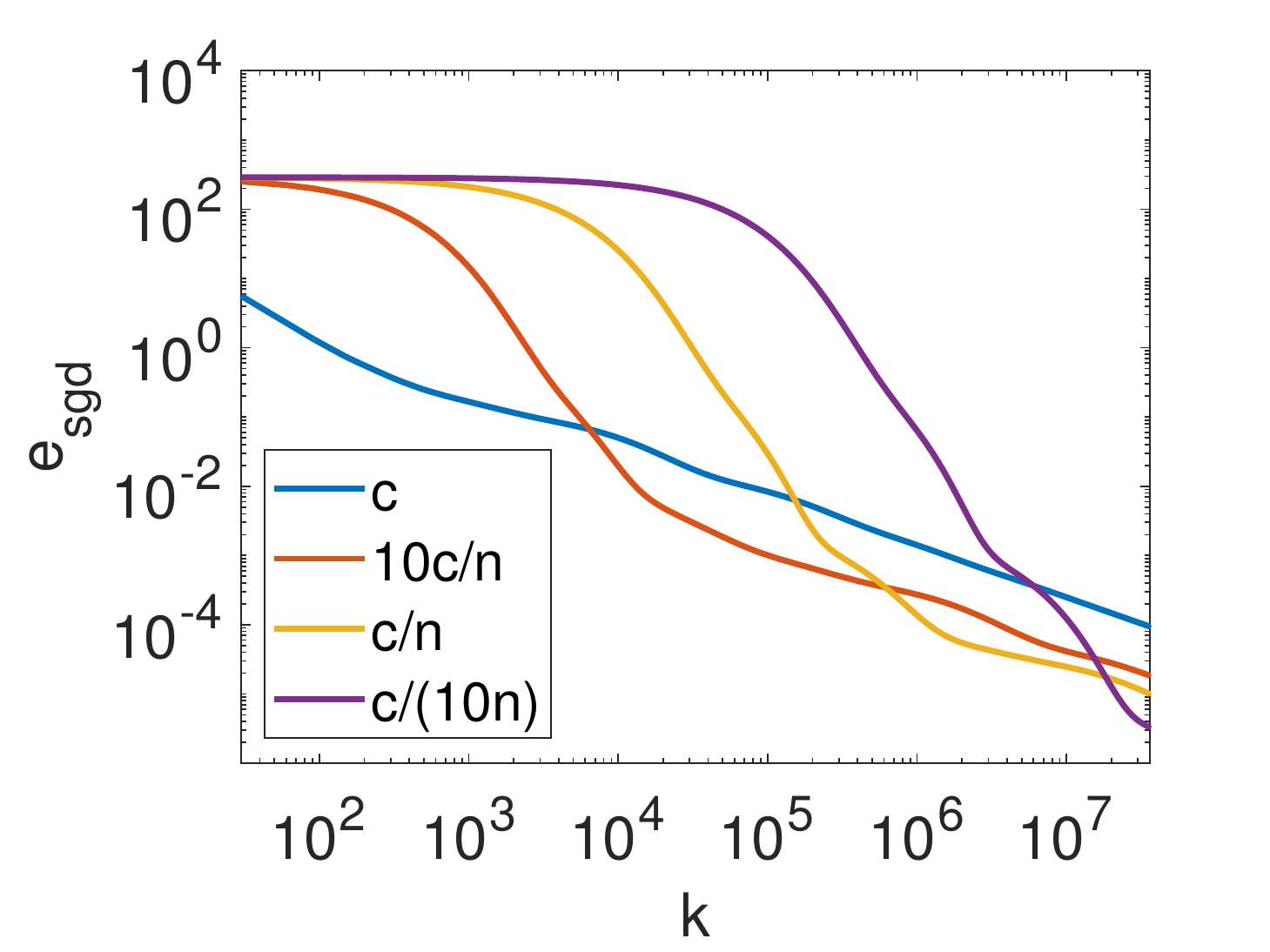}&
\includegraphics[width=0.31\textwidth,trim={1.5cm 0 0.5cm 0.5cm}]{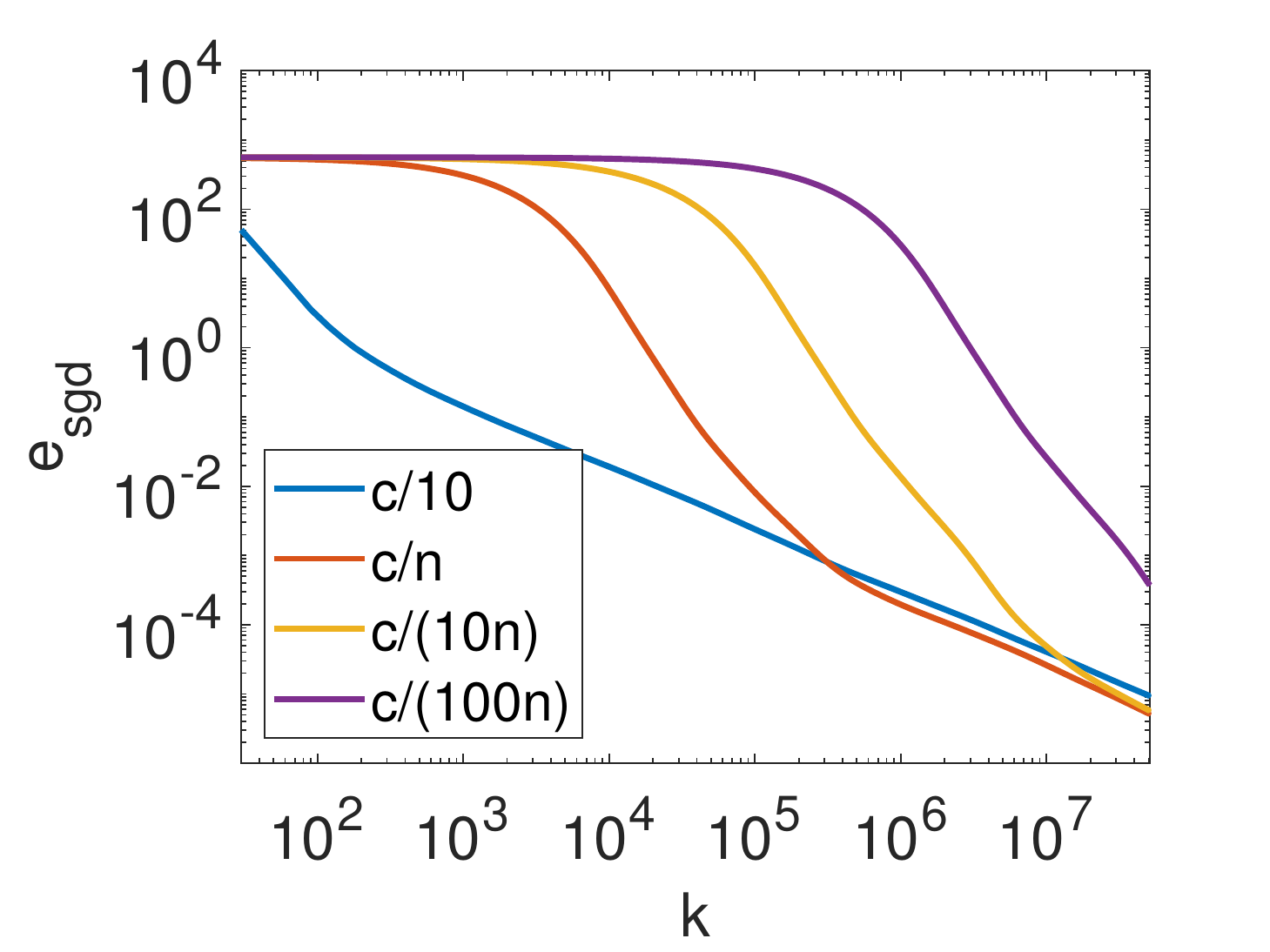}&
\includegraphics[width=0.31\textwidth,trim={1.5cm 0 0.5cm 0.5cm}]{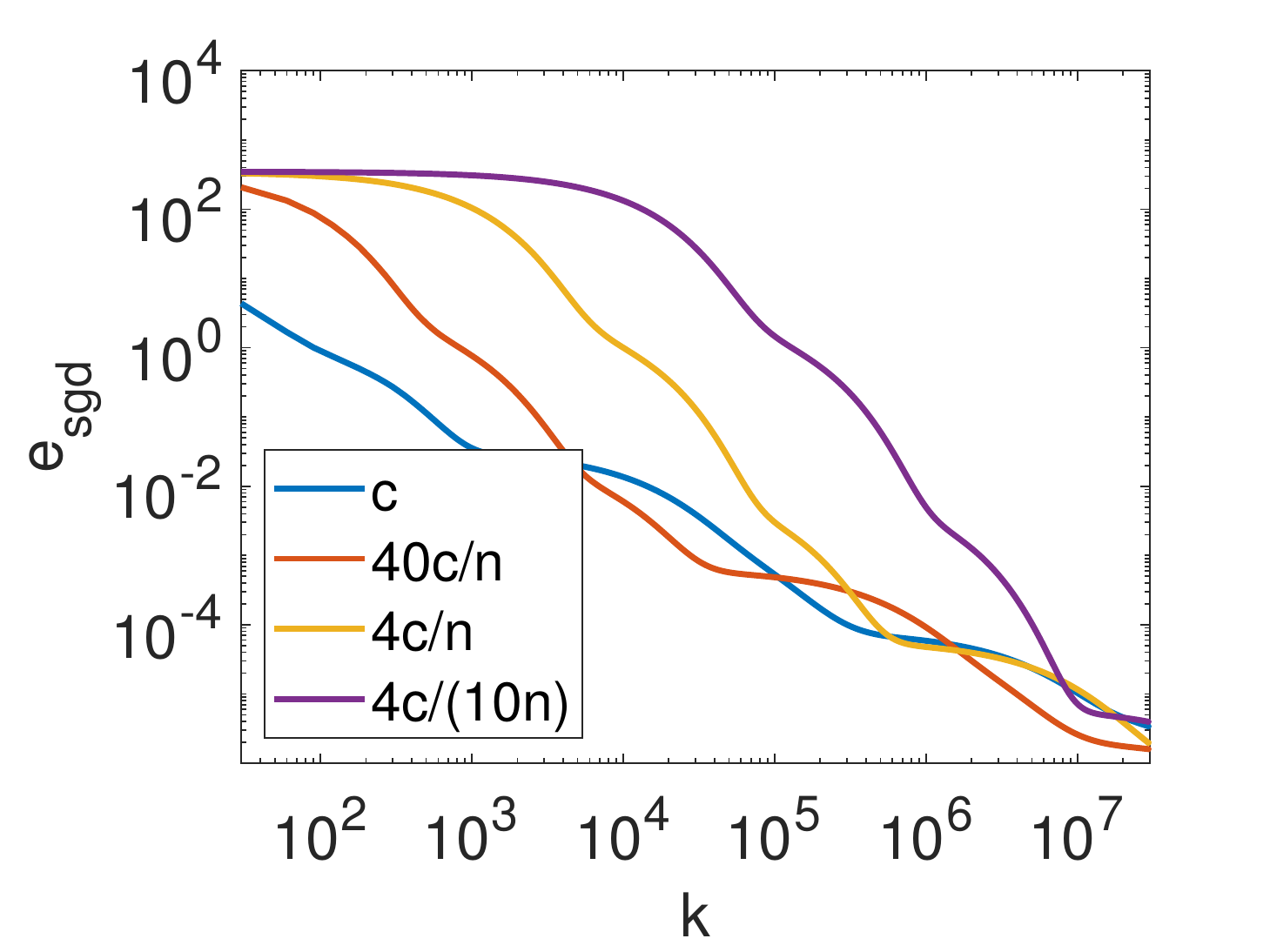}\\
\texttt{s-phillips}& \texttt{s-gravity} & \texttt{s-shaw}
\end{tabular}
\caption{The convergence trajectory of the SGD error with different initial stepsize $c_0$ for the examples
with $\nu=1$. The top and bottom rows are for $\alpha=0$ and $\alpha=0.1$, respectively. \label{fig:err-NL0}}
\end{figure}

\begin{figure}[hbt!]
\centering
  \setlength{\tabcolsep}{4pt}
\begin{tabular}{ccc}
\includegraphics[width=0.31\textwidth,trim={1.5cm 0 0.5cm 0.5cm}]{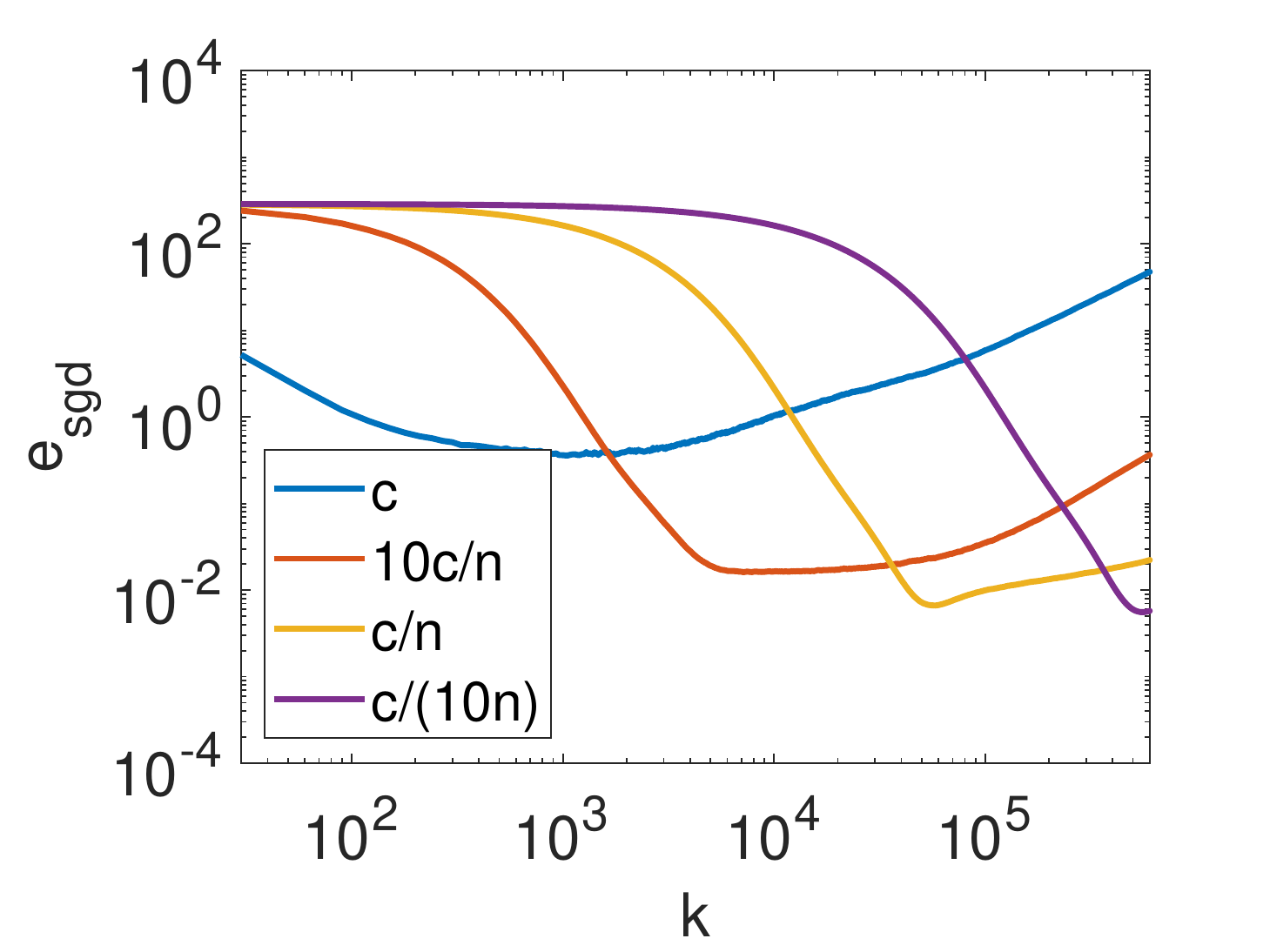}&
\includegraphics[width=0.31\textwidth,trim={1.5cm 0 0.5cm 0.5cm}]{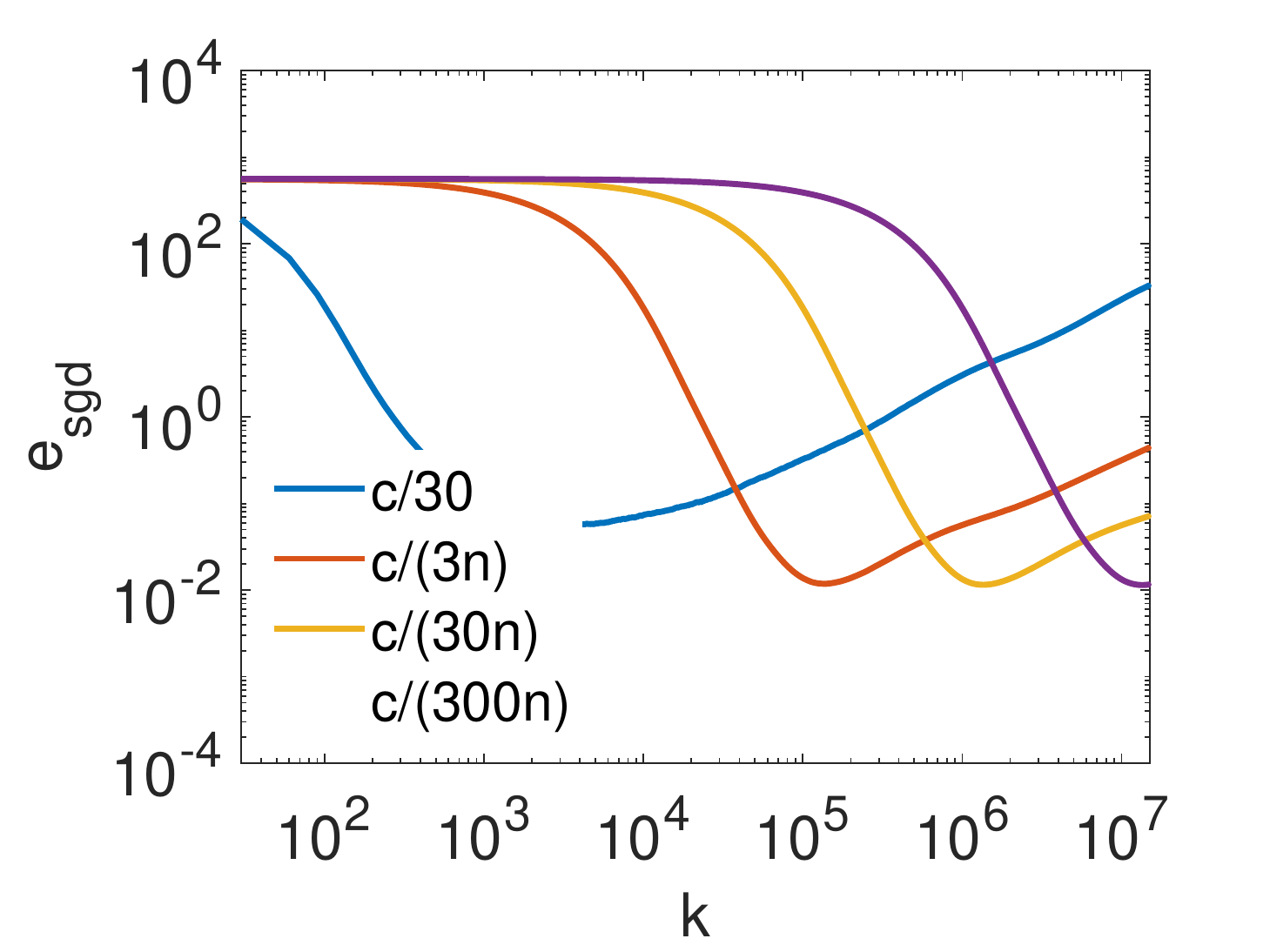}&
\includegraphics[width=0.31\textwidth,trim={1.5cm 0 0.5cm 0.5cm}]{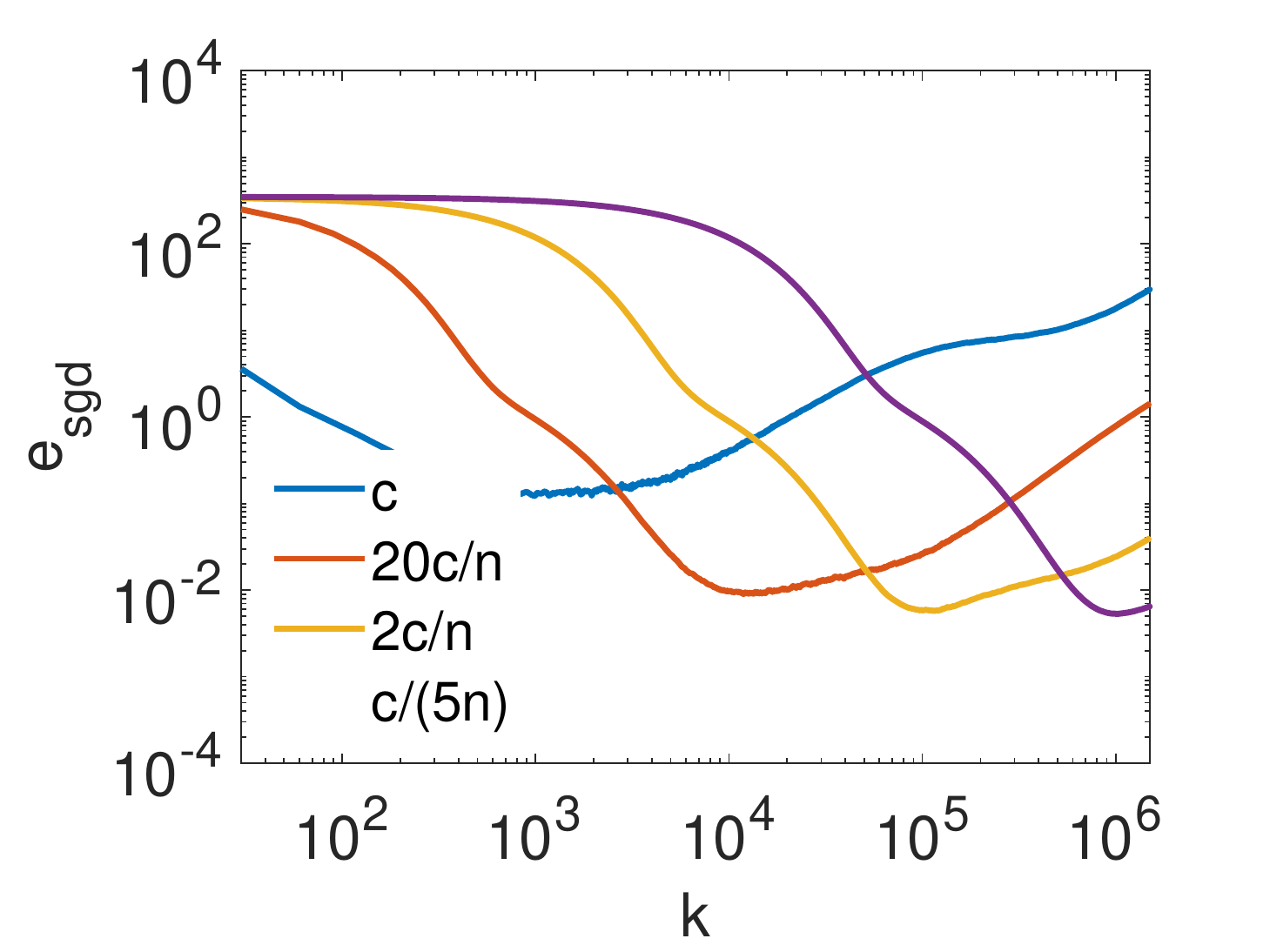}\\
\includegraphics[width=0.31\textwidth,trim={1.5cm 0 0.5cm 0.5cm}]{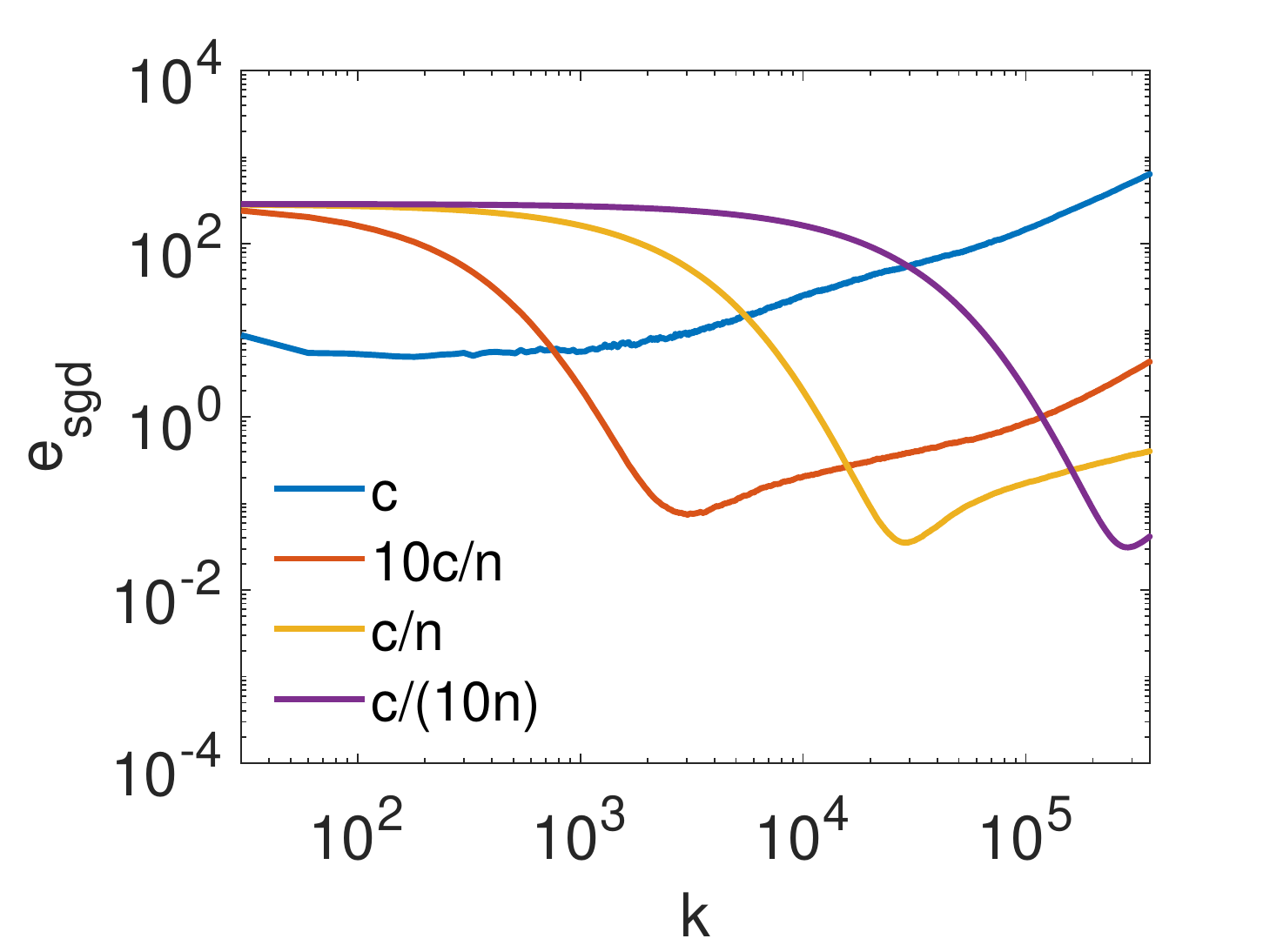}&
\includegraphics[width=0.31\textwidth,trim={1.5cm 0 0.5cm 0.5cm}]{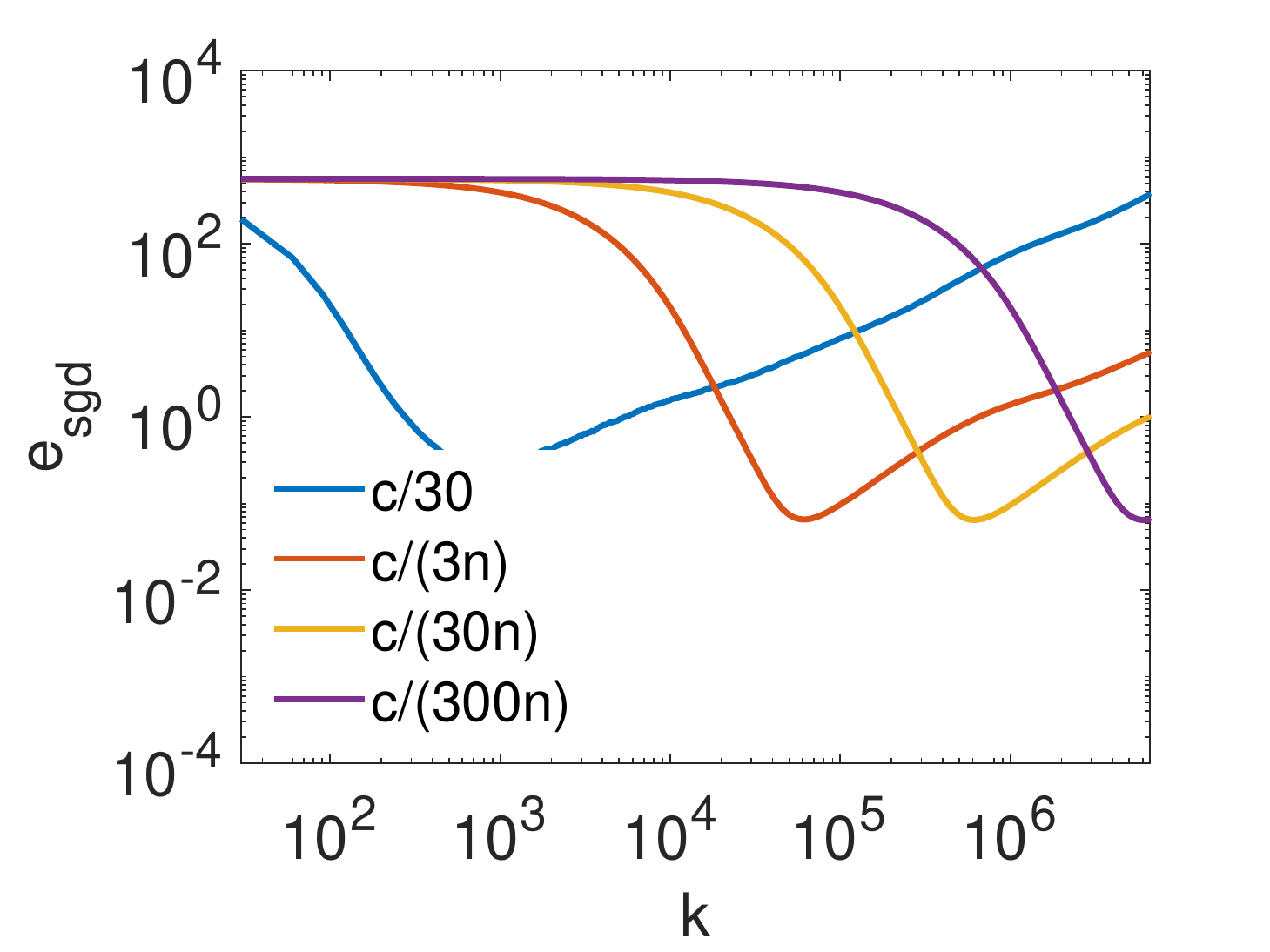}&
\includegraphics[width=0.31\textwidth,trim={1.5cm 0 0.5cm 0.5cm}]{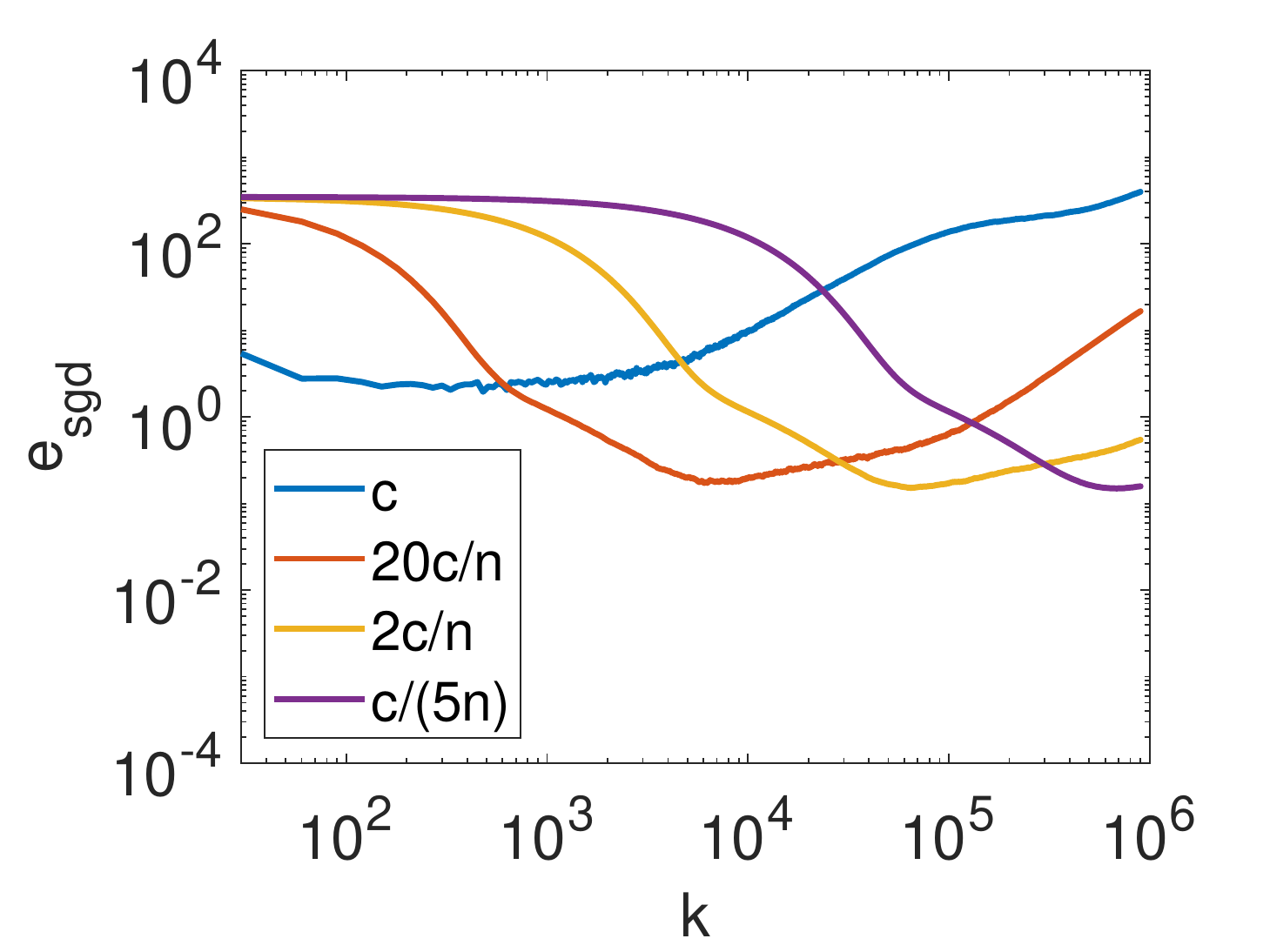}\\
\texttt{s-phillips}& \texttt{s-gravity} & \texttt{s-shaw}
\end{tabular}
\caption{The convergence trajectory of the SGD error (with $\alpha=0$) with different initial stepsize $c_0$ for the examples with
$\nu=1$. The top and bottom rows are for $\epsilon=$1e-2 and $\epsilon$=5e-2, respectively.\label{fig:err}}
\end{figure}

\subsection{On Assumption \ref{ass:stepsize}(iii)}\label{sec:condition}
\begin{table}[htp!]
  \centering\small
  \begin{threeparttable}
  \caption{Comparison between SGD with $\alpha=0$ for \texttt{s-phillips} with $A$ and $\tilde A$.\label{tab:phil_UA}}
    \begin{tabular}{cccccccccccc}
    \toprule
    \multicolumn{3}{c}{Method}&
    \multicolumn{2}{c}{ SGD with $A$}& \multicolumn{2}{c}{ SGD with $\tilde A$}&\\
    \cmidrule(lr){4-5} \cmidrule(lr){6-7}
    $\nu$& $\epsilon$ &$c_0$&$e$&$k$&$e$&$k$&\\
    \midrule
    $0$&1e-3 & $4c/n$  & 1.66e-2 &  4691.28 & 1.65e-2 & 4738.4  \cr
       &5e-3 & $4c/n$  & 9.35e-2 &  782.10  & 9.28e-2 & 835.35   \cr
       &1e-2 & $4c/n$  & 1.29e-1 &  204.90 & 1.28e-1 & 198.75 \cr
       &5e-2 & $4c/n$  & 5.42e-1 &  108.90  & 5.40e-1 & 111.85 \cr
    \hline
    $1$&1e-3 & $c/n$   & 3.48e-4 &  539.19  & 2.29e-4 & 507.55\cr
       &5e-3 & $c/n$   & 3.69e-3 &  73.44  & 2.87e-3 &  71.2   \cr
       &1e-2 & $c/n$   & 6.64e-3 &  57.81  & 5.72e-3 & 57.75 \cr
       &5e-2 & $c/n$   & 3.52e-2 &  29.40  & 3.84e-2 & 30.4 \cr
    \hline
    $2$&1e-3 &$c/(30n)$& 7.02e-5 &  2115.54 & 3.49e-5& 2021.6\cr
       &5e-3 &$c/(30n)$& 4.47e-4 &  1197.48 & 3.66e-4 & 1186.10\cr
       &1e-2 &$c/(30n)$& 1.09e-3 &  938.70 & 9.90e-4 & 934.75\cr
       &5e-2 &$c/(30n)$& 2.92e-2 &  636.51  & 2.94e-2& 639.60 \cr
    \hline
    $4$&1e-3 &$c/(30n)$& 9.77e-5 &  1966.38 & 2.49e-5& 1103.00 \cr
       &5e-3 &$c/(30n)$& 7.55e-4 &  879.51 & 6.34e-4& 869.40 \cr
       &1e-2 &$c/(30n)$& 2.56e-3 &  785.94  & 2.43e-3& 781.80\cr
       &5e-2 &$c/(30n)$& 5.23e-2 &  596.73  & 5.24e-2& 597.60\cr
    \bottomrule
    \end{tabular}
    \end{threeparttable}
\end{table}

{
The convergence analysis in Section \ref{sec:conv} requires Assumption \ref{ass:stepsize}(iii). This appears largely to be
a limitation of the analysis technique. To illustrate
this, we compare the results between the systems with a general matrix $A$ and with one that
satisfies  Assumption \ref{ass:stepsize}(iii). The latter can be constructed from the former as follows.
Let $A=U\Sigma V^t$ be the singular value decomposition of $A$. Then we replace $A$ by
$\tilde A = U^tA$ and $y^\delta$ by $\tilde y^\delta = U^ty^\delta$ so that $\tilde A$ satisfies
Assumption \ref{ass:stepsize}(ii)--(iii). The numerical results for \texttt{s-phillips} are shown in
Table \ref{tab:phil_UA}. It is observed that the results obtained by SGD with $A$ and $\tilde A$ are
largely comparable with each other for all the noise levels and smooth indices, {especially when the
amount of the data noise is not too small.} Although not presented, the observations are identical
for other examples, including multiplying the matrix $A$ by an arbitrary orthonormal matrix so
long as $c_0$ is sufficiently small. These observations are also confirmed by the corresponding
trajectories: The trajectory of the mean squared error for the three examples with $\nu=1$ for $A$ and
$\tilde A$ nearly overlay each other {when the data is not too small (as in most practical inverse
problems)} (cf. Fig. \ref{fig:err_UA}). For exact data, the trajectory overlaps up to a certain point
around 1e-4 (which depends on the value of $c_0$), and the value of levelling off is observed to
further decrease by choosing smaller $c_0$. One interesting open question is thus to establish the saturation-overcoming
phenomenon without Assumption \ref{ass:stepsize}(iii), as the experimental results suggest.}


\begin{figure}[hbt!]
\centering
  \setlength{\tabcolsep}{4pt}
\begin{tabular}{ccc}
\includegraphics[width=0.31\textwidth,trim={1.5cm 0 0.5cm 0.5cm}]{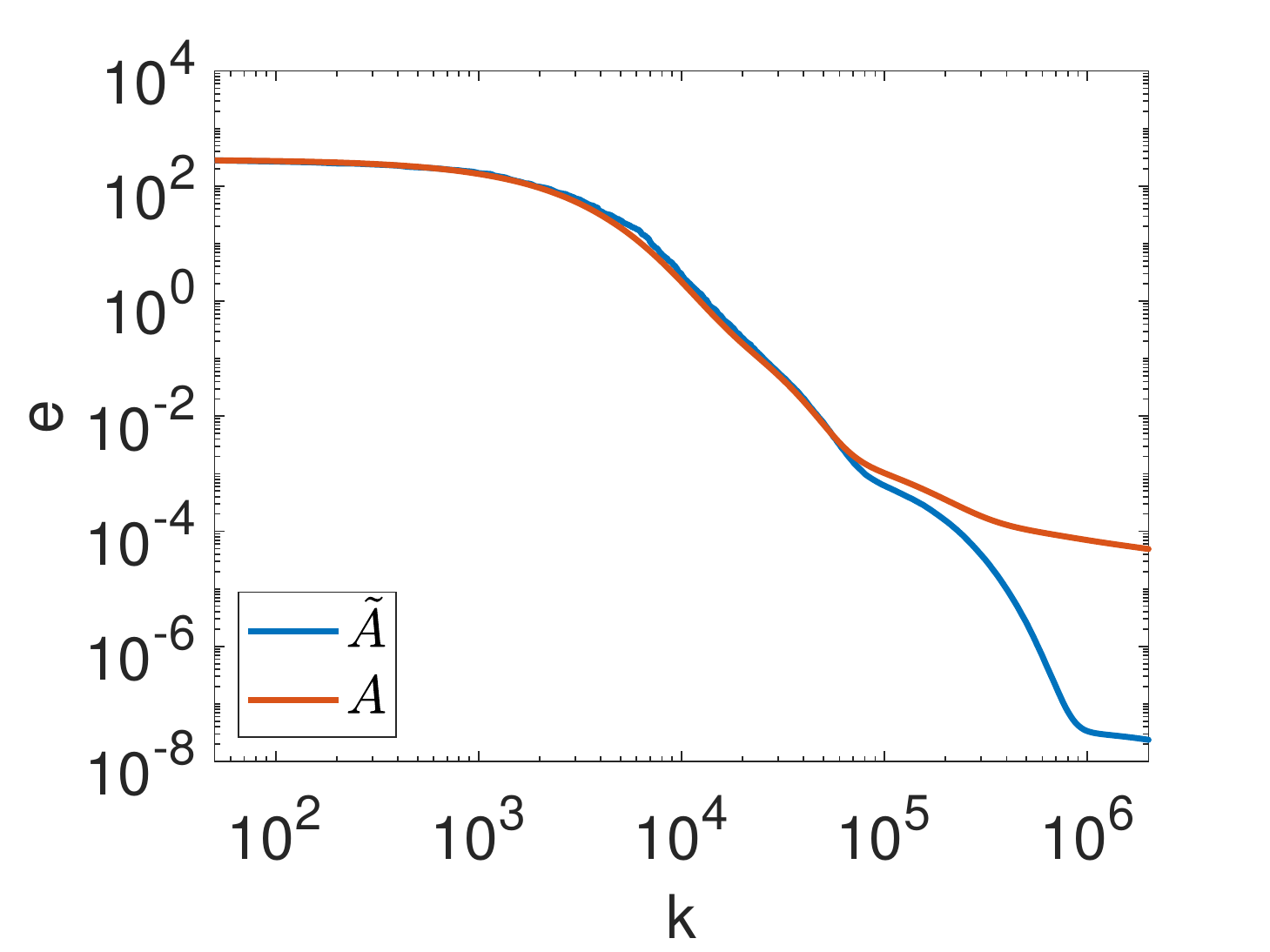}&
\includegraphics[width=0.31\textwidth,trim={1.5cm 0 0.5cm 0.5cm}]{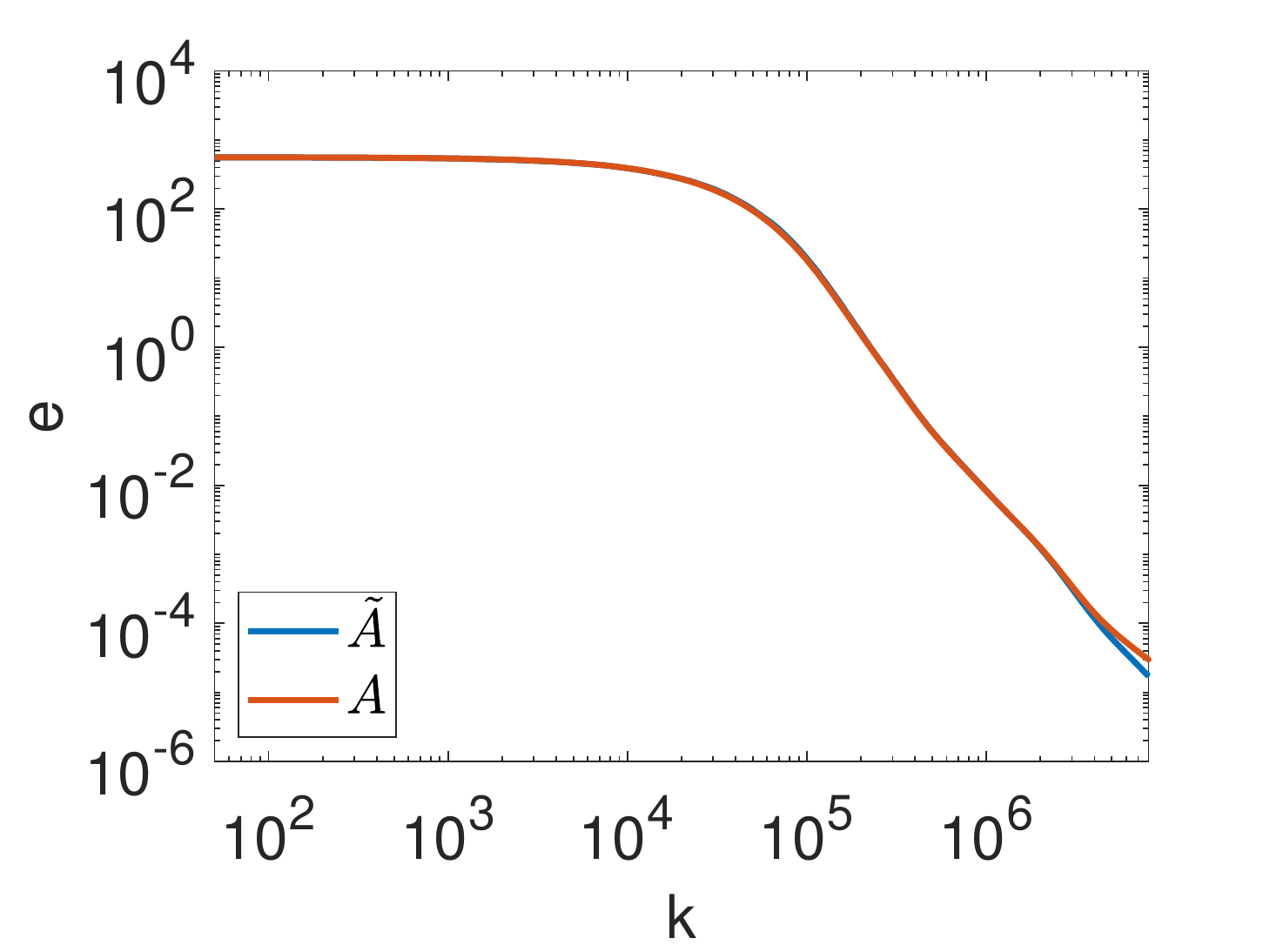}&
\includegraphics[width=0.31\textwidth,trim={1.5cm 0 0.5cm 0.5cm}]{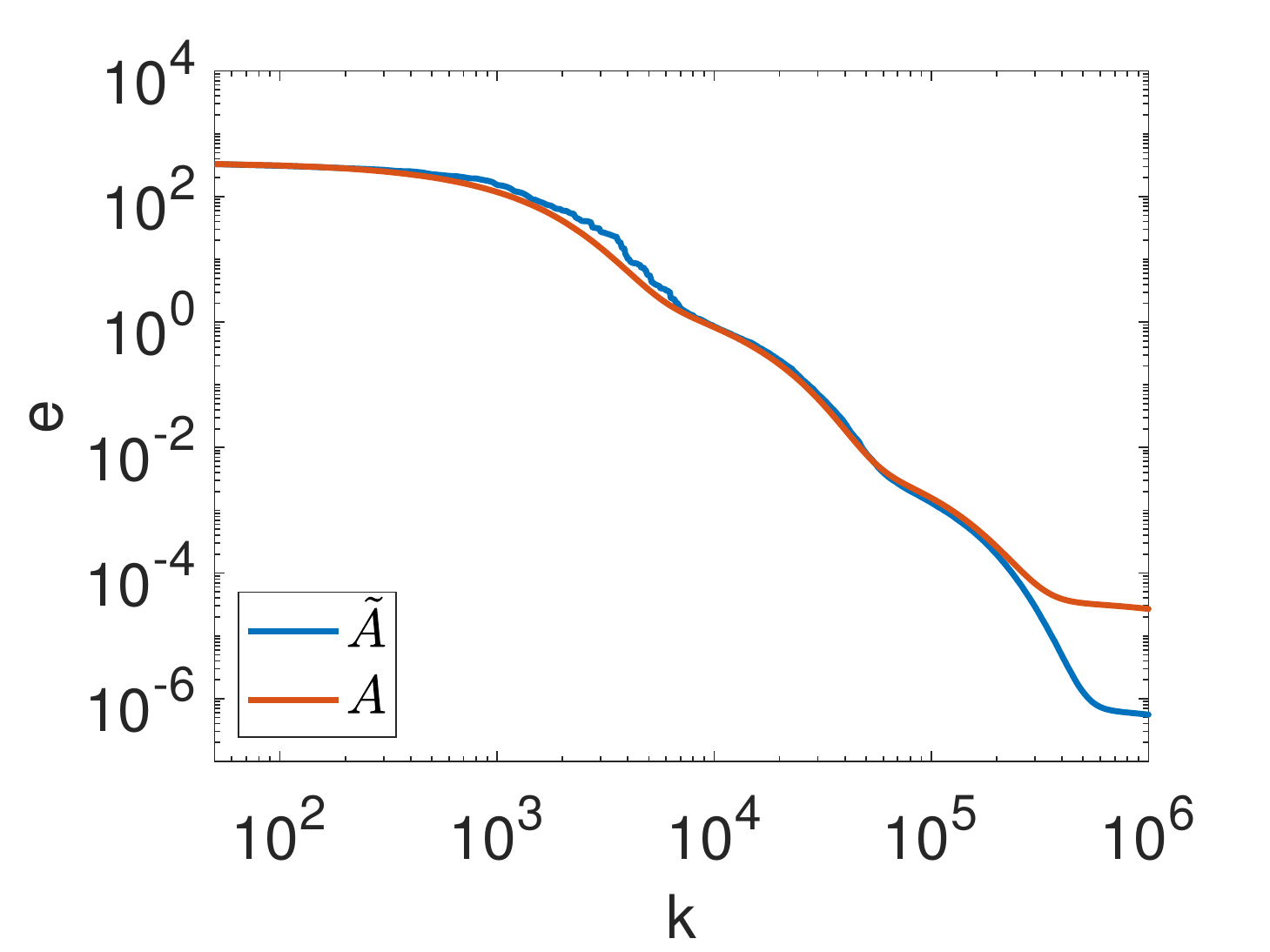}\\
\includegraphics[width=0.31\textwidth,trim={1.5cm 0 0.5cm 0.5cm}]{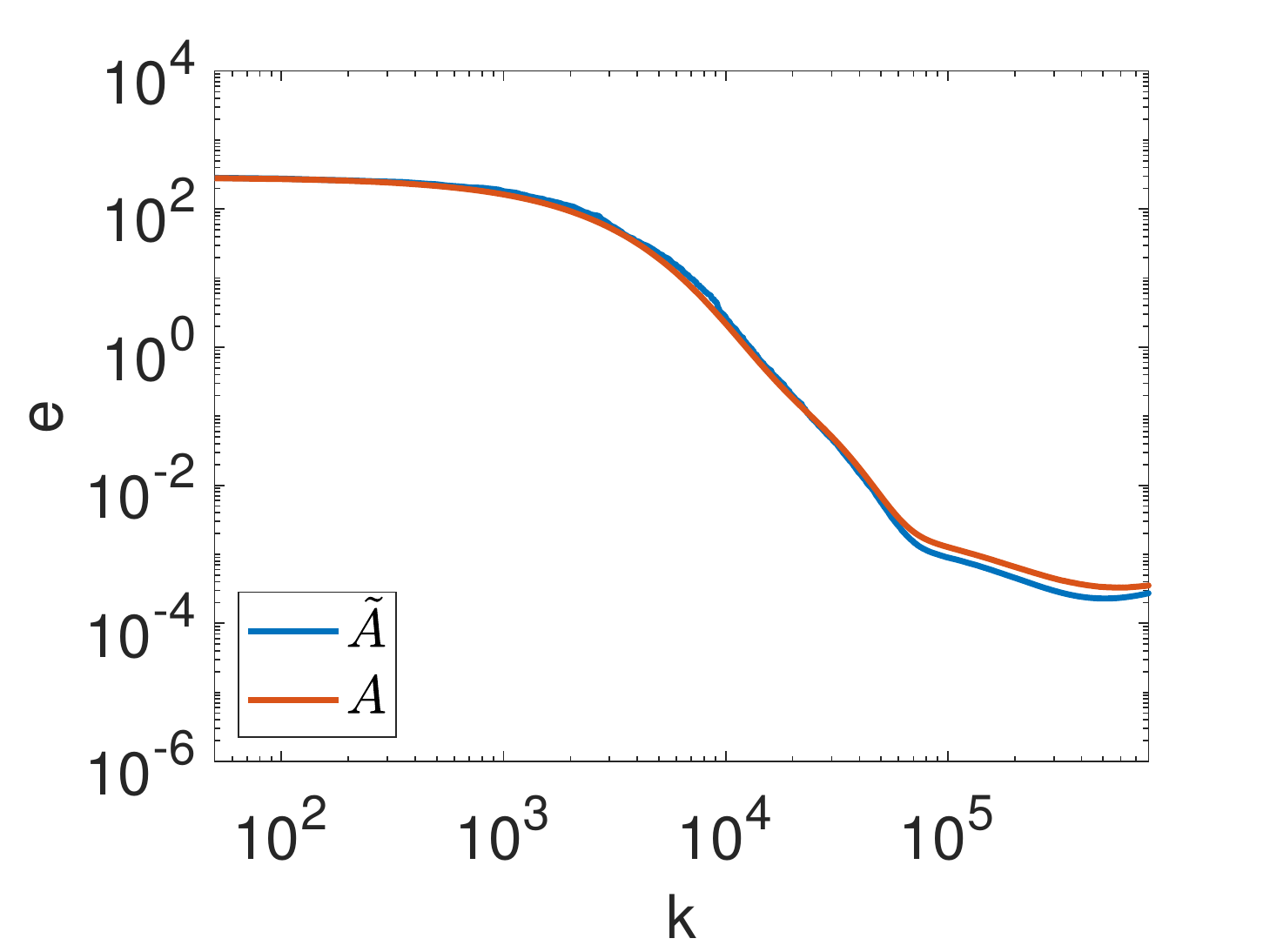}&
\includegraphics[width=0.31\textwidth,trim={1.5cm 0 0.5cm 0.5cm}]{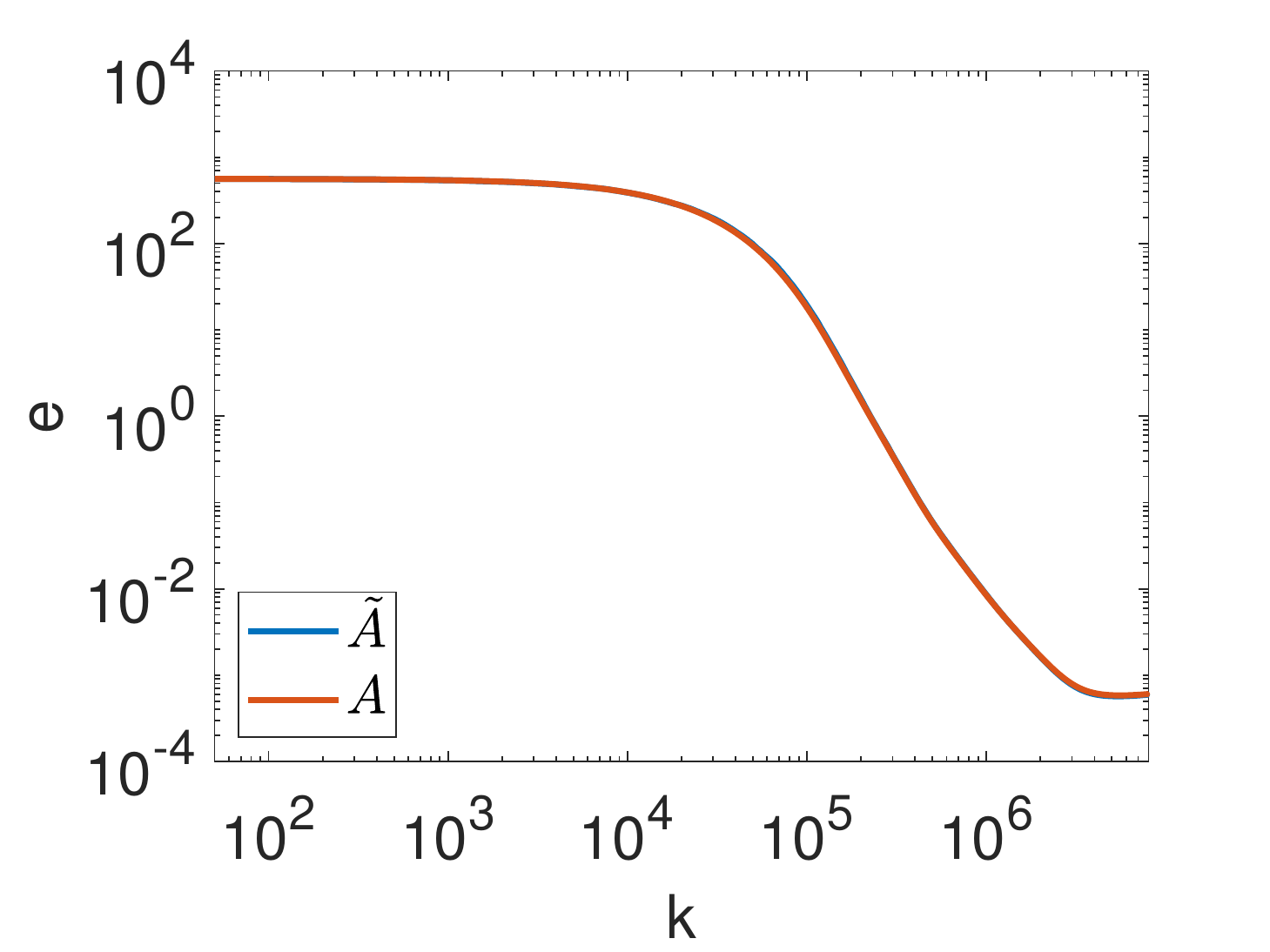}&
\includegraphics[width=0.31\textwidth,trim={1.5cm 0 0.5cm 0.5cm}]{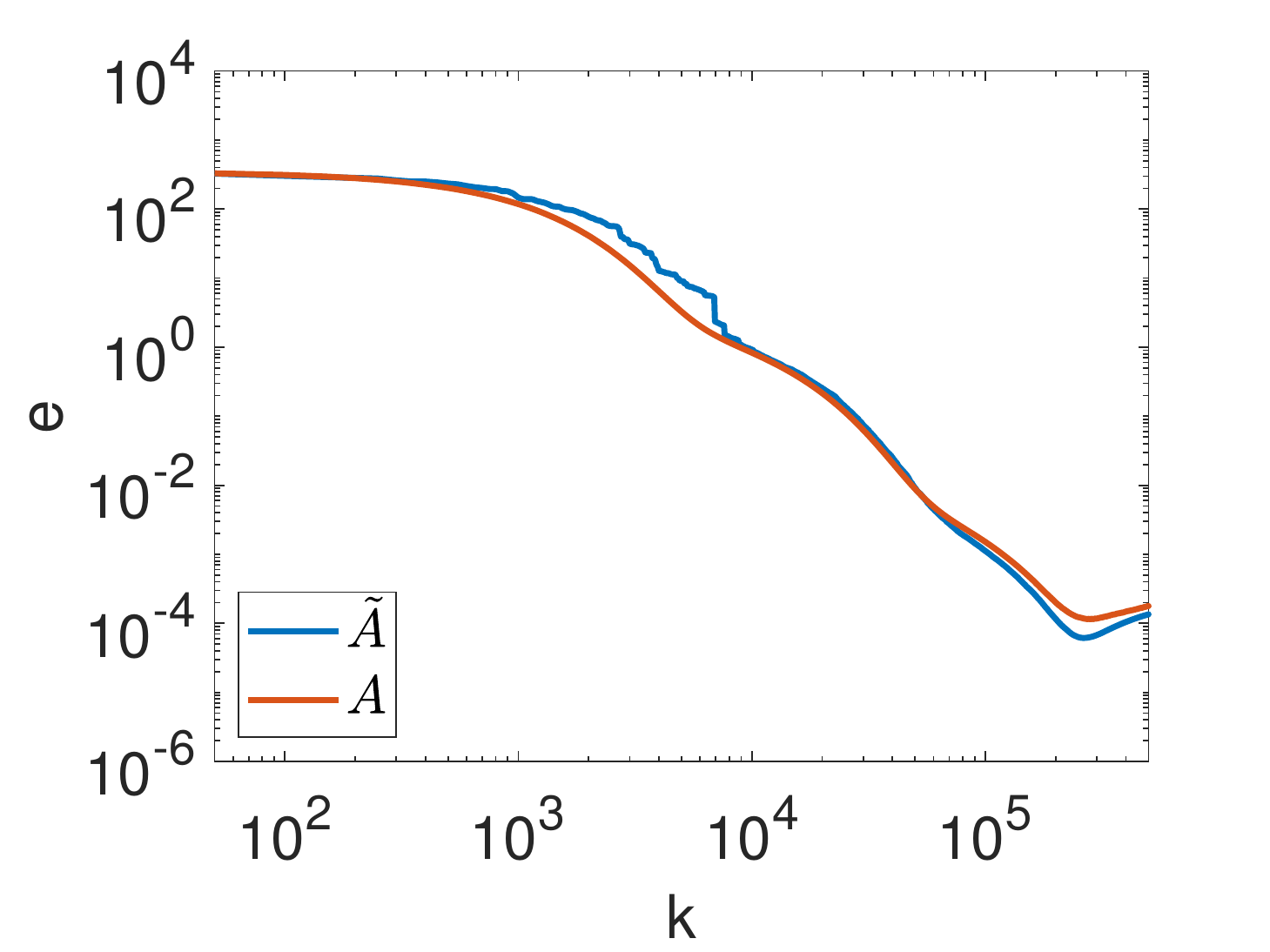}\\
\includegraphics[width=0.31\textwidth,trim={1.5cm 0 0.5cm 0.5cm}]{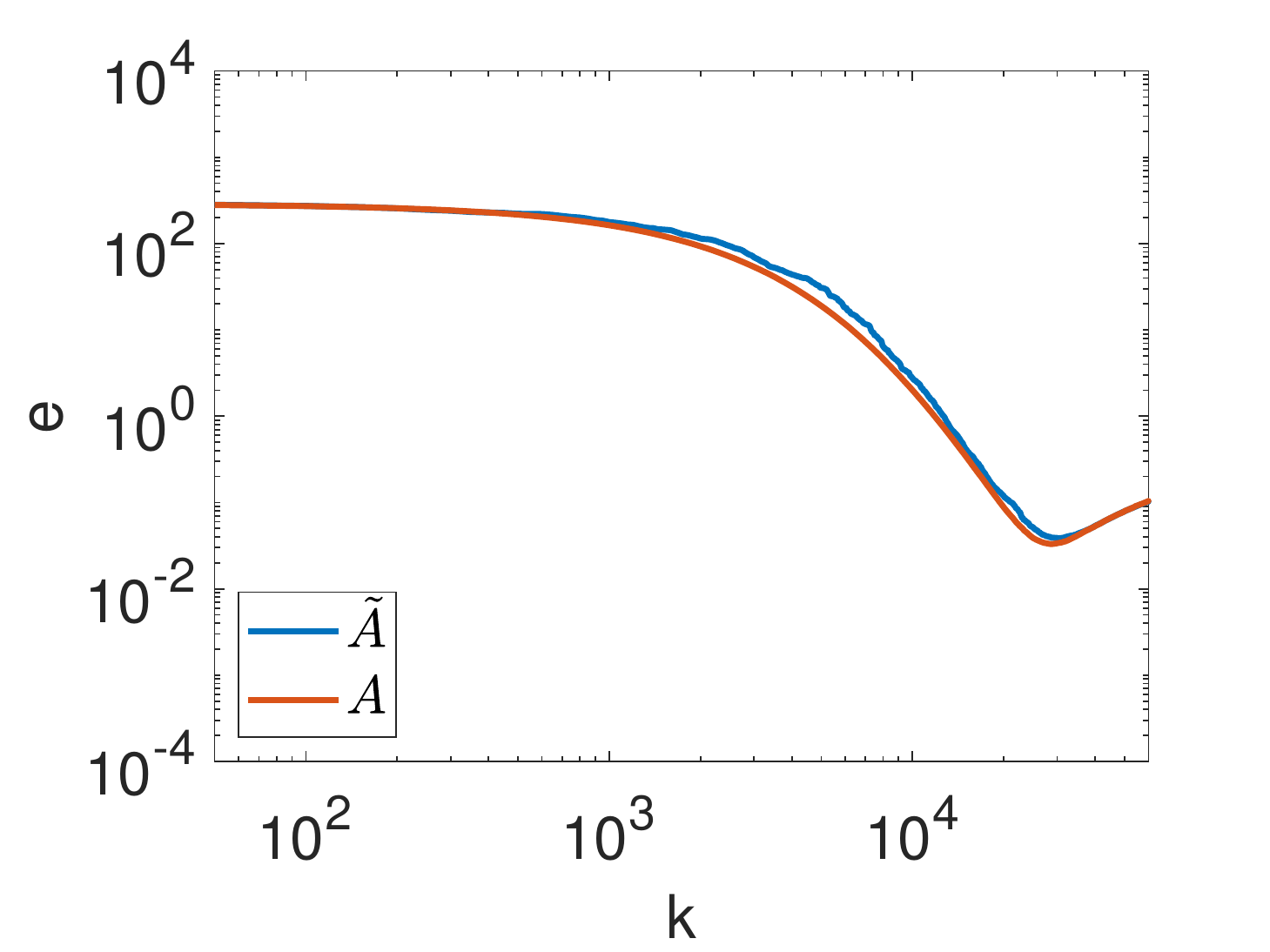}&
\includegraphics[width=0.31\textwidth,trim={1.5cm 0 0.5cm 0.5cm}]{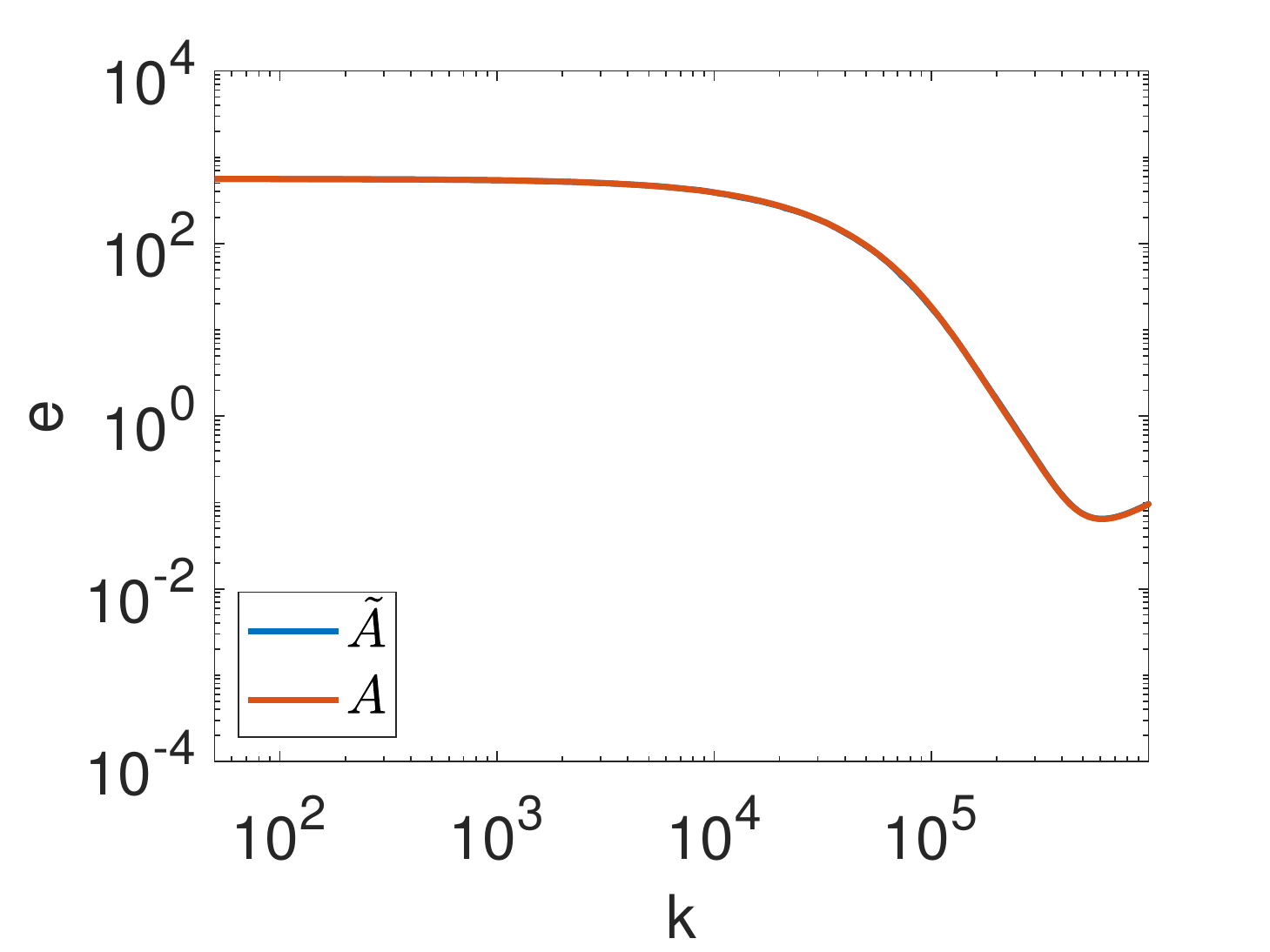}&
\includegraphics[width=0.31\textwidth,trim={1.5cm 0 0.5cm 0.5cm}]{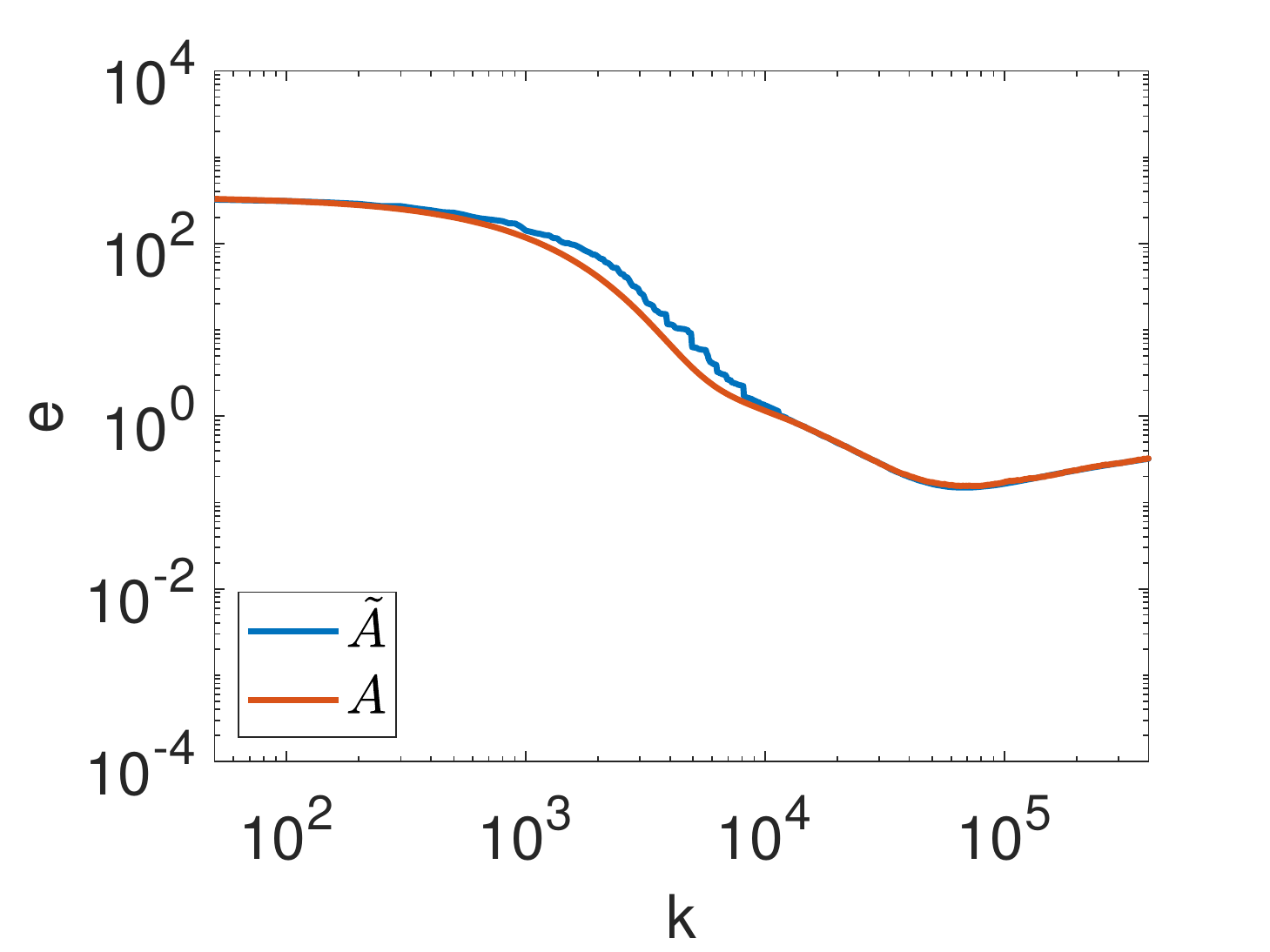}\\
\texttt{s-phillips}& \texttt{s-gravity} & \texttt{s-shaw}
\end{tabular}
\caption{The convergence of the error $e$ versus iteration number for the examples with
$\nu=1$, computed using $A$ and $\tilde A$. The rows from top to bottom rows are for $\epsilon=0$,
$\epsilon=$1e-3 and $\epsilon$=5e-2, respectively.\label{fig:err_UA}}
\end{figure}

\section{Concluding remarks}
In this work, we have presented a refined convergence rate analysis of stochastic gradient descent with
a polynomially decaying stepsize schedule for
linear inverse problems, using a finer error decomposition. The analysis indicates that the saturation
phenomenon exhibited by existing analysis actually does not occur, provided that the initial stepsize $c_0$
is sufficiently small. The analysis is also confirmed by several numerical experiments, which show that
with a small $c_0$, the accuracy of SGD is indeed comparable with the order-optimal Landweber method.

The numerical experiments show that Assumption \ref{ass:stepsize}(iii) is actually not needed
for the optimality, so long as the initial stepsize $c_0$ is sufficiently small, although the analysis
requires the condition. One outstanding issue is to
close the gap between the mathematical theory and practical performance.
The study naturally leads to the question whether there is a ``large'' stepsize schedule that can
achieve optimal convergence rates. The numerical experiments indicate that within polynomially decaying
stepsize schedules, a small value of $c_0$ seems necessary for order optimality, but it leaves open nonpolynomial ones,
e.g., stagewise SGD \cite{YuanYanJinYang:2019,GeKakade:2019}. Intuitively, the small initial stepsize can be viewed as
a form of implicit variance reduction, and thus it is also of interest to analyze existing explicit
variance reduction techniques, e.g., SVRG \cite{JohnsonZhang:2013} and SAG \cite{LeRouxSchmidtBach:2012}.
{ The current work discusses only deterministic noise. Naturally it is also of interest to extend
the analysis to the case of random noise; See, e.g., the work \cite{BissantzHohageMunkRuymgaart:2007,HarrachJahn:2020}
for relevant results for statistical inverse problems in a Hilbert space setting.}

\section*{Acknowledgments}
The authors would like to thank the two anonymous referees for their many constructive comments,
which have greatly helped improve the quality of the paper.

\bibliographystyle{abbrv}
\bibliography{sgd}
\end{document}